\documentclass[10pt]{amsart}
\usepackage[T1]{fontenc}
\usepackage[utf8]{inputenc}
\usepackage[english]{babel}
\usepackage{csquotes} 
\usepackage[a4paper,margin=3cm]{geometry}
\allowdisplaybreaks
\usepackage{graphicx}
\usepackage[backend=biber,style=alphabetic,bibencoding=utf8,language=auto,autolang=other,giveninits=true,doi=false,isbn=false,url=false,maxnames=10,sorting=nty,sortcites=true]{biblatex}
\usepackage[width=0.96\textwidth]{caption} 
\usepackage[hyperfootnotes=false]{hyperref} 
\hypersetup{
	colorlinks = true,
	linkcolor = {blue},
	urlcolor = {red},
	citecolor = {blue}
}
\usepackage[nameinlink,capitalise,sort]{cleveref}
\crefname{equation}{}{} 
\crefname{enumi}{}{} 

\crefname{figure}{Figure}{Figures}

\usepackage{amsmath}
\usepackage{amsthm}
\usepackage{amssymb}
\usepackage{esint} 
\usepackage{mathrsfs} 
\usepackage{faktor} 
\usepackage{dsfont} 
\usepackage[dvipsnames]{xcolor}
\usepackage{array}
\usepackage{hhline}
\usepackage[normalem]{ulem}
\usepackage{enumitem} 
\usepackage{comment} 
\usepackage[toc,page]{appendix} 
\usepackage{imakeidx} 
\usepackage{xparse} 
\usepackage{mathtools}
\usepackage{fancyhdr} 
\usepackage{ifthen} 
\usepackage{forloop} 
\usepackage{xstring}
\usepackage{tikz}
\usetikzlibrary{positioning, arrows.meta}
\usetikzlibrary{babel} 
\usetikzlibrary{cd} 
\usepackage{tikz-cd} 
\usepackage{emptypage} 
\usepackage{listings} 
\usepackage{mathabx} 
\usepackage{subcaption}
\usepackage{multirow}
\usepackage{booktabs}

\theoremstyle{plain}
\newtheorem{lemma}{Lemma}[section]

\newtheorem{theorem}[lemma]{Theorem}

\newtheorem{remark}[lemma]{Remark}
\newtheorem{example}{Experiment}[section]

\theoremstyle{definition}

\newtheorem{examplequad}{Example}[section]

\numberwithin{equation}{section}

\newcommand{\lip}{\mathrm{Lip}}
\newcommand{\R}{\mathbb{R}}
\newcommand{\N}{\mathbb{N}}

\newcommand{\eps}{{\varepsilon}}

\DeclareMathOperator*{\essinf}{ess\,\inf}

\DeclareMathOperator*{\esssup}{ess\,\sup}
\newcommand{\sgn}[1]{\mathrm{sign}\left(#1\right)}

\newcommand{\pt}{\partial_t}

\newcommand{\px}{\partial_x }
\newcommand{\pxx}{\partial_{xx}^2}

\newcommand{\norm}[1]{\left\lVert#1\right\rVert}

\newcommand{\dt}{{\, \mathrm{d}t}}
\newcommand{\dx}{{\, \mathrm{d}x}}

\newcommand{\dz}{{\, \mathrm{d}z}}
\renewcommand{\d}{\mathrm{d}}

\newcommand{\dd}{\,\mathrm{d}}

\begin{document}


\title[Nonlocal AC Schemes]{Asymptotically compatible entropy-consistent discretization for a class of nonlocal conservation laws}

\author[N.~De Nitti]{Nicola De Nitti}
\address[N.~De Nitti]{Università di Pisa, Dipartimento di Matematica, Largo Bruno Pontecorvo 5, 56127 Pisa, Italy.}
\email[]{nicola.denitti@unipi.it}

\author[K.~Huang]{Kuang Huang}
\address[K.~Huang]{The Chinese University of Hong Kong, Department of Mathematics, Shatin, New Territories, Hong Kong SAR, P.\,R.~China.}
\email[]{kuanghuang@cuhk.edu.hk}

\keywords{Nonlocal conservation laws, nonlocal flux, nonlocal LWR model, asymptotically compatible numerical discretization, finite volume schemes, nonlocal-to-local limit, traffic flow models.}

\subjclass[2020]{%
35L65, 
65M08, 
65M12, 
65M15, 
35B40, 
76A30
.}

\begin{abstract}
We consider a class of nonlocal conservation laws modeling traffic flows, given by
\( \partial_t \rho_\eps + \partial_x(V(\rho_\eps \ast \gamma_\eps) \rho_\eps) = 0\) with a suitable convex kernel $\gamma_\eps$, and its Godunov-type numerical discretization. We prove that, as the nonlocal parameter $\eps$ and mesh size $h$ tend to zero simultaneously, the discrete approximation $W_{\eps,h}$ of $W_\eps \coloneqq \rho_\eps \ast \gamma_\eps$ converges to the entropy solution of the (local) scalar conservation law \( \partial_t \rho + \partial_x(V(\rho) \rho) = 0\), with an explicit convergence rate estimate of order $\eps+h+\sqrt{\eps\, t}+\sqrt{h\,t}$. In particular, with an exponential kernel, we establish the same convergence result for the discrete approximation $\rho_{\eps,h}$ of $\rho_\eps$, along with an $\mathrm{L}^1$-contraction property for $W_\eps$. The key ingredients in proving these results are uniform $\mathrm{L}^\infty$- and TV-estimates that ensure compactness of approximate solutions, and discrete entropy inequalities that ensure the entropy admissibility of the limit solution.
\end{abstract}

\maketitle

\section{Introduction}
\label{sec:intro}

\subsection{Nonlocal conservation laws and the singular limit problem}
\label{ssec:model}

\textit{Nonlocal conservation laws} have proven effective for diverse applications, including traffic flow, supply chains, crowd dynamics, opinion formation,  spectrum of large random matrices, chemical engineering processes, sedimentation, slow erosion of granular matter, materials with fading memory effects, and conveyor belt dynamics (see, e.\,g., \cite{KEIMER2023} for a recent survey). In particular, in this paper, we focus on a nonlocal version of a macroscopic traffic flow model (as introduced by Lighthill--Whitham--Richards, \cite{MR72606,MR75522}): the traffic density $\rho_\eps: \R_+ \times \R \to \R$ satisfies the Cauchy problem
\begin{align}\label{eq:cle}
\begin{cases}
\partial_t \rho_\eps(t,x)  + \partial_x\Big(V\big(W_{\eps}[\rho_\eps](t,x))\,\rho_\eps(t,x)\Big)   	= 0,	& (t,x)  \in \R_+ \times \R, \\
\rho_\eps(0,x) = \rho_0(x),	&  x\in\R,
\end{cases}
\end{align}
where the \textit{initial datum}\footnote{~With slight abuse of notation, we use $\rho_0$ to denote the initial datum, where the subscript does not correspond to $\eps=0$ in $\rho_\eps$.} 
\begin{align}\label{ass:ic}
    \rho_0 \in \mathrm{L}^\infty(\R), \qquad  0 \leq \rho_0 \leq 1, \qquad \mathrm{TV}(\rho_0) < + \infty,
\end{align}
represents the initial traffic density, where $\rho=0$ indicates empty-road traffic and $\rho=1$ indicates bump-to-bump traffic; 
the \textit{velocity function} 
\begin{align}\label{ass:V}
V \in \lip([0,1]) \quad  \text{ and } \quad V' \leq 0 \text{ in $[0,1]$},
\end{align}
is decreasing (i.\,e., the higher the density of cars on the road,  the lower their speed); 
\begin{align}\label{eq:W} 
    W_{\eps}[\rho_\eps](t,x)\coloneqq\frac{1}{\eps}\int_{x}^{\infty} \gamma \left(\frac{x-y}{\eps}\right)\rho_\eps(t,y)\dd y , \qquad (t,x)\in \R_+ \times \R,
\end{align}
is the \emph{nonlocal impact} that decides the car speed $v\coloneqq V(W_{\eps}[\rho_\eps])$; and the \emph{nonlocal kernel} 
\begin{align}\label{ass:general-k}
\begin{aligned}
& \gamma \in \mathrm{BV}(\R) \cap \mathrm{L}^\infty(\R),  \quad \operatorname{supp}\gamma \subset ]-\infty,0], \quad  \gamma \ge 0, \\ & \gamma \text{ non-decreasing and convex in $]-\infty,0]$}, \quad \int_{-\infty}^0 \gamma(z) \, \d z = 1,
\end{aligned}
\end{align} 
rescaled as $\gamma_\eps \coloneqq \tfrac{1}{\varepsilon} \gamma\left(\tfrac{\cdot}{\eps}\right)$ with the  \emph{nonlocal horizon parameter} $\eps >0$, which serves as a characteristic length scale for the nonlocal effect encoded by $\gamma_\eps$. For traffic flow modeling, it is reasonable to assume that \(\gamma\) is anisotropic and, in particular, supported
in $]-\infty,0]$ and non-decreasing. This means that the drivers adjust their speed based only on the ``downstream'' traffic density (i.\,e., only looking forward) and give it more consideration the closer it is to their position.

For the existence, uniqueness, and stability of \textit{weak} (distributional) solutions of  the nonlocal conservation law \cref{eq:cle} with the nonlocal impact defined by \cref{eq:W}, we refer to \cite{zbMATH06756308,bressan2020traffic,zbMATH07615111} and references therein.  Notably, an important feature of \cref{eq:cle} is that those well-posedness results do not require an entropy condition, owing to the nonlocal effect.

The problem of the convergence of $\{\rho_{\varepsilon}\}_{\varepsilon >0}$, as $\varepsilon \searrow 0$ (i.\,e., as the kernel $\gamma_\eps$ converges to a Dirac delta distribution\footnote{~We say that the family $\{\gamma_\varepsilon\}_{\varepsilon >0}$ converges to the Dirac delta $\delta_0$ (or is an \textit{approximation to the identity} (of the convolution product); see \cite[\S\,II.6.31]{MR1157815}) if 
$
\lim_{\varepsilon \to 0} \int_{\R} \varphi(z) \gamma_\varepsilon(z) \, \mathrm d z = \varphi(0)
$
for all test functions $\varphi \in \mathrm{C}(\mathbb R)$.
}) 
to the (unique\footnote{~We refer to \cite{MR1304494,MR3443431} for the well-posedness of entropy solutions of \cref{eq:cl}. 
}) entropy solution $\rho$ of the conservation law 
\begin{align}\label{eq:cl}
    \begin{cases}
        \pt \rho(t,x) + \px\Big(V(\rho(t,x))\,\rho(t,x)\Big) = 0, & (t,x) \in \R_+ \times \R, \\
        \rho(0,x) = \rho_0(x), & x \in\R,
    \end{cases}
\end{align}
has drawn much attention in the last few years. The aim is to bridge the gap between nonlocal and local modeling of traffic flow and other phenomena described by conservation laws.

First, in \cite{MR3342191}, this convergence was observed numerically. However, in \cite{MR3961295}, several counterexamples showed that it does not hold in general for physically unreasonable kernels (e.\,g., in the context of traffic modeling, kernels looking backward and forward, or only backward). On the other hand, positive results on the nonlocal-to-local convergence were obtained in more particular situations: in \cite{MR1704419} for even convolution kernels, provided that the limit entropy solution is smooth; in \cite{MR3944408}, for a large class of nonlocal conservation laws with monotone initial data, exploiting the fact that monotonicity is preserved throughout the evolution; and, in \cite{MR4300935}, under the assumptions that the initial datum has bounded total variation, is bounded away from zero, and satisfies a one-sided Lipschitz condition, and the kernel grows at most exponentially (that is, there exists $D>0$ such that \(\gamma(z) \le D \, \gamma'(z)\), for a.\,e.~$z \in ]-\infty, 0[$). In \cite{MR4110434,MR4283539}, Bressan and Shen proved a convergence result for the exponential kernel $\gamma \coloneqq \mathds{1}_{]-\infty,0]}(\cdot)\exp(\cdot)$, provided that the initial datum is bounded away from zero and has bounded total variation, by reformulating the nonlocal conservation law as a hyperbolic system with a relaxation term. If the initial datum is not bounded away from zero, then, as observed in \cite{MR4300935}, establishing compactness properties of $\{\rho_{\varepsilon}\}_{\varepsilon >0}$ is difficult because the total variation of $\rho_\varepsilon$, denoted  $\mathrm{TV}(\rho_\varepsilon)$, may blow up. 

To overcome these restrictions on initial data, it is convenient to work with the family $\{W_{\varepsilon}\}_{\varepsilon >0}$ instead (where we use the simplified notation $W_\eps \coloneqq W_\eps[\rho_\eps]$), which has better stability and convergence properties.
It was demonstrated in \cite[Theorem 1.1]{MR4553943} that $\mathrm{TV}(W_\varepsilon) \le \mathrm{TV}(\rho_0)$ holds, under the assumptions \cref{ass:ic}, \cref{ass:V}, and \cref{ass:general-k}.\footnote{~On the other hand, \cite[Theorem 1.4]{MR4553943} shows that, without the convexity assumption in \cref{ass:general-k} (which is not entirely standard in traffic flow modeling), $\mathrm{TV}(W_\varepsilon)$ may increase.\label{ft:convexity}} This estimate yields strong convergence in \( \mathrm{L}^1_{\mathrm{loc}} \) of the family \( \{W_\varepsilon\}_{\varepsilon > 0} \) to a limit function, which, in turn, is shown to be the  entropy solution $\rho$ of \cref{eq:cl} (see, e.\,g., \cite[Theorem 1.2]{MR4553943}). 

In summary, as established in \cite[Theorem 1.3]{MR4553943}, $\{W_\eps\}_{\eps >0}$ converges strongly in $\mathrm{L}^1_{\mathrm{loc}}(\R_+\times\mathbb{R})$ to the entropy solution $\rho$ of \cref{eq:cl} as $\varepsilon \searrow 0$; moreover, if the first moment of $\gamma$ is finite, i.\,e., 
\begin{align}\label{ass:momentum}
    \gamma(z) \, z \in \mathrm{L}^1(\R),
\end{align}
then the following convergence rate estimate holds:
\begin{equation}\label{eq:rate-cont}
\|W_\eps(t, \cdot)- \rho(t,\cdot)\|_{\mathrm{L}^1} \leq K \left( \eps + \sqrt{\eps\, t} \right) \text{TV}(\rho_0), \quad \text{for every } \eps > 0 \text{ and a.\,e.~$t > 0$},
\end{equation}
where the constant $K>0$ depends only on $\gamma$ and $V$. 
The main tool in proving these results is the fact that (by \cite[Eq.~(3.1)]{MR4553943}) $W_\eps$ solves  \begin{align}\label{eq:Wevo-2}
\begin{cases}
\displaystyle  \partial_t W_{\eps}(t,x)+V\left(W_{\eps}(t,x)\right) \partial_x W_{\eps}(t,x) & {} \\ 
\displaystyle \quad=\frac{1}{\eps^2} \int_x^{\infty} \gamma^{\prime}\left(\frac{x-y}{\eps}\right)\big( V\left(W_{\eps}(t, x)\right)-V\left(W_{\eps}(t, y)\right) \big) \rho_{\eps}(t, y) \, \mathrm d y, & (t,x) \in \R_+ \times \R, \\ 
  \displaystyle  W_{\eps}(0,x)=\frac{1}{\eps}\int_{x}^{\infty}\gamma\left(\frac{x-y}{\eps}\right)\rho_{0}(y)\dd y, & x\in\R.
   \end{cases}
\end{align}

The TV-estimate $\mathrm{TV}(W_\varepsilon) \le \mathrm{TV}(\rho_0)$ and the convergence result for $W_\eps$ were established earlier in \cite[Theorems 3.2 and 4.2]{MR4651679} in the particular case where $\gamma \coloneqq \mathds{1}_{]-\infty,0]}(\cdot)\exp(\cdot)$, leading to the identity  
\begin{align}\label{eq:WWxq}
    \partial_{x}W_{\eps}[\rho_{\eps}](t,x)&=\partial_{x}\, \left( \frac{1}{\eps}\int_{x}^{\infty}\exp\left(\frac{x-y}{\eps}\right)\rho_{\eps}(t,y)\dd y \right) =\frac{1}{\eps}W_{\eps}[\rho_{\eps}](t,x)-\frac{1}{\eps}\rho_{\eps}(t,x).
\end{align}
The analysis is based on the following evolution equation for $W_\eps$ (see \cite[Lemma 3.1]{MR4651679}):
\begin{align}\label{eq:Wevo}
\begin{cases}
\partial_t W_\eps(t,x)+V(W_\eps(t,x)) \partial_x W_\eps(t,x) & {}\\
   \displaystyle \quad  =-\frac{1}{\eps}\int_{x}^{\infty}\exp\left(\frac{x-y}{\eps}\right) V'(W_\eps(t,y))\partial_y W_\eps(t,y)W_\eps(t,y)\dd y,  & (t,x) \in \R_+ \times \R,\\
  \displaystyle  W_{\eps}(0,x)=\frac{1}{\eps}\int_{x}^{\infty}\exp\left(\frac{x-y}{\eps}\right)\rho_{0}(y)\dd y, & x\in\R,
   \end{cases}
\end{align}
which is analogous to \cref{eq:Wevo-2}, but written purely in $W_\varepsilon$ owing to \cref{eq:WWxq}.
In this case, the combination of \cref{eq:WWxq} and the \( \mathrm{TV} \)-estimate on $W_{\varepsilon}$ further allows us to deduce the convergence of \( \{\rho_\varepsilon\}_{\varepsilon > 0} \) to the same limit as \( \{W_\varepsilon\}_{\varepsilon > 0} \) (see \cite[Corollary 4.1]{MR4651679}) as $\eps\searrow0$. 

Initial data with unbounded variation can also be addressed (in specific cases) using an Ole\u{\i}nik-type regularization effect, as demonstrated in \cite{MR4855163,MR4656976}. 
Moreover, results on the singular limit problem for certain classes of nonlocal hyperbolic systems are available in \cite{MR4731171,CocDeN24,marconi2025}, while different types of nonlocal approximations have been studied in \cite{MR4857742,MR4720572,MR4638873,ghoshal2025nonconservativenonlocalapproximationburgers}.  
Furthermore, the study of the singular limit in the presence of artificial viscosity (which is relevant because many numerical tests used to conjecture the convergence results employed a (dissipative) Lax--Friedrichs scheme) has also been investigated, subject to a suitable \textit{balance condition} between viscosity and nonlocal parameters, in \cite{MR3961295,MR4340167,zbMATH07315482,MR715133,MR4265719}.
We refer to \cite{colombo2023overviewlocallimitnonlocal,KEIMER2023} for further discussion, results, and references.

\subsection{Numerical discretizations for nonlocal conservation laws}
\label{ssec:numerical-history}

A substantial body of literature addresses numerical discretizations for conservation laws with nonlocal fluxes, including first-order finite-volume methods such as the Lax--Friedrichs scheme (cf.~\cite{MR3342191,MR3447130,aggarwal2015nonlocal,chiarello2018global}) and the Godunov scheme (cf.~\cite{MR3917881}), the second-order Nessyahu--Tadmor central scheme (cf.~\cite{kurganov2009non,betancourt2011nonlocal,goatin2016well}), and higher-order WENO and DG methods (cf.~\cite{MR3759879,MR3934112}). See \cite{JF2023,MR4742183} for discussions on a broader class of finite-volume methods and \cite{MR4700412,MR4839139} for convergence rate results therein.
We also refer to \cite[Chapter~3]{Pflug2018} or \cite[Section~5]{KEIMER2023} for a scheme based on the method of characteristics, and to \cite{chiarello2020micro,ridder2019traveling,goatin2017traffic,radici2021entropysolutionsnonlocalscalar,DiFrancesco2025,MR3356989,MR3906267} for particle discretizations.
These methods often extend to a broader class of nonlocal conservation laws, with \cref{eq:cle} as a special case, and have been employed to establish well-posedness of the underlying continuous problems. In the associated numerical analysis, stability and convergence are generally established for a fixed $\eps>0$.

However, a recurring issue in these results is that stability estimates such as TV-bounds and entropy inequalities typically deteriorate as $\varepsilon$ vanishes, with no uniform convergence rates available in terms of both the nonlocal horizon parameter $\varepsilon$ and the mesh size $h$. This challenge is particularly pronounced in contrast to local conservation laws like \cref{eq:cl}, where monotone finite-volume schemes enjoy maximum principles, total variation diminishing (TVD) properties, and discrete entropy inequalities. These properties ensure uniform stability estimates on numerical solutions, which are essential for convergence analysis (including convergence rates). However, the presence of the nonlocal effect in \cref{eq:cle} disrupts this monotonicity even for first-order finite-volume schemes, leading to stability estimates that lose uniformity in the singular limit as $\varepsilon \searrow 0$, rendering them incompatible with the uniform stability estimates and convergence results from the continuous problem \cref{eq:cle} (as discussed in \cref{ssec:model}, e.\,g., TVD properties for $W_\varepsilon$). 
This underscores a significant gap between numerical and analytical stability and convergence properties in the singular limit.
Developing a numerical discretization that remains stable and accurate in the singular limit for \cref{eq:cle} is therefore of both theoretical and practical importance.

\subsection{Novel contributions}
\label{ssec:novelty}

The gap in the literature highlighted in \cref{ssec:numerical-history} motivates us to study \emph{asymptotically compatible} (AC) numerical discretizations for problem \cref{eq:cle} and their convergence rates. The main goal of AC discretizations is to use a uniform mesh across all $\eps>0$ while accurately capturing the system's behavior in the limit $\eps \searrow 0$. This concept is illustrated in the following diagram:
\begin{center}
    \begin{tikzpicture}
    \matrix (m) [matrix of math nodes,row sep=4em,column sep=5em,minimum width=3em]
    {
        \mathcal P_{\eps,h} & \mathcal P_h \\
        \mathcal P_\eps & \mathcal P \\};
    \path[-stealth]
    (m-1-1) edge node [left] {\tiny $ h \searrow 0 $} (m-2-1)
    edge  node [above] {\tiny $\eps \searrow 0$} (m-1-2)
    edge[bend right=2] node [above right] {\tiny $\eps,h \searrow 0$} (m-2-2)
    (m-2-1.east) edge node [above] {\tiny $\eps \searrow 0$}
    node [above] {} (m-2-2)
    (m-1-2) edge node [right] {\tiny $ h \searrow 0 $} (m-2-2)
    ;
    \end{tikzpicture}
\end{center}
Here, $\mathcal{P}_{\eps}$ stands for the nonlocal problem \cref{eq:cle} with the nonlocal parameter $\eps >0$; $\mathcal{P}_{\eps,h}$ is a consistent numerical discretization of $\mathcal{P}_{\eps}$ with the mesh size $h>0$; $\mathcal P$ is the local problem \cref{eq:cl}; and $\mathcal P_h$ is a suitable discretization for $\mathcal P$. In this diagram:
\begin{itemize}
    \item the arrow from $\mathcal{P}_{\eps}$ to $\mathcal{P}$ denotes the singular limit of the nonlocal problem as $\eps \searrow 0$, with an established convergence rate of order $\eps + \sqrt{\eps\, t}$ (see \cref{eq:rate-cont});
    \item the arrow from $\mathcal{P}_{\eps,h}$ to $\mathcal{P}_\eps$ indicates the numerical convergence of the nonlocal discretization to the nonlocal problem for a fixed $\eps>0$ (cf.~the references cited in \cref{ssec:numerical-history});
    \item the arrow from $\mathcal{P}_{\eps,h}$ to $\mathcal{P}_h$ captures the relation between the nonlocal and local discretizations as $\eps \searrow 0$ for a fixed $h > 0$;
    \item the arrow from $\mathcal{P}_h$ to $\mathcal{P}$ then represents the well-established numerical convergence for the local problem, with a rate of order $h + \sqrt{h\, t}$ (see, e.\,g., \cite{MR3443431}).
\end{itemize}
A numerical discretization is said to be asymptotically compatible if it ensures the convergence from $\mathcal{P}_{\eps,h}$ to $\mathcal{P}$ as $\eps,\, h \searrow 0$ \emph{along any limit paths}, thereby making the diagram commutative.
Studies on AC numerical discretizations for \cref{eq:cle} are scarce in literature. To the authors' knowledge, the only result is in \cite{MR4742183}, considering restrictive initial data that satisfy a one-sided Lipschitz condition and are bounded away from zero, within the framework of \cite{MR4300935}.

In this work, we consider the following \textit{Godunov-type} numerical scheme\footnote{Following \cite{MR3917881}, we term the scheme a Godunov-type (or simply Godunov) scheme, interpreted as considering the Riemann problem for \cref{eq:cle} with $V=V(W_{\eps}[\rho_\eps])$ as a given velocity field; it can also be viewed as an upwinding scheme, as \cref{eq:cle} is linear in $\rho_\eps$ with the given velocity field.} for \cref{eq:cle} and \cref{eq:W}:
\begin{align} \label{eq:godunov}
\rho_j^{n+1}&=\rho_j^n+\lambda (\rho_{j-1}^n V(W_j^n) -\rho_j^n V(W_{j+1}^n)), && j \in \mathbb Z, \ n\geq0,\\
 \label{eq:num_nonlocal_W}
    W_j^n &= \sum_{k=0}^\infty \gamma_k^{\eps,h} \rho_{j+k}^n, && j \in \mathbb Z, \ n\geq0,
\end{align}
where $\lambda\coloneqq \tau/h$ is the \textit{CFL} (\textit{Courant--Friedrichs--Levy}, \cite{MR1512478,MR213764}) \textit{ratio}, and $\{\gamma_k^{\eps,h}\}_{k\ge0}$ is a sequence of numerical quadrature weights such that 
$$
\gamma_{\eps,h} \coloneqq \sum_{k=0}^\infty \gamma_k^{\eps,h} \, \mathds{1}_{]- \frac{(k+1)\eps}h, - \frac{k\eps}h]}
$$
is a piecewise constant approximation of the nonlocal kernel $\gamma_\eps =\tfrac{1}{\varepsilon} \gamma\left(\tfrac{\cdot}{\eps}\right)$ on the spatial mesh grids with mesh size $h$. 
We omit the superscripts $\eps,\,h$ in $\rho_j^n$ and $W_j^n$ but keep in mind that they depend on both $\eps$ and $h$.
The initial condition is discretized as
\begin{align}\label{eq:ini_cond_discrete}
    \rho_j^0 = \frac1h\int_{(j-\frac12)h}^{(j+\frac12)h} \rho_0(x)\dx, \quad j\in\mathbb{Z}.
\end{align}
Inspired by singular limit results for \cref{eq:cle}, we focus on the discretized nonlocal impact $\{W_j^n\}_{j \in \mathbb Z}^{n \ge 0}$.
From \crefrange{eq:godunov}{eq:num_nonlocal_W}, we derive the following time-step update for $\{W_j^n\}_{j \in \mathbb Z}^{n \ge 0}$:
\begin{align}\label{eq:W-2}
    W_j^{n+1} = W_j^n + \lambda \sum_{k=0}^\infty \gamma_k^{\eps,h} \left( \rho_{j+k-1}^n V(W_{j+k}^n) - \rho_{j+k}^n V(W_{j+k+1}^n) \right), \qquad j \in \mathbb Z, \ n\geq0.
\end{align}

The main contributions of this work are threefold. First, we establish the total variation diminishing (TVD) property for the discretized nonlocal impact $\{W_j^n\}_{j \in \mathbb Z}^{n \ge 0}$ through a refined analysis that exploits the nonlocal kernel's convexity, overcoming limitations of standard monotonicity arguments or nonlocal versions of Harten's lemma (e.\,g., those used in \cite{du2017nonlocal,fjordholm2021second}) for \cref{eq:cle}. This TVD property ensures uniform TV-bounds across $\mathcal{P}_{\eps,h}$, $\mathcal{P}_\eps$, $\mathcal{P}_h$, and $\mathcal{P}$. Second, we introduce a novel entropy condition for \cref{eq:cle}:
\begin{align*}
 \pt \eta(W_\eps) + \partial_x \psi(W_\eps) &\le \px \big(\eta'(W_\eps)   \big(V(W_\eps)W_\eps - (V(W_\eps) \rho_\eps) \ast \gamma_\eps\big)\big) \\ & \qquad +
     \partial_x \big( \big( \rho_\eps H_\eta(W_\eps(t,x)\mid W_\eps(t,\cdot)) \big) \ast\gamma_\eps \big) ,
\end{align*}
where $\psi'(\xi)\coloneqq\eta'(\xi)\,(\xi\,V(\xi))'$, $H_\eta(a\mid b) \coloneqq  I_\eta(b) - I_\eta(a) - V(b)(\eta'(b)-\eta'(a))$, and $I_\eta'(\xi) \coloneqq  \eta''(\xi)V(\xi)$.
This condition is compatible with the entropy condition for the scalar conservation law \cref{eq:cl} in the limit $\eps\searrow0$, and it allows for the following discrete Kru\v{z}kov-type entropy inequality (see \cref{lm:discrete-entropy-2}):
\begin{align*}
    &\frac{|W_j^{n+1}-c| -|W_j^n-c|}{\tau} + \frac{\Psi_{j+1/2}^n - \Psi_{j-1/2}^n}{h} \leq 0, \qquad \mathrm{for\ all}\  c\in\mathbb{R}, \\
    &\mathrm{with} \quad \Psi_{j-1/2}^n \coloneqq |W_{j-1}^n-c|V(c) - \sum_{k=0}^\infty \gamma_k^{\eps,h} \rho_{j+k-1}^n \left|V(W_{j+k}^n)-V(c)\right|,
\end{align*}
using Kru\v{z}kov's entropy function $\eta(\xi)\coloneqq|\xi-c|$. This form of entropy condition, novel in the literature (see \cref{rem:entr_lit}), ensures entropy admissibility for the limit of $\mathcal{P}_{\eps,h}$ and, combined with uniform TV-bounds, establishes convergence from $\mathcal{P}_{\eps,h}$ to $\mathcal{P}$ (i.\,e., the scheme \crefrange{eq:godunov}{eq:ini_cond_discrete} is asymptotically compatible).
Third, applying Kuznetsov’s argument with this entropy condition, we derive an asymptotically compatible convergence rate of order $\eps + h + \sqrt{\eps\,t} + \sqrt{h\,t}$ for $\mathcal{P}_{\eps,h}$ to $\mathcal{P}$. This rate is consistent with the known convergence rates for $\mathcal{P}_\eps$ to $\mathcal{P}$ and $\mathcal{P}_h$ to $\mathcal{P}$.

\section{Main result and outline of the paper}
\label{ssec:main-outline}

We begin by specifying the conditions imposed on the numerical quadrature weights, which will be used in the formulation and analysis of the numerical scheme \crefrange{eq:godunov}{eq:ini_cond_discrete}. We assume that the family of quadrature weights  $\{\gamma_k^{\eps,h}\}_{k\ge0}$ satisfies the following conditions: 
\begin{align}\label{ass:pos-mono}
&\gamma_k^{\eps,h} \geq \gamma_{k+1}^{\eps,h}\geq 0, \qquad \text{for all } k\geq0;\\
&\sum_{k=0}^{\infty} \gamma_k^{\eps,h} = 1; \label{ass:normalize}\\ 
&\gamma_{k-1}^{\eps,h}+\gamma_{k+1}^{\eps,h}-2\gamma_k^{\eps,h}\geq 0, \qquad \text{for all } k\geq1; \label{ass:convexity} \\
\label{eq:quad_weight_localize}
&\lim_{R\to\infty}\sup_{\eps,h>0} \sum_{k=0}^\infty \mathds{1}_{\frac{kh}\eps\ge R} \gamma^{\eps,h}_k = 0;
\\
&\sum_{k=0}^{\infty} k \gamma_k^{\eps,h} \leq c_\gamma \frac{\eps}h, \qquad \text{with $c_\gamma>0$ depending only on $\gamma$.} \label{ass:gamma-mom}
\end{align}

Our main result is the convergence, in the strong topology of $\mathrm{L}^1_{\mathrm{loc}}(\R_+\times \R)$, of the piecewise constant reconstruction of $\{W_j^n\}_{j \in \mathbb Z}^{n \ge 0}$ to the unique entropy solution of \cref{eq:cl} as $\eps,h \searrow 0$ \emph{along any limiting paths}, with an $\mathrm{L}^1$-convergence rate estimate.

\begin{theorem} \label{th:main-2} 
Let us assume that \crefrange{ass:ic}{ass:V} hold, the quadrature weights satisfy \crefrange{ass:pos-mono}{eq:quad_weight_localize}, and the CFL condition 
\begin{align}\label{ass:CFL}
   \lambda \Big(\norm{V}_{\mathrm{L}^\infty}+2\|V'\|_{\mathrm{L}^\infty}\Big) \leq 1
\end{align}
holds, with the CFL ratio $\lambda\coloneqq \tau/h$ fixed. Let us consider the numerical solutions $\{\rho_j^n\}_{j\in\mathbb{Z}}^{n\geq0}$ and $\{W_j^n\}_{j\in\mathbb{Z}}^{n\geq0}$ constructed with the numerical scheme \crefrange{eq:godunov}{eq:ini_cond_discrete}, and let $W_{\eps,h}$ be the piecewise constant reconstruction of $\{W_j^n\}_{j\in\mathbb{Z}}^{n\geq0}$, i.\,e., 
\begin{align}\label{eq:W_piece_recon}
    W_{\eps,h} \coloneqq \sum_{n=0}^\infty \sum_{j\in\mathbb{Z}} W_j^n \,\cdot\, \mathds{1}_{[n\tau,(n+1)\tau[\times[(j-\frac12)h,(j+\frac12)h[}.
\end{align}
Then, as $\eps,h\searrow 0$, the approximate solution $W_{\eps,h}$ converges strongly in $\mathrm{L}^1_{\mathrm{loc}}$ to the unique entropy solution $\rho$ of \cref{eq:cl}. 
Moreover, assuming \cref{ass:gamma-mom}, the following error estimate holds: 
\begin{align}\label{eq:rate-thm}
\|W_{\eps,h}(t,\cdot)-\rho(t,\cdot)\|_{\mathrm{L}^1} \leq K\left( \eps+h+\sqrt{\eps\, t}+\sqrt{h\,t} \right)\, \mathrm{TV}(\rho_0), \quad \textnormal{for every } \eps,h > 0, \, t > 0,
\end{align}
where the constant $K>0$ only depends on $\lambda$, $\|V\|_{\mathrm{L}^\infty}$, $\|V'\|_{\mathrm{L}^\infty}$, and $c_\gamma$ (as specified in \cref{ass:gamma-mom}).
\end{theorem}

We prove \cref{th:main-2} in \cref{sec:proof-2}. This theorem extends the result of \cite{MR4742183} to initial data satisfying only \cref{ass:ic}, with general convex kernels and nonlinear velocity functions. The work \cite{MR4742183} focuses exclusively on the convergence of $\rho_{\varepsilon,h}$; its result and \cref{th:main-2} are complementary, with neither containing the other. In \cref{th:main-1} presented below in \cref{sec:exp}, we provide the analogue of \cref{th:main-2} for the exponential kernel $\gamma \coloneqq \mathds{1}_{]-\infty,0]}(\cdot)\exp(\cdot)$, where we can additionally establish convergence of $\rho_{\varepsilon,h}$ as $\varepsilon,h \searrow 0$.

\subsection{Discussions on the main result}

In this subsection, we provide remarks on the assumptions on quadrature weights and on how our main result relates to existing convergence rate results for local and nonlocal conservation laws.

Let us first discuss the assumptions on quadrature weights. The conditions \crefrange{ass:pos-mono}{ass:normalize} specify non-negativity, monotonicity, and normalization requirements (analogous to \cref{ass:general-k}), which are needed to establish a maximum principle for the numerical scheme \crefrange{eq:godunov}{eq:ini_cond_discrete}. The convexity condition \cref{ass:convexity} is needed to prove that the scheme is TVD, analogous to the corresponding condition in \cref{ass:general-k} in the continuous setting. The condition \cref{eq:quad_weight_localize} is motivated by a characterization of approximations to the Dirac delta distribution\footnote{~We recall that a family $\{K_n\}_{n >0} \subset \mathrm{L}^1(\mathbb R)$ converges to the Dirac delta $\delta_0$ if 
\[ K_n \ge 0, \qquad 
\|K_n\|_{\mathrm{L}^1} = 1, \qquad \int_{-\infty}^{-R} K_n(z) \, \mathrm d z + \int_{R}^{+\infty} K_n(z) \, \mathrm d z \to 0 \text{ as $n \to \infty$ for every fixed $R>0$.}
\]} and is used in the convergence proof.
Finally, the momentum condition in \cref{ass:gamma-mom}---which actually implies \cref{eq:quad_weight_localize} and is analogous to \cref{ass:momentum}---is required to establish the convergence rate. 

The choice of quadrature weights for a given kernel is non-unique. Here, we give two examples.
\begin{examplequad}\label{ex:exact}
The \emph{exact quadrature weights}
\begin{align}\label{eq:quad-2}
    \gamma_k^{\eps,h} = \int_{-(k+1)h}^{-kh} \frac1\eps \gamma\left(\frac{z}{\eps}\right) \dd z = \int_{-(k+1)h/\eps}^{-kh/\eps} \gamma(z) \, \mathrm d z,
\end{align}
satisfy \crefrange{ass:pos-mono}{ass:gamma-mom}, with $c_\gamma\coloneqq \int_{\mathbb{R}} |z|\gamma(z)\, \dz$ being the first order moment of $\gamma$.
\end{examplequad}

\begin{examplequad}\label{ex:Riemann} 
The \emph{normalized Riemann quadrature weights}
\begin{align}\label{eq:quad-3}
\gamma_k^{\eps,h} = \frac1{\sum_{k=0}^\infty \tilde{\gamma}_k^{\eps,h}}\tilde{\gamma}_k^{\eps,h} \qquad \text{with } \qquad \tilde{\gamma}_k^{\eps,h} = \frac{h}\eps\, \gamma\left(-\frac{kh}\eps\right),
\end{align}
satisfy \crefrange{ass:pos-mono}{ass:gamma-mom} too, where $\gamma(0)$ is taken as $\gamma(0-)$ when $k=0$, and $c_\gamma\coloneqq 2\int_{\mathbb{R}} |z|\gamma(z)\, \mathrm d z$.
\end{examplequad}

Next, we offer some remarks on the limiting cases for the convergence result in \cref{th:main-2}.

\begin{remark}[Limit for $\varepsilon \searrow 0$, when $h>0$ is fixed]\label{rk:hfixed}
Let us fix $h>0$ and let $\varepsilon\searrow0$ in the numerical scheme \crefrange{eq:godunov}{eq:ini_cond_discrete}. Under \cref{ass:normalize}, the limit of \cref{eq:num_nonlocal_W} gives $W_j^n=\rho_j^n$, thus \cref{eq:godunov} becomes
\begin{align}\label{eq:local_scheme}
    \rho_j^{n+1} = \rho_j^n + \lambda (\rho_{j-1}^n V(\rho_j^n) - \rho_j^n V(\rho_{j+1}^n)),
\end{align}
which together with \cref{eq:ini_cond_discrete} is a monotone scheme for the local conservation law \cref{eq:cl}.
Moreover, suppose that the kernel $\gamma$ is supported on a finite interval, without loss of generality, let us assume that $\operatorname{supp}\gamma\subset[-1,0]$. Then the scheme \crefrange{eq:godunov}{eq:ini_cond_discrete} reduces to \cref{eq:local_scheme} (with the initial condition \cref{eq:ini_cond_discrete}) when $\eps\le h$.  
\end{remark}

\begin{remark}[Limit for $h \searrow 0$, when $\varepsilon>0$ is fixed]\label{rk:epsfixed}
Assuming
    \begin{align}
    \frac{h}{\varepsilon}\gamma\left(-\frac{(k+1)h}\varepsilon\right)\le\gamma_k^{\eps,h}\le \frac{h}{\varepsilon}\gamma\left(-\frac{kh}\varepsilon\right) \label{ass:consist},
\end{align}
the numerical scheme \crefrange{eq:godunov}{eq:ini_cond_discrete} is consistent with the nonlocal conservation law \cref{eq:cle} for any fixed $\varepsilon >0$  (see \cite[Section 1]{MR4742183}, in particular the discussion surrounding Assumption 3, for details on the role of \cref{ass:consist}). 

We note that the assumption \cref{ass:consist}, together with \cref{ass:general-k}, implies \cref{ass:pos-mono} and \cref{eq:quad_weight_localize}; moreover, when combined with \cref{ass:momentum}, it implies \cref{ass:gamma-mom}.
\end{remark}

Finally, in the next remark, we discuss the optimality of the convergence rate in \cref{th:main-2}, inspired by the rich literature on convergence rates of monotone approximations of (local) scalar conservation laws.

\begin{remark}[Convergence rates and monotone approximations of local conservation laws]\label{rk:local-convergence-references}

For local conservation laws, the vanishing viscosity approximation (obtained by adding a parabolic regularization term $\eps \pxx \rho_\eps$) and monotone numerical schemes are known to converge to the entropy solution with at most first-order accuracy (see \cite{MR551288,MR413526}). 

Although viscous approximations and monotone schemes are formally first-order, they may lose half-order accuracy across shocks. Indeed, Kuznetsov's $O(\sqrt{\eps})$ or $O(\sqrt{h})$ convergence rates (see \cite{MR483509} and also \cite{MR811184,MR679435,MR1111445}) is indeed optimal for all monotone approximations applied to linear advection equations (see in \cite{MR1270625}) as well as for genuinely nonlinear fluxes in the case of BV data (see \cite{MR1480382}). 

For the special case of monotonic initial data, rates of $O(\eps|\log \eps|)$ or $O(h |\log h|)$ have been obtained and are optimal (see \cite{MR929127}). This also holds for piecewise smooth initial data, as proven in \cite{MR1397446,MR1451109,MR1646744}, where the rate is actually $O(\eps)$ or $O(h)$ for initial data with non-interacting shocks, provided that no shocks form at later times as well.

In \cref{ssec:num-convergence}, we present numerical investigations of the convergence rates. For Riemann shock initial data, the rate appears to reach $O(\varepsilon + h)$, whereas for Riemann rarefaction initial data it lies between $O(\varepsilon + h)$ and $O(\sqrt{\varepsilon} + \sqrt{h})$. A more detailed study of the optimality of the rate in \cref{eq:rate-thm} will be addressed in future work.

\end{remark}

\subsection{Outline of the paper}
\label{ssec:outline}

We present the outline of the paper and the structure of the main theorems and lemmas, as illustrated in \cref{fig:logic_diagram}.

In \crefrange{ssec:l-infty}{ssec:TVD}, under appropriate CFL conditions, we establish the maximum principle (\cref{lem:maxm_prcp}), the total variation diminishing (TVD) property (\cref{lm:TVD-x-2}), and a temporal total variation estimate (\cref{lm:TVD-t-general}) for the approximate solution $W_{\eps,h}$. From these stability properties, by Helly's compactness theorem (see, e.\,g., \cite[Theorem 2.3, p.~14]{MR1816648}), we deduce the convergence of $W_{\eps,h}$ (up to subsequences) to a limit point $\rho^*$ in \cref{lm:limit}. Then, in \cref{ssec:entropy-consistency-2}, we use a discrete entropy inequality (\cref{lm:discrete-entropy-2}) to show that $\rho^*$ coincides with the entropy solution $\rho$ of \cref{eq:cl} (\cref{lem:entropy_admissibility-2}).

In \cref{ssec:kuz}, we derive a convergence rate estimate (\cref{th:rate}) using a Kuznetsov-type lemma (recalled in \cref{lm:kut}) through careful estimation of a ``relative entropy''.
Finally, in \cref{ssec:conclusion-of-proof}, we assemble all the preceding components to complete the proof of \cref{th:main-2}.

In the particular case of an exponential kernel $\gamma \coloneqq \mathds{1}_{]-\infty,0]}(\cdot)\exp(\cdot)$, we also prove the convergence of $\rho_{\eps,h}$, a piecewise constant reconstruction of $\{\rho_j^n\}_{j \in \mathbb Z}^{n \ge 0}$. The statement of the result, \cref{th:main-1}, and its proof are given in \cref{sec:exp}: it follows from \cref{th:main-2} upon noticing a suitable $\mathrm{L}^1$-deviation estimate between $\rho_j^n$ and $W_j^n$ (see \cref{lm:limit-same}). We particularly highlight that, in this case, the TVD property can be shown more straightforwardly (see \cref{lm:TVD-x}) and that $\{\rho_j^n\}_{j \in \mathbb Z}^{n \ge 0}$ also satisfies a suitable discrete entropy inequality (see \cref{lm:entro-disc-2} in \cref{ssec:entropy-consistency}). Furthermore, in \cref{ssec:L1-contr}, we establish an $\mathrm{L}^1$-contraction property for the evolution equation \cref{eq:Wevo} satisfied by $W_\varepsilon$ (see \cref{th:L1.contr}) and discuss its consequences.

In \cref{sec:nums}, we present comprehensive numerical experiments to illustrate our main results and suggest further conjectures. In \cref{sec:conclusion}, we conclude the paper with a summary of our main findings and directions for future work.

\begin{figure}[htbp]
    \centering
    \scalebox{0.88}{
    \begin{tikzpicture}[
        node distance=1.5cm and 1.5cm,
        lemma/.style={rectangle, draw, rounded corners, minimum height=1cm, minimum width=3cm, align=center, font=\small},
        theorem/.style={rectangle, draw, double, minimum height=1cm, minimum width=3cm, align=center, font=\small, fill=gray!20},
        arrow/.style={-Stealth, thick},
        every node/.style={font=\small},
        invisible/.style={inner sep=0, minimum size=0, draw=none, fill=none}
    ]
    \node[lemma] (L31) {Maximum principle \\ \cref{lem:maxm_prcp}};
    
    \node[lemma, below left=1cm and -1.4cm of L31] (CK_TV) {TV-estimates \\ \cref{lm:TVD-x-2}, \cref{lm:TVD-t-general}};
    \node[lemma, below right=1cm and -1.4cm of L31] (CK_Entr) {Entropy inequality \\ \cref{lm:discrete-entropy-2}};
    \node[lemma, below left=3.2cm and -1.4cm of L31] (CK_Conv) {Convergence \\ \cref{lm:limit}, \cref{lem:entropy_admissibility-2}};
    \node[lemma, below right=3.2cm and -1.4cm of L31] (CK_Rate) {Convergence rate \\ \cref{th:rate}};
    \node[theorem, below=5.3cm of L31] (CK_Thm) {\cref{th:main-2}};

    \node[invisible, below=0.95cm of L31] (CK_Row1t) {};
    \node[invisible, below=2.05cm of L31] (CK_Row1b) {};
    \node[invisible, below=3.2cm of L31] (CK_Row2t) {};
    \node[invisible, below=4.3cm of L31] (CK_Row2b) {};

    \node[theorem, right=5cm of L31] (EK_L1) {$\mathrm{L}^1$-contraction \\ \cref{th:L1.contr}};
    \node[lemma, below left=1cm and -1.4cm of EK_L1] (EK_TV) {TV-estimates \\ \cref{lm:TVD-x}, \cref{lm:TVD-t-general}};
    \node[lemma, below right=1cm and -1.4cm of EK_L1] (EK_Entr) {Entropy inequality \\ \cref{lm:entro-disc-2}};
    \node[theorem, right=5cm of CK_Thm] (EK_Thm) {\cref{th:main-1}}; 
    \node[lemma, above=1.1cm of EK_Thm] (EK_Dev) {Deviation between $\rho_j^n$ and $W_j^n$ \\ \cref{lm:limit-same}};

    \draw[arrow] (L31) -- (CK_Row1t);
    \draw[arrow] (CK_Row1b) -- (CK_Row2t);
    \draw[arrow] (CK_Row2b) -- (CK_Thm);
    
    \draw[arrow] (EK_Dev) -- (EK_Thm);
    \draw[arrow] (CK_Thm) -- (EK_Thm);

    \node at (0, 1.2) {\textbf{Convex kernel}};
    \node at (8, 1.2) {\textbf{Exponential kernel}};
\end{tikzpicture}
}
    \caption{Logical structure of the paper: dependencies of the main theorems and lemmas for convex (\textsc{left}) and exponential (\textsc{right}) kernels.}
    \label{fig:logic_diagram}
\end{figure}
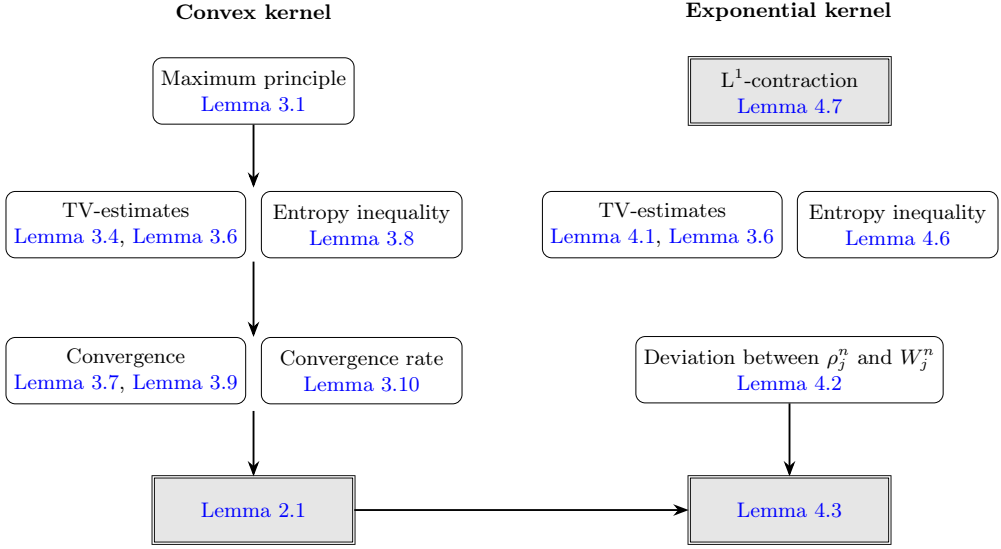

\section{Proof of the main theorem}
\label{sec:proof-2}

We will start by proving the strong pre-compactness of the family $\{W_{\eps,h}\}_{\eps,h>0}$ in $\mathrm{L}^1_{\mathrm{loc}}(\R)$. To this end, in \cref{ssec:l-infty}, we prove uniform $\mathrm{L}^\infty$-bounds through a maximum principle; in \cref{ssec:TVD}, we prove that $W_{\eps,h}$ is \emph{total variation diminishing} (TVD), implying uniform total variation estimates. These estimates imply the convergence of $W_{\eps,h}$ (up to subsequences) to a limit point $\rho^\ast$ strongly in $\mathrm{L}^1_{\mathrm{loc}}$ as $\eps,h\searrow0$. To show that this limit point $\rho^\ast$ is the entropy solution of the local scalar conservation law \cref{eq:cl}, we will introduce a discrete entropy inequality for $W_{\eps,h}$ in \cref{ssec:entropy-consistency-2}. Finally, in \cref{ssec:kuz}, we will adapt Kuznetsov's argument \cite{MR483509} to derive a convergence rate estimate in $\mathrm{L}^1$. Combining these ingredients, in \cref{ssec:conclusion-of-proof}, we will complete the proof of \cref{th:main-2}.

\subsection{Maximum principle and uniform \texorpdfstring{$\mathrm{L}^\infty$}{L-infinity}-bounds} 
\label{ssec:l-infty}
As a first step, we establish a maximum principle for the numerical scheme \crefrange{eq:godunov}{eq:ini_cond_discrete}. Similar results have been shown for a broader class of numerical schemes in \cite{MR3447130,JF2023,MR4742183}. For completeness, we state the result here and provide a proof.

\begin{lemma}[Maximum principle]\label{lem:maxm_prcp}
Let us assume that \crefrange{ass:ic}{ass:V} hold, the quadrature weights satisfy \crefrange{ass:pos-mono}{ass:normalize}, and the CFL condition
\begin{align}\label{ass:CFL-max}
\lambda \Big(\|V\|_{\mathrm{L}^{\infty}} + \|V'\|_{\mathrm{L}^{\infty}}\Big) \leq1
\end{align} holds. Let $\{\rho_j^n\}_{j\in\mathbb{Z}}^{n\geq0}$ and $\{W_j^n\}_{j\in\mathbb{Z}}^{n\geq0}$ be the numerical solutions constructed with the numerical scheme \crefrange{eq:godunov}{eq:ini_cond_discrete}. Then the following uniform bounds hold:
\begin{align*}
    \rho_{\mathrm{min}} \leq \rho_j^n \leq \rho_{\mathrm{max}} \qquad \textnormal{for all } j\in\mathbb{Z}, \ n\geq0,
\end{align*}
where $\rho_{\mathrm{min}}\coloneqq\inf_{x\in\mathbb{R}}\rho_0(x)\geq0$ and $\rho_{\mathrm{max}}\coloneqq\sup_{x\in\mathbb{R}}\rho_0(x)\leq1$.
\end{lemma}

\begin{proof}
We show $\rho_j^n \leq \rho_{\mathrm{max}}$ by induction. The base step for $n = 0$ follows from the definition of $\rho_{\mathrm{max}}$. Now, assuming that the result holds for $n$, we prove it for $n+1$. We have
\begin{align*}
\rho_j^{n+1} - \rho_{\mathrm{max}} &= (\rho_j^n - \rho_{\mathrm{max}}) \left(1 - \lambda V(W_{j+1}^n)\right) + \lambda \rho_{j-1}^n V(W_j^n) - \lambda \rho_{\mathrm{max}} V(W_{j+1}^n) \\
&\leq (\rho_j^n - \rho_{\mathrm{max}}) \left(1 - \lambda V(W_{j+1}^n)\right) + \lambda \rho_{\mathrm{max}}  \left( V(W_j^n) - V(W_{j+1}^n) \right).
\end{align*}
Using \cref{eq:num_nonlocal_W}, we deduce that 
\begin{align*}
    W_j^n - W_{j+1}^n &= \sum_{k=0}^\infty \gamma_k^{\eps,h} (\rho_{j+k}^n - \rho_{j+k+1}^n) 
    \\ &= \gamma_0^{\eps,h} \rho_j^n + \sum_{k=0}^\infty (\gamma_{k+1}^{\eps,h} - \gamma_k^{\eps,h}) \rho^n_{j+k+1} \\
    &\ge \gamma_0^{\eps,h} \rho_j^n + \rho_{\mathrm{max}} \sum_{k=0}^\infty (\gamma_{k+1}^{\eps,h}-\gamma_k^{\eps,h}) = \gamma_0^{\eps,h} (\rho_j^n - \rho_{\mathrm{max}}).
\end{align*}
We recall that $V' \leq 0$ and deduce that
\begin{align*}
    \rho_j^{n+1} - \rho_{\mathrm{max}} \leq (\rho_j^n - \rho_{\mathrm{max}}) \left(1 - \lambda V(W_{j+1}^n) - \lambda \rho_{\mathrm{max}}  \|V'\|_{\mathrm{L}^\infty} \gamma_0^{\eps,h} \right) \leq 0, 
\end{align*}
provided that $1 - \lambda V(W_{j+1}^n) - \lambda \rho_{\mathrm{max}}  \|V'\|_{\mathrm{L}^\infty} \gamma_0^{\eps,h} \geq 1 - \lambda \left(\|V\|_{\mathrm{L}^{\infty}} + \|V'\|_{\mathrm{L}^{\infty}}\right)\geq0$ due to the CFL condition \cref{ass:CFL-max}.
The lower bound estimate $\rho_j^n \geq \rho_{\mathrm{min}}$ can be proved in a similar manner.
\end{proof}

In particular, \cref{lem:maxm_prcp} implies that the numerical scheme preserves positivity. Consequently, it also conserves the $\mathrm{L}^1$-norm.\footnote{~In \cref{ass:ic}, we do not assume that $\|\rho_0\|_{\mathrm{L}^1}$ is finite; so, in principle, both sides of \cref{eq:L1_bound} could be infinity.}

\begin{lemma}[$\mathrm{L}^1$-conservation]\label{cor:L1_bound}
Under the conditions of \cref{lem:maxm_prcp}, we have
\begin{equation}
h\sum_{j\in\mathbb{Z}} |\rho_j^n| = h\sum_{j\in\mathbb{Z}} |\rho_j^0|\qquad \textnormal{for all } j\in\mathbb{Z}, \ n\geq0. \label{eq:L1_bound}
\end{equation}
\end{lemma}

\begin{proof}
Owing to \cref{lem:maxm_prcp}, we have $\rho_j^n \ge 0$ for all $j \in \mathbb Z$ and $n \ge 0$. We can then compute 
\begin{align*}
h \sum_j \rho_j^{n+1} = h \sum_j \left( \rho_j^n + \lambda \rho_{j-1}^n V(W_j^n) -\lambda \rho_j^n V(W_{j+1}^n)  \right) = h \sum_j \rho_j^n,
\end{align*}
which yields \cref{eq:L1_bound}. 
\end{proof}

\begin{remark}[$\mathrm{L}^\infty$-bound and $\mathrm{L}^1$-conservation for $W_j^n$]\label{rk:W}
As a direct consequence of \crefrange{lem:maxm_prcp}{cor:L1_bound} and of the conditions \crefrange{ass:pos-mono}{ass:normalize} on quadrature weights, we deduce that $\{W_j^n\}_{j\in\mathbb{Z}}^{n\geq0}$ also satisfies
\begin{align}
    \rho_{\min} \le W_j^n \le \rho_{\max} \quad \textnormal{and}\quad h\sum_{j\in\mathbb{Z}} |W_j^n| = h\sum_{j\in\mathbb{Z}} |\rho_j^0| \qquad \textnormal{for all}\ j\in\mathbb{Z}, \ n\geq0,
\end{align}
with $\rho_{\mathrm{min}}\coloneqq\inf_{x\in\mathbb{R}}\rho_0(x)\geq0$ and $\rho_{\mathrm{max}}\coloneqq\sup_{x\in\mathbb{R}}\rho_0(x)\leq1$.
\end{remark}

\subsection{TVD property and limits of approximate solutions}
\label{ssec:TVD}

In this subsection, we show that the scheme \cref{eq:W-2} for $W$ exhibits the TVD property with respect to the spatial variable, leading us to derive a uniform estimate of the total variation of $W_{\eps,h}$ in space and time.

\begin{lemma}[TVD in space]\label{lm:TVD-x-2}
Let us assume that \crefrange{ass:ic}{ass:V} hold, the quadrature weights satisfy \crefrange{ass:pos-mono}{ass:convexity}, and the CFL condition \cref{ass:CFL} holds. Let $\{\rho_j^n\}_{j\in\mathbb{Z}}^{n\geq0}$ and $\{W_j^n\}_{j\in\mathbb{Z}}^{n\geq0}$ be the numerical solutions constructed with the numerical scheme \crefrange{eq:godunov}{eq:ini_cond_discrete}. Then the following spatial $\mathrm{TV}$-estimate holds: 
\begin{align}\label{eq:TVD-x-2}
\sum_{j\in\mathbb{Z}} |W_{j+1}^{n+1}-W_j^{n+1}| \le \sum_{j\in\mathbb{Z}} |W_{j+1}^n-W_j^n| \leq \mathrm{TV}(\rho_0), \qquad \textnormal{for all $n\ge0$.}
\end{align}
\end{lemma}

Before proving this TVD property, we first prove a lemma that will be used in the proof and subsequent results.

\begin{lemma}\label{lm:use_for_tv}
Under the conditions of \cref{lem:maxm_prcp}, we have
\begin{align}\label{eq:TV-p1}
    \sum_{j\in\mathbb{Z}}\rho_{j+k}^n |V(W_{j+k+1}^n) -V(W_j^n)| \leq \sum_{j\in\mathbb{Z}} \left(\sum_{l=0}^k \rho_{j+l}^n\right) |V(W_{j+1}^n) -V(W_j^n)|.
\end{align}
\end{lemma}

\begin{proof}
A direct calculation gives
    \begin{align*}
    \sum_{j\in\mathbb{Z}}\rho_{j+k}^n |V(W_{j+k+1}^n) -V(W_j^n)| &\leq \sum_{j\in\mathbb{Z}}\rho_{j+k}^n \sum_{l=0}^k |V(W_{j+l+1}^n) -V(W_{j+l}^n)| \\
    &= \sum_{l=0}^k \sum_{j\in\mathbb{Z}} \rho_{j+k}^n |V(W_{j+l+1}^n) -V(W_{j+l}^n)| \\
    &= \sum_{l=0}^k \sum_{j\in\mathbb{Z}} \rho_{j+k-l}^n |V(W_{j+1}^n) -V(W_j^n)|  \\
    &= \sum_{j\in\mathbb{Z}} \left(\sum_{l=0}^k \rho_{j+l}^n\right) |V(W_{j+1}^n) -V(W_j^n)|. 
\end{align*}
\end{proof}

\begin{proof}[Proof of \cref{lm:TVD-x-2}]
From \cref{eq:W-2}, we have 
\begin{align*}
    &W_{j+1}^{n+1} - W_j^{n+1} \\
    &= W_{j+1}^n - W_j^n - \lambda \sum_{k=0}^\infty \gamma_k^{\eps,h} \rho_{j+k+1}^n V(W_{j+k+2}^n) + 2 \lambda \sum_{k=0}^\infty \gamma_k^{\eps,h} \rho_{j+k}^n V(W_{j+k+1}^n) \\ & \qquad - \lambda \sum_{k=0}^\infty \gamma_k^{\eps,h} \rho_{j+k-1}^n V(W_{j+k}^n) \\
    &= W_{j+1}^n - W_j^n + \lambda \left(2\gamma_0^{\eps,h} - \gamma_1^{\eps,h}\right) \rho_j^n V(W_{j+1}^n) - \lambda \gamma_0^{\eps,h} \rho_{j-1}^n V(W_j^n) \\
    &\qquad + \lambda \sum_{k=1}^\infty \left(2\gamma_k^{\eps,h}-\gamma_{k-1}^{\eps,h}-\gamma_{k+1}^{\eps,h}\right) \rho_{j+k}^n V(W_{j+k+1}^n) \\
    &= W_{j+1}^n - W_j^n + \lambda \left(2\gamma_0^{\eps,h} - \gamma_1^{\eps,h}\right) \rho_j^n V(W_{j+1}^n) - \lambda \gamma_0^{\eps,h} \rho_{j-1}^n V(W_j^n) \\ & \qquad  + \lambda V(W_j^n) \left( 2\left(W_j^n-\gamma_0^{\eps,h} \rho_j^n\right) - W_{j+1}^n  - \left(W_{j-1}^n - \gamma_0^{\eps,h} \rho_{j-1}^n - \gamma_1^{\eps,h} \rho_j^n\right) \right) \\
    &\qquad + \lambda \sum_{k=1}^\infty \left(2\gamma_k^{\eps,h}-\gamma_{k-1}^{\eps,h}-\gamma_{k+1}^{\eps,h}\right) \rho_{j+k}^n \left(V(W_{j+k+1}^n) -V(W_j^n)\right) \\
    &= W_{j+1}^n - W_j^n + \lambda \left(2\gamma_0^{\eps,h} - \gamma_1^{\eps,h}\right) \rho_j^n \left(V(W_{j+1}^n) - V(W_j^n)\right) + V(W_j^n) (2W_j^n - W_{j-1}^n - W_{j+1}^n) \\
    &\qquad + \lambda \sum_{k=1}^\infty \left(2\gamma_k^{\eps,h}-\gamma_{k-1}^{\eps,h}-\gamma_{k+1}^{\eps,h}\right) \rho_{j+k}^n \left(V(W_{j+k+1}^n) -V(W_j^n)\right) \\
    &= \lambda V(W_j^n)(W_j^n - W_{j-1}^n) + \left( 1-\lambda V(W_j^n) + \lambda \left(2\gamma_0^{\eps,h}-\gamma_1^{\eps,h}\right) \rho_j^n\alpha_j^n \right) (W_{j+1}^n - W_j^n) \\
    &\qquad + \lambda \sum_{k=1}^\infty \left(2\gamma_k^{\eps,h}-\gamma_{k-1}^{\eps,h}-\gamma_{k+1}^{\eps,h}\right) \rho_{j+k}^n \left(V(W_{j+k+1}^n) -V(W_j^n)\right),
\end{align*}
where $\alpha_j^n\coloneqq\frac{V(W_{j+1}^n) - V(W_j^n)}{W_{j+1}^n - W_j^n}$. From this, we obtain
\begin{align}\label{eq:tv_tmp}
    \sum_{j\in\mathbb{Z}} |W_{j+1}^{n+1} - W_j^{n+1}| &\leq \sum_{j\in\mathbb{Z}}\!\Big( 1 + \lambda \Big(2\gamma_0^{\eps,h}-\gamma_1^{\eps,h}\Big) \rho_j^n\alpha_j^n + \lambda(V(W_{j+1}^n)-V(W_j^n)) \Big) |W_{j+1}^n \!-\! W_j^n| \\
    &\qquad + \lambda \sum_{k=1}^\infty \Big(\gamma_{k-1}^{\eps,h}+\gamma_{k+1}^{\eps,h}-2\gamma_k^{\eps,h}\Big) \sum_{j\in\mathbb{Z}}\rho_{j+k}^n \left|V(W_{j+k+1}^n) -V(W_j^n)\right|, \notag
\end{align}
where we have used the CFL condition \cref{ass:CFL} to ensure 
\[ 
1 + \lambda \left(2\gamma_0^{\eps,h}-\gamma_1^{\eps,h}\right) \rho_j^n\alpha_j^n + \lambda\left(V(W_{j+1}^n)-V(W_j^n)\right)\geq0,
\] 
and the condition \cref{ass:convexity} to give $\gamma_{k-1}^{\eps,h}+\gamma_{k+1}^{\eps,h}-2\gamma_k^{\eps,h}\geq0$.

For the second term on the right-hand side of \cref{eq:tv_tmp}, owing to \cref{eq:TV-p1} in \cref{lm:use_for_tv}, we have 
\begin{align*}
    &\sum_{k=1}^\infty \left(\gamma_{k-1}^{\eps,h}+\gamma_{k+1}^{\eps,h}-2\gamma_k^{\eps,h}\right) \sum_{j\in\mathbb{Z}}\rho_{j+k}^n \left|V(W_{j+k+1}^n) -V(W_j^n)\right| \\
    &\leq \sum_{k=1}^\infty \left(\gamma_{k-1}^{\eps,h}+\gamma_{k+1}^{\eps,h}-2\gamma_k^{\eps,h}\right) \sum_{j\in\mathbb{Z}} \left(\sum_{l=0}^k \rho_{j+l}^n\right) \left|V(W_{j+1}^n) -V(W_j^n)\right| \\
    &= \sum_{j\in\mathbb{Z}} \left( \sum_{k=1}^\infty \left(\gamma_{k+1}^{\eps,h}-\gamma_k^{\eps,h}\right) \sum_{l=0}^k \rho_{j+l}^n - \sum_{k=1}^\infty \left(\gamma_k^{\eps,h}-\gamma_{k-1}^{\eps,h}\right) \sum_{l=0}^k \rho_{j+l}^n \right) \left|V(W_{j+1}^n) -V(W_j^n)\right| \\
    &= \sum_{j\in\mathbb{Z}} \left( -\sum_{k=0}^\infty \left(\gamma_{k+1}^{\eps,h}-\gamma_k^{\eps,h}\right) \rho_{j+k+1}^n - \left(\gamma_1^{\eps,h}-\gamma_0^{\eps,h}\right) \rho_j^n \right) \left|V(W_{j+1}^n) -V(W_j^n)\right| \\
    &= \sum_{j\in\mathbb{Z}} \left( W_{j+1}^n-W_j^n + \left(2\gamma_0^{\eps,h}-\gamma_1^{\eps,h}\right)\rho_j^n \right) \left|V(W_{j+1}^n)-V(W_j^n)\right| \\
    &= \sum_{j\in\mathbb{Z}} -\alpha_j^n \left( W_{j+1}^n-W_j^n+ \left(2\gamma_0^{\eps,h}-\gamma_1^{\eps,h}\right)\rho_j^n \right) \left|W_{j+1}^n-W_j^n\right| \\
    &= \sum_{j\in\mathbb{Z}} \left( V(W_j^n)-V(W_{j+1}^n) - \left(2\gamma_0^{\eps,h}-\gamma_1^{\eps,h}\right)\rho_j^n\alpha_j^n \right) |W_{j+1}^n-W_j^n|.
\end{align*}
In the last two lines, we have used the fact that $\alpha_j^n=\frac{V(W_{j+1}^n)-V(W_j^n)}{W_{j+1}^n-W_j^n}\leq0$.

Combining the above estimates, we deduce that
\begin{align*}
    \sum_{j\in\mathbb{Z}} |W_{j+1}^{n+1} - W_j^{n+1}| \leq& \sum_{j\in\mathbb{Z}}\left( 1 + \lambda \left(2\gamma_0^{\eps,h}-\gamma_1^{\eps,h}\right) \rho_j^n\alpha_j^n + \lambda\left(V(W_{j+1}^n)-V(W_j^n)\right) \right) |W_{j+1}^n - W_j^n| \\
    &\qquad + \lambda \sum_{j\in\mathbb{Z}} \left( V(W_j^n)-V(W_{j+1}^n) - \left(2\gamma_0^{\eps,h}-\gamma_1^{\eps,h}\right)\rho_j^n\alpha_j^n \right) |W_{j+1}^n-W_j^n| \\
   &= \sum_{j\in\mathbb{Z}} |W_{j+1}^n-W_j^n| .
\end{align*}
Therefore, we conclude that 
\[
\sum_{j\in\mathbb{Z}} |W_{j+1}^{n} - W_j^{n}| \leq \sum_{j\in\mathbb{Z}} |W_{j+1}^0-W_j^0| \leq \sum_{j \in \mathbb Z} |\rho^0_{j+1} - \rho^0_j| \leq \mathrm{TV}(\rho_0), \qquad \text{for all $n\ge0$.}
\]

\end{proof}

From \cref{lm:TVD-x-2}, we can derive a temporal total variation estimate. 

\begin{lemma}[Temporal TV-estimate]\label{lm:TVD-t-general}
Under the conditions of \cref{lm:TVD-x-2},  the following temporal $\mathrm{TV}$-estimate holds:  
\begin{align}\label{eq:TVD-t-general}
    \sum_{j\in\mathbb{Z}} |W_j^{n+1}-W_j^n| \leq \lambda \left(\|V\|_{\mathrm{L}^{\infty}} + \|V'\|_{\mathrm{L}^{\infty}}\right) \mathrm{TV}(\rho_0), \qquad \textnormal{for all $n\ge0$.}
\end{align}
\end{lemma}

\begin{proof}
It follows from \cref{eq:W-2} and \cref{eq:num_nonlocal_W} that
\begin{align*}
    W_j^{n+1} &= W_j^n + \lambda ( W_{j-1}^n - W_j^n ) V(W_j^n)  \\
    &\qquad + \lambda \sum_{k=0}^\infty \gamma_k^{\eps,h} \Big( \rho_{j+k-1}^n \left( V(W_{j+k}^n) - V(W_j^n) \right) - \rho_{j+k}^n \left( V(W_{j+k+1}^n) - V(W_j^n) \right) \Big) \\
    &= W_j^n + \lambda ( W_{j-1}^n - W_j^n ) V(W_j^n) + \lambda \sum_{k=0}^\infty \left(\gamma_{k+1}^{\eps,h}-\gamma_k^{\eps,h}\right) \rho_{j+k}^n \left( V(W_{j+k+1}^n) - V(W_j^n) \right),
\end{align*}
where in the last step we have used the summation by parts.
Then we have
\begin{align*}
    \sum_{j\in\mathbb{Z}} |W_j^{n+1}-W_j^n| &\leq \lambda \norm{V}_{\mathrm{L}^\infty} \sum_{j\in\mathbb{Z}} |W_j^n-W_{j-1}^n| \\ & \qquad  + \lambda \sum_{k=0}^\infty \left(\gamma_k^{\eps,h}-\gamma_{k+1}^{\eps,h}\right) \sum_{j\in\mathbb{Z}} \rho_{j+k}^n \left| V(W_{j+k+1}^n) - V(W_j^n) \right|.
\end{align*}
Next, we use \cref{eq:TV-p1} in \cref{lm:use_for_tv} to deduce
\begin{align*}
    &\sum_{k=0}^\infty \left(\gamma_k^{\eps,h}-\gamma_{k+1}^{\eps,h}\right) \sum_{j\in\mathbb{Z}} \rho_{j+k}^n \left| V(W_{j+k+1}^n) - V(W_j^n) \right| \\
    &\leq \sum_{j\in\mathbb{Z}} \left( \sum_{k=0}^\infty \left(\gamma_k^{\eps,h}-\gamma_{k+1}^{\eps,h}\right) \left(\sum_{l=0}^k \rho_{j+l}^n\right) \right) |V(W_{j+1}^n) -V(W_j^n)| \\
    &= \sum_{j\in\mathbb{Z}} \left( \gamma_0\rho_j^n + \sum_{k=0}^\infty \gamma_{k+1}^{\eps,h} \rho_{j+k+1}^n \right) |V(W_{j+1}^n) -V(W_j^n)| \\
    &= \sum_{j\in\mathbb{Z}} W_j^n |V(W_{j+1}^n) -V(W_j^n)| \\
    &\leq \|V'\|_{\mathrm{L}^\infty} \sum_{j\in\mathbb{Z}} |W_{j+1}^n-W_j^n|.
\end{align*}
Then, using \cref{lm:TVD-x-2}, we conclude that \cref{eq:TVD-t-general} holds.
\end{proof}

From the uniform $\mathrm{L}^\infty$-bounds and total variation estimates for $\{W_j^n\}_{j\in\mathbb{Z}}^{n\geq0}$ obtained in the previous lemmas, we are able to show the strong convergence in $\mathrm{L}^1_{\mathrm{loc}}$ and almost everywhere of the piecewise constant reconstruction of $\{W_j^n\}_{j\in\mathbb{Z}}^{n\geq0}$ to a limit point, up to a subsequence. 

\begin{lemma}[Convergence] \label{lm:limit}
Let us assume that \crefrange{ass:ic}{ass:V} hold, the quadrature weights satisfy \crefrange{ass:pos-mono}{ass:convexity}, and the CFL condition \cref{ass:CFL} holds with the CFL ratio $\lambda\coloneqq \tau/h$ fixed. Let $\{\rho_j^n\}_{j\in\mathbb{Z}}^{n\geq0}$ and $\{W_j^n\}_{j\in\mathbb{Z}}^{n\geq0}$ be the numerical solutions constructed with the numerical scheme \crefrange{eq:godunov}{eq:ini_cond_discrete}. Then the approximate solution $W_{\eps,h}$ constructed from $\{W_j^n\}_{j\in\mathbb{Z}}^{n\geq0}$ using \cref{eq:W_piece_recon}, as $\eps,h \searrow 0$, converges strongly in $\mathrm{L}^1_{\mathrm{loc}}(\R_+\times \R)$ and almost everywhere, up to a subsequence, to a limit point $\rho^*\in \mathrm{L}^1_{\mathrm{loc}}(\R_+\times \R)$. Moreover, $0 \le \rho^* \le 1$ almost everywhere.  
\end{lemma}

\begin{proof}
Leveraging the $\mathrm{L}^\infty$-bounds from \cref{lem:maxm_prcp} and \cref{rk:W}, and the spatial and temporal TV-estimates from \cref{lm:TVD-x-2} and \cref{lm:TVD-t-general}, all uniform in $\varepsilon$ and $h$, we can apply Helly's compactness theorem (see, e.\,g., \cite[Theorem 2.3, p.~14]{MR1816648}) to conclude the strong convergence of a subsequence of $W_{\eps,h}$ as $\eps,h \searrow 0$, in $\mathrm{L}^1_{\mathrm{loc}}$ and almost everywhere, to a limit point $\rho^*\in \mathrm{L}^1_{\mathrm{loc}}(\R_+\times \R)$. The fact that $0 \le \rho^* \le 1$ follows from \cref{lem:maxm_prcp} and \cref{rk:W}.
\end{proof}

\subsection{Consistency with the entropy admissibility condition}
\label{ssec:entropy-consistency-2}

We now need to show that the limit point $\rho^*$ obtained in \cref{lm:limit} is the (unique) entropy solution $\rho$ of \cref{eq:cl}.
To this end, we derive a discrete entropy inequality for $W^n_j$ that aligns with a continuous entropy inequality, adapted from \cite{MR4553943} to support a discrete version.
We outline below the key (formal) computation in the continuous setting.

Let $\eta \in \mathrm{C}^2(\R)$ be a convex \emph{entropy} function and let $\psi$ be the corresponding \emph{entropy flux}, which is defined by $\psi'(\xi) \coloneqq \eta'(\xi) \, (V(\xi)\,\xi)'$. Then, applying the chain rule on \cref{eq:cle}, we compute
\begin{align*}
    \pt \eta(W_\eps) + \partial_x \psi(W_\eps) &= \eta'(W_\eps)  \px \big(V(W_\eps)W_\eps - (V(W_\eps) \rho_\eps) \ast \gamma_\eps\big)
    \\&= \px \big(\eta'(W_\eps)   \big(V(W_\eps)W_\eps - (V(W_\eps) \rho_\eps) \ast \gamma_\eps\big)\big) \\ & \qquad - \eta''(W_\eps)   \px W_\eps \big(V(W_\eps)W_\eps - (V(W_\eps) \rho_\eps) \ast \gamma_\eps\big)
    \\ &\eqqcolon I_{1,\eps} + I_{2,\eps}. 
\end{align*}
We introduce the function $I_\eta'(\xi) \coloneqq  \eta''(\xi)V(\xi)$ and rewrite $I_{2,\eps}$:
\begin{align*}
    I_{2,\eps} &= \partial_x \eta'(W_\eps) \, ((V(W_\eps) \rho_\eps) \ast \gamma_\eps) - \partial_x I_\eta(W_\eps) \, W_\eps \\
    &= \partial_x \eta'(W_\eps) \, ((V(W_\eps) \rho_\eps) \ast \gamma_\eps) - \partial_x I_\eta(W_\eps) \, (\rho_\eps \ast \gamma_\eps) \\ &\eqqcolon I_{2a,\eps} - I_{2b,\eps}.
\end{align*}
In turn, we write 
\begin{align*}
    I_{2a,\eps} 
    &= \partial_x \big( \eta'(W_\eps) (V(W_\eps) \rho_\eps) \ast \gamma_\eps - (\eta'(W_\eps) V(W_\eps) \rho_\eps) \ast \gamma_\eps \big) \\
    &\qquad - \big( \eta'(W_\eps) (V(W_\eps) \rho_\eps) \ast \gamma'_\eps - (\eta'(W_\eps) V(W_\eps) \rho_\eps) \ast \gamma'_\eps \big);
\\
    I_{2b,\eps} &= \partial_x \big( I_\eta(W_\eps)  (\rho_\eps \ast \gamma_\eps) - (I_\eta(W_\eps) \rho_\eps) \ast \gamma_\eps \big) \\ &\qquad - \big( I_\eta(W_\eps) (\rho_\eps \ast \gamma'_\eps) - (I_\eta(W_\eps) \rho_\eps) \ast \gamma'_\eps \big).
\end{align*}
By introducing the function $H_\eta(a\mid b) \coloneqq  I_\eta(b) - I_\eta(a) - V(b)(\eta'(b)-\eta'(a))$, we can write
\begin{align*}
    I_{2,\eps}   &= \partial_x\int_x^\infty \rho_\eps(t,y) H_\eta(W_\eps(t,x)\mid W_\eps(t,y)) \gamma_\eps(x-y)\,\mathrm dy \\ &\qquad- \int_x^\infty \rho_\eps(t,y) H_\eta(W_\eps(t,x)\mid W_\eps(t,y)) \gamma'_\eps(x-y)\,\mathrm dy
\\ &= \partial_x \big( \big( \rho_\eps(t,\cdot) H_\eta(W_\eps(t,x)\mid W_\eps(t,\cdot)) \big) \ast\gamma_\eps \big) - \big( \rho_\eps(t,\cdot)  H_\eta(W_\eps(t,x)\mid W_\eps(t,\cdot)) \big) \ast\gamma'_\eps.
\end{align*}
In conclusion, 
\begin{align*}
 \pt \eta(W_\eps) + \partial_x \psi(W_\eps) &= \px \big(\eta'(W_\eps)   \big(V(W_\eps)W_\eps - (V(W_\eps) \rho_\eps) \ast \gamma_\eps\big)\big) \\ & \qquad +
     \partial_x \big( \big( \rho_\eps(t,\cdot) H_\eta(W_\eps(t,x)\mid W_\eps(t,\cdot)) \big) \ast\gamma_\eps \big) \\ & \qquad  - \big( \rho_\eps(t,\cdot) H_\eta(W_\eps(t,x)\mid W_\eps(t,\cdot)) \big) \ast\gamma'_\eps.
\end{align*}
Moreover, since the function $H_\eta$ satisfies $H_\eta(a\mid b)\geq H_\eta(b\mid b)=0$ for all $a,b\in\mathbb{R}$ and $\gamma'_\eps\geq0$, we deduce that $\big( \rho_\eps H_\eta(W_\eps(t,x)\mid W_\eps(t,\cdot)) \big) \ast\gamma'_\eps\geq0$, which yields 
\begin{align*}
 \pt \eta(W_\eps) + \partial_x \psi(W_\eps) &\le \px \big(\eta'(W_\eps)   \big(V(W_\eps)W_\eps - (V(W_\eps) \rho_\eps) \ast \gamma_\eps\big)\big) \\ & \qquad +
     \partial_x \big( \big( \rho_\eps H_\eta(W_\eps(t,x)\mid W_\eps(t,\cdot)) \big) \ast\gamma_\eps \big). 
\end{align*}

This motivates us to define the \emph{nonlocal} entropy flux function
\begin{align*}
    \Psi_{\gamma_\eps}(\rho_\eps,W_\eps) \coloneqq  \psi(W_\eps) - \eta'(W_\eps) \big(V(W_\eps)W_\eps - (V(W_\eps) \rho_\eps) \ast \gamma_\eps\big) - \big( \rho_\eps  H_\eta(W_\eps(t,x)\mid W_\eps(t,\cdot))  \big) \ast \gamma_\eps,
\end{align*}
which leads to a nonlocal entropy inequality
\begin{align*}
     \pt \eta(W_\eps) + \partial_x \Psi_{\gamma_\eps}(\rho_\eps,W_\eps) \le 0.
\end{align*}
Specifically, taking \textit{Kru\v{z}kov's entropy function}\footnote{~Kru\v{z}kov's entropy does not belong to $\mathrm{C}^2(\R)$, but an approximation argument solves this technical issue; see \cite[Chapter 2, pp.~56--58]{MR3443431}.} $\eta_c(\xi)\coloneqq |\xi-c|$, for any constant $c\in\R$, we have $\eta_c'(\xi)=\sgn{\xi-c}$, $\psi_c(\xi)=\sgn{\xi-c}\cdot(\xi\, V(\xi)-c\,V(c))$, $I_{\eta_c}(\xi)=\sgn{\xi-c}\cdot V(c)$, and $H_{\eta_c}(a\mid b) = |V(b)-V(c)| + \sgn{a-c}(V(b)-V(c))$. 
A direct calculation gives
\[ \Psi_{c,\gamma_\eps}(\rho_\varepsilon,W_\varepsilon) = |W_\varepsilon-c|V(c) - \left( \rho_\varepsilon|V(W_\varepsilon)-V(c)| \right) \ast \gamma_\eps, \]
and the following estimate regarding the compatibility between $\Psi_{c,\gamma_\eps}$ and $\psi_c$:
\[ \left|\Psi_{c,\gamma_\eps}(\rho_\varepsilon,W_\varepsilon) - \psi_c(W_\eps) \right| \leq \left| V(W_\eps(t,\cdot)) - V(W_\eps(t,x))\right| \ast \gamma_\eps. \]
Thus, by passing to the limit $\eps\searrow 0$ and using $W_\eps\to\rho^\ast$, we deduce that
\[ \lim_{\eps\to0} \Psi_{c,\gamma_\eps}(\rho_\eps,W_\eps) = \psi_c(\rho^\ast). \]

These considerations lead us to formulating the following discrete entropy inequality.

\begin{lemma}[Discrete entropy inequality]\label{lm:discrete-entropy-2}
Under the conditions of \cref{lem:maxm_prcp}, the following discrete entropy inequality holds:
\begin{align}\label{eq:discrete-entropy-2}
    &\frac{|W_j^{n+1}-c| -|W_j^n-c|}{\tau} + \frac{\Psi_{j+1/2}^n - \Psi_{j-1/2}^n}{h} \leq 0, \qquad \mathrm{for\ all}\  c\in\mathbb{R}, \\
    \label{eq:entropy-flux-2} &\mathrm{with} \quad \Psi_{j-1/2}^n \coloneqq |W_{j-1}^n-c|V(c) - \sum_{k=0}^\infty \gamma_k^{\eps,h} \rho_{j+k-1}^n \left|V(W_{j+k}^n)-V(c)\right|.
\end{align}
\end{lemma}

\begin{proof}
From \cref{eq:W-2}, we have 
\begin{align*}
    W_j^{n+1}-c 
    &= W_j^n-c + \lambda \sum_{k=0}^\infty \gamma_k^{\eps,h} \left( \rho_{j+k-1}^n \left(V(W_{j+k}^n)-V(c)\right) - \rho_{j+k}^n \left(V(W_{j+k+1}^n)-V(c)\right) \right) \\ &\qquad + \lambda V(c) (W_{j-1}^n-W_j^n) \\
    &= \left(1-\lambda V(c)\right) (W_j^n-c) + \lambda V(c) (W_{j-1}^n-c) \\
    &\qquad + \lambda \sum_{k=0}^\infty \gamma_k^{\eps,h} \left( \rho_{j+k-1}^n \left(V(W_{j+k}^n)-V(c)\right) - \rho_{j+k}^n \left(V(W_{j+k+1}^n)-V(c)\right) \right) \\
    &= \left(1-\lambda V(c)\right) (W_j^n-c) + \lambda V(c) (W_{j-1}^n-c) + \lambda \gamma_0^{\eps,h}\rho_{j-1}^n \left(V(W_j^n)-V(c)\right) \\
    &\qquad +\lambda \sum_{k=0}^\infty \left(\gamma_{k+1}^{\eps,h}-\gamma_k^{\eps,h}\right) \rho_{j+k}^n \left(V(W_{j+k+1}^n)-V(c)\right),
\end{align*}
where we have used the summation by parts in the last step. Noting that 
\[ \left|\lambda \gamma_0^{\eps,h}\rho_{j-1}^n \left(V(W_j^n)-V(c)\right)\right| \leq \lambda \|V'\|_{\mathrm{L}^\infty} |W_j^n-c| \leq \left(1-\lambda V(c)\right) |W_j^n-c|, \]
provided the CFL condition \cref{ass:CFL-max}, which implies that
\begin{align*}
    &\left|\left(1-\lambda V(c)\right) (W_j^n-c) + \lambda \gamma_0^{\eps,h}\rho_{j-1}^n \left(V(W_j^n)-V(c)\right)\right| \\
    &= \sgn{W_j^n-c} \left( \left(1-\lambda V(c)\right) (W_j^n-c) + \lambda \gamma_0^{\eps,h}\rho_{j-1}^n \left(V(W_j^n)-V(c)\right) \right) \\
    &= \left(1-\lambda V(c)\right) |W_j^n-c| - \lambda \gamma_0^{\eps,h}\rho_{j-1}^n \left|V(W_j^n)-V(c)\right|.
\end{align*}
Then we obtain
\begin{align*}
    |W_j^{n+1}-c| 
    &\leq (1-\lambda V(c)) |W_j^n-c| + \lambda V(c) |W_{j-1}^n-c| - \lambda \gamma_0^{\eps,h}\rho_{j-1}^n \left|V(W_j^n)-V(c)\right| \\
    &\qquad -\lambda \sum_{k=0}^\infty \left(\gamma_{k+1}^{\eps,h}-\gamma_k^{\eps,h}\right) \rho_{j+k}^n \left|V(W_{j+k+1}^n)-V(c)\right| \\
    &= |W_j^n-c| + \lambda V(c) \left(|W_{j-1}^n-c|-|W_j^n-c|\right) \\
    &\qquad -\lambda \sum_{k=0}^\infty \gamma_k^{\eps,h} \rho_{j+k-1}^n \left|V(W_{j+k}^n)-V(c)\right| +\lambda \sum_{k=0}^\infty \gamma_k^{\eps,h} \rho_{j+k}^n \left|V(W_{j+k+1}^n)-V(c)\right| \\
    &= |W_j^n-c| + \lambda (\Psi_{j-1/2}^n - \Psi_{j+1/2}^n),
\end{align*}
which leads to the desired entropy inequality \crefrange{eq:discrete-entropy-2}{eq:entropy-flux-2}.
\end{proof}

With the discrete entropy inequality in \cref{lm:discrete-entropy-2}, we can now show that the limit point $\rho^\ast$ obtained in \cref{lm:limit} is entropy-admissible. 

\begin{lemma}[Entropy admissibility]\label{lem:entropy_admissibility-2}
Let us assume that \crefrange{ass:ic}{ass:V} hold, the quadrature weights satisfy \crefrange{ass:pos-mono}{eq:quad_weight_localize}, and the CFL condition \cref{ass:CFL} holds with the CFL ratio $\lambda\coloneqq \tau/h$ fixed. Let $\{\rho_j^n\}_{j\in\mathbb{Z}}^{n\geq0}$ and $\{W_j^n\}_{j\in\mathbb{Z}}^{n\geq0}$ be the numerical solutions constructed with the numerical scheme \crefrange{eq:godunov}{eq:ini_cond_discrete}. Then any limit point $\rho^\ast$ (as $\eps,h \searrow 0$) of the approximate solution $W_{\eps,h}$ constructed from $\{W_j^n\}_{j\in\mathbb{Z}}^{n\geq0}$ using \cref{eq:W_piece_recon}, in the strong topology of $\mathrm{L}^1_{\mathrm{loc}}$, is the unique entropy solution of \cref{eq:cl}.
\end{lemma}

\begin{proof}
From \cref{lm:discrete-entropy-2}, multiplying \cref{eq:discrete-entropy-2} by $ \phi_j^n \tau h\coloneqq\phi(n\tau,jh)\tau h$, where $\phi\in \mathrm{C}_{\mathrm{c}}^1(\mathbb{R}_+\times\mathbb{R})$ is any test function, and summing it over all $j\in\mathbb{Z}$ and $n\geq0$, we obtain
\begin{align*}
    \tau h\sum_{n=0}^\infty\sum_{j\in\mathbb{Z}} |W_j^{n+1}-c| \frac{\phi_j^{n+1}-\phi_j^n}{\tau} + \Psi_{j-1/2}^n \frac{\phi_j^n-\phi_{j-1}^n}{h} \geq 0,
\end{align*}
for all $c\in\mathbb{R}$. By passing to the limit as $\tau=\lambda h\searrow 0,\,\eps\searrow 0$, and using the fact that
\begin{align*}
    \Psi_{j-1/2}^n &= |W_{j-1}^n-c|V(c) - \sum_{k=0}^\infty \gamma_k^{\eps,h} \rho_{j+k-1}^n \left|V(W_{j+k}^n)-V(c)\right| \\
    &= |W_{j-1}^n-c|V(c) - \left|V(W_j^n)-V(c)\right| W_{j-1}^n \\
    &\qquad + \sum_{k=0}^\infty \gamma_k^{\eps,h} \rho_{j+k-1}^n \left( \left|V(W_j^n)-V(c)\right| - \left|V(W_{j+k}^n)-V(c)\right| \right),
\end{align*}
where the last term is controlled by $\|V'\|_{\mathrm{L}^\infty} \sum_{k=0}^\infty \gamma_k^{\eps,h} |W_{j+k}^n - W_j^n|$,
we deduce that
\begin{align*}
    &\int_0^\infty\int_{\mathbb{R}} \big( |\rho^\ast-c|\partial_t\phi + \psi_c(\rho^\ast)\partial_x\phi \big) \dd x\dd t + \limsup_{\substack{\eps,h\searrow0}} \sum_{k=0}^\infty \gamma_k^{\eps,h} H^{\eps,h}_k \geq 0,
\end{align*}
where $\psi_c(\xi)\coloneqq\sgn{\xi-c}(\xi V(\xi)-cV(c))=|\xi-c|V(c)-|V(\xi)-V(c)|\xi$ is Kru\v{z}kov's entropy flux function, and
\begin{align*}
    H^{\eps,h}_k \coloneqq \tau h \sum_{n=0}^\infty\sum_{j\in\mathbb{Z}} |W_{j+k}^n - W_j^n| \left|\frac{\phi_j^n - \phi_{j-1}^n}{h}\right|.
\end{align*}
We keep in mind that all $W_j^n$ are $\eps,h$-dependent.

Suppose that $\mathrm{supp}(\phi)\subset[0,T]\times[-M,M]$. Following \cite[Lemma 4.1 and Theorem 1.2]{MR4553943}, we split the sum over \( k \) and write
\begin{align*}
\sum_{k=0}^\infty \gamma_k^{\eps,h} H^{\eps,h}_k
&= \sum_{k=0}^\infty \mathds{1}_{\frac{kh}\eps< R}\gamma_k^{\eps,h} H^{\eps,h}_k + \sum_{k=0}^\infty \mathds{1}_{\frac{kh}\eps\ge R}\gamma_k^{\eps,h} H^{\eps,h}_k,
\end{align*}
for any $R>0$.

On the one hand, noting that $H^{\eps,h}_k$ is uniformly bounded, i.\,e.,
$$
H^{\eps,h}_k \le 4 \|\partial_x \phi\|_{\mathrm{L}^\infty} MT \quad \text{for all}\ k\geq0,\,\eps,h>0,
$$
by \cref{eq:quad_weight_localize} we have
\begin{align*}
    \lim_{R\to\infty} \sup_{\eps,h>0}  \sum_{k=0}^\infty \mathds{1}_{\frac{kh}\eps\ge R}\gamma_k^{\eps,h} H^{\eps,h}_k  = 0.
\end{align*}
On the other hand, when $\frac{kh}\eps< R$, we have
\begin{align*}
    H^{\eps,h}_k &\le \|\partial_x \phi\|_{\mathrm{L}^\infty} \int_0^T\int_{-M-kh}^M |W_{\eps,h}(t,x+kh)-W_{\eps,h}(t,x)| \dx\dt \\
    &\leq \|\partial_x \phi\|_{\mathrm{L}^\infty} \sup_{0\le\xi\le R\eps} \int_0^T\int_{-M-R\eps}^M |W_{\eps,h}(t,x+\xi)-W_{\eps,h}(t,x)| \dx\dt \\
    &\coloneqq J^{\eps,h}_R.
\end{align*}
By Fréchet--Kolmogorov--Riesz--Sudakov's theorem (see \cite{zbMATH05817642,zbMATH06579393,zbMATH07077498}), we deduce that $J^{\eps,h}_R\to0$ as $\eps,h\searrow0$. 

Now, for any $\nu>0$, we first choose $R$ sufficiently large such that
\begin{align*}
    \sup_{\eps,h>0} \sum_{k=0}^\infty \mathds{1}_{\frac{kh}\eps\ge R}\gamma_k^{\eps,h} H^{\eps,h}_k \leq \nu.
\end{align*}
Then, we estimate
\begin{align*}
    \sum_{k=0}^\infty \mathds{1}_{\frac{kh}\eps< R}\gamma_k^{\eps,h} H^{\eps,h}_k  \leq \|\partial_x \phi\|_{\mathrm{L}^\infty} \sum_{k=0}^\infty \mathds{1}_{\frac{kh}\eps< R}\gamma_k^{\eps,h} J^{\eps,h}_R \le \|\partial_x \phi\|_{\mathrm{L}^\infty} J^{\eps,h}_R,
\end{align*}
which converges to zero as $\eps,h\searrow0$, implying that $\limsup_{\eps,h\searrow0} \sum_{k=0}^\infty \gamma_k^{\eps,h} H^{\eps,h}_k \le \nu$.
Since $\nu>0$ is arbitrarily chosen, we conclude that $\limsup_{\eps,h\searrow0} \sum_{k=0}^\infty \gamma_k^{\eps,h} H^{\eps,h}_k=0$.
Therefore, we obtain
\begin{align*}
    \int_0^\infty\int_{\mathbb{R}} \big( |\rho^\ast-c|\partial_t\phi + \psi_c(\rho^\ast)\partial_x\phi \big) \dd x\dd t \geq 0,
\end{align*}
for all $\phi\in \mathrm{C}_{\mathrm{c}}^1(\mathbb{R}_+\times\mathbb{R})$ and $c\in\mathbb{R}$, thus $\rho^\ast$ is the unique entropy solution of \cref{eq:cl}.
\end{proof}

\subsection{Asymptotically compatible Kuznetsov's convergence rate}
\label{ssec:kuz}

We now use Kuznetsov's argument (see \cite{MR483509}) to establish a convergence rate estimate for $W_{\eps,h}$ towards the unique entropy solution $\rho$ of \cref{eq:cl}. A similar approach was carried out in \cite[Proposition~5.1 \& pp.~18--22]{MR4553943} to quantify the nonlocal-to-local limit from $W_\varepsilon$ to $\rho$. In contrast, our setting involves both the nonlocal horizon parameter $\varepsilon$ and the discretization parameter $h$, with the argument extended to address.

Throughout this subsection, we let $T>0$ be a selected time and estimate $\|W_{\eps,h}(T,\cdot)-\rho(T,\cdot)\|_{\mathrm{L}^1}$.

\begin{lemma}[Convergence rate]\label{th:rate}
Let us assume that \crefrange{ass:ic}{ass:V} hold, the quadrature weights satisfy  \crefrange{ass:pos-mono}{ass:convexity} and \cref{ass:gamma-mom}, and the CFL condition \cref{ass:CFL} holds with the CFL ratio $\lambda\coloneqq \tau/h$ fixed. Let $\{\rho_j^n\}_{j\in\mathbb{Z}}^{n\geq0}$ and $\{W_j^n\}_{j\in\mathbb{Z}}^{n\geq0}$ be the numerical solutions constructed with the numerical scheme \crefrange{eq:godunov}{eq:ini_cond_discrete}. Let $W_{\eps,h}$ be the approximate solution constructed from $\{W_j^n\}_{j\in\mathbb{Z}}^{n\geq0}$ using \cref{eq:W_piece_recon}, and $\rho$ be the unique entropy solution of \cref{eq:cl}.
Then the following error estimate holds:
\begin{align*}
\|W_{\eps,h}(T,\cdot)-\rho(T,\cdot)\|_{\mathrm{L}^1} \leq K\left( \eps+h+\sqrt{\eps\, T}+\sqrt{h\, T} \right) \mathrm{TV}(\rho_0), \quad \textnormal{for every } \eps,h > 0, \, T > 0,
\end{align*}
where the constant $K>0$ only depends on $\lambda$, $\|V\|_{\mathrm{L}^\infty}$, $\|V'\|_{\mathrm{L}^\infty}$, and $c_\gamma$ (as specified in \cref{ass:gamma-mom}).
\end{lemma}

A key tool in the proof of \cref{th:rate} is the following lemma (see \cite{MR483509}; we also refer to \cite[Theorem 3.14]{MR3443431}, \cite[Lemma 3.2]{MR4700412}, and \cite[Lemma 2.1]{MR1825698}), which provides an explicit estimate of the difference between any function $v$ in a suitable class $\mathcal K$ and the entropy solution $u$ of a scalar conservation law \cref{eq:cl-f} in terms of their \emph{relative (Kru\v{z}kov's) entropy. 
}

\begin{lemma}[Kuznetsov's lemma] \label{lm:kut} 
Let 
\begin{align*} v \in \mathcal K \coloneqq \Big\{ v: \R_+\times \R \to \R \, : \ & v(t,\cdot) \in \mathrm{L}^1_{\mathrm{loc}}(\R) \textnormal{ and the right and left limits $v(t \pm, \cdot)$ exist in $\mathrm L^1_{\mathrm{loc}}$} \\ 
&\textnormal{for all $t>0$ \quad\ \, and \quad\ \,  } \|v\|_{\mathrm{L}^\infty(\R_+\times \R)} + \sup\nolimits_{t>0} \mathrm{TV}(v(t,\cdot)) < +\infty\Big\}
\end{align*}
and let $u$ be the entropy solution of the scalar conservation law 
\begin{align}\label{eq:cl-f}
    \begin{cases}
        \pt u(t,x) + \px f\big(u(t,x)\big) = 0, & (t,x) \in \R_+ \times \R, \\
        u(0,x) = u_0(x), & x \in\R.
    \end{cases}
\end{align}
If $0<\delta_0<T$ and $\delta>0$, then
\begin{align*}
\|v(T-,\cdot)-u(T,\cdot)\|_{\mathrm{L}^1} &\leq \left\|v(0,\cdot)-u_0\right\|_{\mathrm{L}^1}+\operatorname{TV}\left(u_0\right)\big(2 \delta+\delta_0\|f'\|_{\mathrm{L}^{\infty}}\big)  +\nu\left(v, \delta_0\right)-\Lambda_{\delta, \delta_0}(v, u),
\end{align*}
where
\begin{align*}
\nu_t(v, \sigma)&\coloneqq \sup_{s\in[0,\sigma]} \|v(t+s,\cdot)-v(t,\cdot)\|_{\mathrm{L}^1}, \qquad \mathrm{for\ all}\ t>0,\,\sigma>0, \\ 
\nu(v, \sigma)&\coloneqq \sup _{t \in[0, T]} \nu_t(v, \sigma), \qquad \mathrm{for\ all}\ \sigma>0, \\
\Lambda_T(v, \phi, c)&\coloneqq  \int_0^T \int_{\mathbb{R}}\left(|v-c| \partial_t \phi+\psi_c(v) \partial_x \phi\right) \, \mathrm d x \, \mathrm d t -\int_{\mathbb{R}}|v(T,x)-c| \phi(T,x) \, \mathrm  d x \\
& \qquad\qquad +\int_{\mathbb{R}}|v(0,x)-c| \phi(0,x) \, \mathrm d x, \qquad \mathrm{for\ all}\ \phi \in \mathrm{C}_\mathrm{c}^1(\mathbb{R}_+\times\mathbb{R},\mathbb{R}_+) ,\, c\in\R, \\
\Lambda_{\delta, \delta_0}(v, u)&\coloneqq  \int_0^T \int_{\mathbb{R}} \Lambda_T\left(v, \Omega\left(\cdot, t^{\prime}, \cdot, x^{\prime}\right), u\left(t^{\prime},x^{\prime}\right)\right) \, \mathrm d x^{\prime} \, \mathrm d t', 
\end{align*}
with  $\omega \in C_c^{\infty}(\mathbb{R})$ being a standard mollifier, i.\,e., an even function satisfying $\operatorname{supp} \omega \subset[-1,1]$, $0 \leq$ $\omega \leq 1$, $\omega$ increasing on $[-1,0]$, and $\int_{\mathbb{R}} \omega(z) \, \mathrm  d z=1$, and $\omega_{\delta}(z)\coloneqq\frac{1}{\delta} \omega\left(\frac{z}{\delta}\right)$, 
$$
\Omega\left(t, t^{\prime},x, x^{\prime}\right)\coloneqq\omega_{\delta_0}\left(t-t^{\prime}\right) \omega_{\delta}\left(x-x^{\prime}\right), \qquad\left(t, t^{\prime},x, x^{\prime}\right) \in \mathbb{R}^4,
$$
and $\psi_c(\xi)\coloneqq\sgn{\xi-c}(f(\xi)-f(c))$ is Kru\v{z}kov's entropy flux function.
\end{lemma}

\begin{proof}[Proof of \cref{th:rate}]
We first assume $T=N\tau$ for $N\in\N_+$ and apply \cref{lm:kut} to $v=W_{\eps,h}$ and $u=\rho$, using the flux function $f(\rho)=\rho V(\rho)$ and Kru\v{z}kov's entropy flux function $\psi_c(\xi)=\sgn{\xi-c}(\xi V(\xi) - cV(c))$. We denote $t_n\coloneqq n\tau$ for $n\geq0$ and $x_{j-1/2}\coloneqq(j-\frac12)h$ for $j\in\mathbb{Z}$.

It is straightforward to verify that $W_{\eps,h}\in\mathcal{K}$. Then \cref{lm:kut} gives
\begin{align*}
\|W_{\eps,h}(T-,\cdot)-\rho(T,\cdot)\|_{\mathrm{L}^1} &\leq \left\|W_{\eps,h}(0,\cdot)-\rho_0\right\|_{\mathrm{L}^1}+\operatorname{TV}\left(\rho_0\right)\left(2 \delta+\delta_0\|V\|_{\mathrm{W}^{1,\infty}}\right) \\
&\qquad +\nu\left(W_{\eps,h}, \delta_0\right)-\Lambda_{\delta, \delta_0}(W_{\eps,h}, \rho),
\end{align*}
for any $0<\delta_0<T$ and $\delta>0$, where $\nu$ and $\Lambda_{\delta, \delta_0}$ are as defined in the statement of \cref{lm:kut}.

\uline{Step 1.} For the first term $\left\|W_{\eps,h}(0,\cdot)-\rho_0\right\|_{\mathrm{L}^1}$, we have
\begin{align*}
    \left\|W_{\eps,h}(0,\cdot)-\rho_0\right\|_{\mathrm{L}^1} \leq \left\|W_{\eps,h}(0,\cdot)-\rho_{\eps,h}(0,\cdot)\right\|_{\mathrm{L}^1} + \left\|\rho_{\eps,h}(0,\cdot)-\rho_0\right\|_{\mathrm{L}^1},
\end{align*}
where
\begin{align*}
    \left\|W_{\eps,h}(0,\cdot)-\rho_{\eps,h}(0,\cdot)\right\|_{\mathrm{L}^1} = h\sum_{j\in\mathbb{Z}} |W_j^0-\rho_j^0| \leq h \sum_{k=0}^\infty k\gamma_k^{\eps,h}\sum_{j\in\mathbb{Z}} |\rho_{j+1}^0-\rho_j^0| \leq c_\gamma\eps \, \mathrm{TV}(\rho_0),
\end{align*}
with $c_\gamma$ specified in \cref{ass:gamma-mom}, and
\begin{align*}
    \left\|\rho_{\eps,h}(0,\cdot)-\rho_0\right\|_{\mathrm{L}^1} = \sum_{j\in\mathbb{Z}} \int_{x_{j-1/2}}^{x_{j+1/2}} |\rho_0(x)-\rho_j^0|\dx \leq h \, \mathrm{TV}(\rho_0).
\end{align*}
Therefore, we have
\begin{align*}
    \left\|W_{\eps,h}(0,\cdot)-\rho_0\right\|_{\mathrm{L}^1} \leq (c_\gamma\eps+h)\, \mathrm{TV}(\rho_0).
\end{align*}

\uline{Step 2.} For the term $\nu(W_{\eps,h}, \delta_0)=\sup_{t\in[0,T]}\nu_t(W_{\eps,h}, \delta_0)$, we first estimate
\[ \nu_t(W_{\eps,h}, \delta_0)= \sup_{0\leq s\leq\delta_0} \|W_{\eps,h}(t+s,\cdot)-W_{\eps,h}(t,\cdot)\|_{\mathrm{L}^1}. \]
For any $t\in[0,T]$ and $s\in[0,\delta_0]$, suppose that $t\in[t_m,t_{m+1}[$ and $t+s\in[t_n,t_{n+1}[$. Owing to the temporal TV-estimate \cref{eq:TVD-t-general} in \cref{lm:TVD-t-general}, we have
\begin{align*}
    \|W_{\eps,h}(t+s,\cdot)-W_{\eps,h}(t,\cdot)\|_{\mathrm{L}^1} = \sum_{j\in\mathbb{Z}} |W_j^n-W_j^m| h \leq \lambda (n-m)h \norm{V}_{\mathrm{W}^{1,\infty}} \mathrm{TV}(\rho_0).
\end{align*}
Noting that $(n-m-1)\tau\le s$ and $\tau=\lambda h$, and taking the supremum over $s\in[0,\delta_0]$ and $t\in[0,T]$ in the above inequality, we obtain that
\begin{align*}
    \nu\left(W_{\eps,h}, \delta_0\right) \leq (\delta_0+\tau) \norm{V}_{\mathrm{W}^{1,\infty}} \mathrm{TV}(\rho_0).
\end{align*}

\uline{Step 3.} For the term $\Lambda_{\delta, \delta_0}(W_{\eps,h}, \rho)$, we first consider
\begin{align*}
    \Lambda_T(W_{\eps,h}, \phi, c)= & \int_0^T \int_{\mathbb{R}}\left(|W_{\eps,h}-c| \partial_t \phi+\psi_c(W_{\eps,h}) \partial_x \phi\right) \, \mathrm d x \, \mathrm d t \\
& \qquad -\int_{\mathbb{R}}|W_{\eps,h}(T,x)-c| \phi(T,x) \, \mathrm  d x+\int_{\mathbb{R}}|W_{\eps,h}(0,x)-c| \phi(0,x) \, \mathrm d x,
\end{align*}
where $\phi\in\mathrm{C}^1_{\mathrm{c}}(\mathbb{R}_+\times\mathbb{R},\mathbb{R}_+)$ and $c\in\R$.
We have
\begin{align*}
    &\int_0^T \int_{\mathbb{R}}\left(|W_{\eps,h}-c| \partial_t \phi+\psi_c(W_{\eps,h}) \partial_x \phi\right) \, \mathrm d x \, \mathrm d t \\
    &= \sum_{n=0}^{N-1}\sum_{j\in\mathbb{Z}} |W_j^n-c| \int_{x_{j-1/2}}^{x_{j+1/2}} \phi(t_{n+1},x)-\phi(t_n,x) \dx + \psi_c(W_j^n) \int_{t_n}^{t_{n+1}} \phi(t,x_{j+1/2})-\phi(t,x_{j-1/2})\dt,
\end{align*}
and
\begin{align*}
    \int_{\mathbb{R}}|W_{\eps,h}(T,x)-c| \phi(T,x) \, \mathrm  d x = \sum_{j\in\mathbb{Z}} |W_j^N-c| \int_{x_{j-1/2}}^{x_{j+1/2}} \phi(t_N,x) \dx, \\
    \int_{\mathbb{R}}|W_{\eps,h}(0,x)-c| \phi(0,x) \, \mathrm  d x = \sum_{j\in\mathbb{Z}} |W_j^0-c| \int_{x_{j-1/2}}^{x_{j+1/2}} \phi(t_0,x) \dx.
\end{align*}
Using summation by parts, we obtain
\begin{align*}
    &\sum_{n=0}^{N-1}\sum_{j\in\mathbb{Z}} |W_j^n-c| \int_{x_{j-1/2}}^{x_{j+1/2}} \phi(t_{n+1},x)-\phi(t_n,x) \dx  \\
    &= \sum_{n=0}^{N-1} \sum_{j\in\mathbb{Z}} \left( |W_j^n-c| - |W_j^{n+1}-c| \right) \int_{x_{j-1/2}}^{x_{j+1/2}} \phi(t_{n+1},x) \dx \\
    &\qquad + \sum_{j\in\mathbb{Z}} |W_j^0-c| \int_{x_{j-1/2}}^{x_{j+1/2}} \phi(t_0,x) \dx - \sum_{j\in\mathbb{Z}} |W_j^N-c| \int_{x_{j-1/2}}^{x_{j+1/2}} \phi(t_N,x) \dx.
\end{align*}
and
\begin{align*}
    &\sum_{n=0}^{N-1}\sum_{j\in\mathbb{Z}} \psi_c(W_j^n) \int_{t_n}^{t_{n+1}} \phi(t,x_{j+1/2})-\phi(t,x_{j-1/2})\dt \\
    &= \sum_{n=0}^{N-1}\sum_{j\in\mathbb{Z}} \left( \psi_c(W_j^n)-\psi_c(W_{j+1}^n) \right) \int_{t_n}^{t_{n+1}} \phi(t,x_{j+1/2}) \dt.
\end{align*}
Therefore, we deduce that
\begin{align*}
    \Lambda_T(W_{\eps,h}, \phi, c) &= \sum_{n=0}^{N-1} \sum_{j\in\mathbb{Z}} \left( |W_j^n-c| - |W_j^{n+1}-c| \right) \int_{x_{j-1/2}}^{x_{j+1/2}} \phi(t_{n+1},x) \dx \\
    &\qquad + \sum_{n=0}^{N-1}\sum_{j\in\mathbb{Z}} \left( \psi_c(W_j^n)-\psi_c(W_{j+1}^n) \right) \int_{t_n}^{t_{n+1}} \phi(t,x_{j+1/2}) \dt.
\end{align*}
Next, we use the discrete entropy inequality \crefrange{eq:discrete-entropy-2}{eq:entropy-flux-2} in \cref{lm:discrete-entropy-2} to obtain
\begin{align*}
    \Lambda_T(W_{\eps,h}, \phi, c) &\geq \sum_{n=0}^{N-1} \sum_{j\in\mathbb{Z}} \lambda (\Psi_{j+1/2}^n - \Psi_{j-1/2}^n) \int_{x_{j-1/2}}^{x_{j+1/2}} \phi(t_{n+1},x) \dx \\
    &\quad + \sum_{n=0}^{N-1}\sum_{j\in\mathbb{Z}} \left( \psi_c(W_j^n)-\psi_c(W_{j+1}^n) \right) \int_{t_n}^{t_{n+1}} \phi(t,x_{j+1/2}) \dt \\
    &= \sum_{n=0}^{N-1} \sum_{j\in\mathbb{Z}} \left( \Psi_{j-1/2}^n - \psi_c(W_j^n) \right) \int_{t_n}^{t_{n+1}} \phi(t,x_{j-1/2}) - \phi(t,x_{j+1/2}) \dt \\
    &\quad + \sum_{n=0}^{N-1}\sum_{j\in\mathbb{Z}} \left( \Psi_{j+1/2}^n - \Psi_{j-1/2}^n \right) \left( \lambda \int_{x_{j-1/2}}^{x_{j+1/2}} \phi(t_{n+1},x) \dx - \int_{t_n}^{t_{n+1}} \phi(t,x_{j+1/2}) \dt \right),
\end{align*}
where we have used summation by parts in the last step, and $\Psi_{j-1/2}^n$ is as defined in \cref{eq:entropy-flux-2}.

Now we take the particular test function $\Omega(t,t',x,x') = \omega_{\delta_0}(t-t')\omega_\delta(x-x')$
and estimate
\begin{align}\label{eq:estm_LamT}
    &\Lambda_T(W_{\eps,h}, \Omega(\cdot,t',\cdot,x'), c) \\
    &\geq -\sum_{n=0}^{N-1} \sum_{j\in\mathbb{Z}} \left| \Psi_{j-1/2}^n - \psi_c(W_j^n) \right| I_j^n(t',x') - \sum_{n=0}^{N-1}\sum_{j\in\mathbb{Z}} \left| \Psi_{j+1/2}^n - \Psi_{j-1/2}^n \right| J_j^n(t',x'), \notag
\end{align}
where
\begin{align*}
    I_j^n(t',x') &\coloneqq \left| \omega_\delta(x_{j-1/2}-x') - \omega_\delta(x_{j+1/2}-x') \right| \int_{t_n}^{t_{n+1}} \omega_{\delta_0}(t-t') \dt, \\
    J_j^n(t',x') &\coloneqq \left| \lambda \omega_{\delta_0}(t_{n+1}-t') \int_{x_{j-1/2}}^{x_{j+1/2}} \omega_\delta(x-x') \dx - \omega_\delta(x_{j+1/2}-x') \int_{t_n}^{t_{n+1}} \omega_{\delta_0}(t-t') \dt \right|.
\end{align*}
To proceed, we need to derive estimates on $\left|\Psi_{j-1/2}^n - \psi_c(W_j^n)\right|$ and $\left|\Psi_{j+1/2}^n - \Psi_{j-1/2}^n\right|$.

For $\left|\Psi_{j-1/2}^n - \psi_c(W_j^n)\right|$, we observe that $\psi_c(W_j^n) = |W_j^n-c|V(c) - W_j^n\left|V(W_j^n)-V(c)\right|$,
thus
\begin{align*}
    \left|\Psi_{j-1/2}^n - \psi_c(W_j^n)\right| &\leq |W_j^n-W_{j-1}^n| V(c) + |W_j^n-W_{j-1}^n|\left|V(W_j^n)-V(c)\right| \\
    &\qquad + \sum_{k=0}^\infty \gamma_k^{\eps,h} \rho_{j+k-1}^n \left|V(W_{j+k}^n)-V(W_j^n)\right| \\
    &\leq 2\norm{V}_{\mathrm{L}^{\infty}} |W_j^n-W_{j-1}^n| + \|V'\|_{\mathrm{L}^\infty} \sum_{k=0}^\infty \gamma_k^{\eps,h} |W_{j+k}^n-W_j^n| \\
    &\eqqcolon K_j^n.
\end{align*}
For $\left|\Psi_{j+1/2}^n - \Psi_{j-1/2}^n\right|$, we have
\begin{align*}
    \left|\Psi_{j+1/2}^n - \Psi_{j-1/2}^n\right| &\leq |W_j^n-W_{j-1}^n| V(c) + |W_j^n-W_{j-1}^n|\left|V(W_j^n)-V(c)\right| \\
    &\quad + \Bigg| \sum_{k=0}^\infty \gamma_k^{\eps,h} \rho_{j+k-1}^n \left(\left|V(W_{j+k}^n)-V(c)\right|-\left|V(W_j^n)-V(c)\right|\right)  \\
    &\qquad\qquad - \sum_{k=0}^\infty \gamma_k^{\eps,h} \rho_{j+k}^n \left(\left|V(W_{j+k+1}^n)-V(c)\right|-\left|V(W_j^n)-V(c)\right|\right) \Bigg| \\
    &\leq 2\norm{V}_{\mathrm{L}^{\infty}} |W_j^n-W_{j-1}^n| \\
    &\quad + \Bigg| \sum_{k=0}^\infty \left(\gamma_{k+1}^{\eps,h}-\gamma_k^{\eps,h}\right) \rho_{j+k}^n \left(\left|V(W_{j+k+1}^n)-V(c)\right|-\left|V(W_j^n)-V(c)\right|\right) \Bigg| \\
    &\leq 2\norm{V}_{\mathrm{L}^{\infty}} |W_j^n-W_{j-1}^n| + \sum_{k=0}^\infty \left(\gamma_k^{\eps,h}-\gamma_{k+1}^{\eps,h}\right) \rho_{j+k}^n \left|V(W_{j+k+1}^n)-V(W_j^n)\right| \\
    &\eqqcolon L_j^n.
\end{align*}
Noting that $K_j^n$ and $L_j^n$ are independent of the choice of $c$, we take $c=\rho(t',x')$ in $\Lambda_T$ and integrate $\Lambda_T(W_{\eps,h}, \Omega(\cdot,t',\cdot,x'),\rho(t',x'))$ with respect to $t',x'$. Then the estimate \cref{eq:estm_LamT} yields
\begin{align*}
    &\int_0^T\int_{\mathbb{R}} \Lambda_T(W_{\eps,h}, \Omega(\cdot,t',\cdot,x'),\rho(t',x')) \dx'\dt' \\
    &\geq -\sum_{n=0}^{N-1} \sum_{j\in\mathbb{Z}} K_j^n \int_0^T\int_{\mathbb{R}} I_j^n(t',x') \dx'\dt' - \sum_{n=0}^{N-1} \sum_{j\in\mathbb{Z}} L_j^n \int_0^T\int_{\mathbb{R}} J_j^n(t',x') \dx'\dt',
\end{align*}
where we have
\begin{align*}
    \int_0^T\int_{\mathbb{R}} I_j^n(t',x') \dx'\dt' &\leq \tau\int_{\mathbb{R}} \left| \omega_\delta(x_{j-1/2}-x') - \omega_\delta(x_{j+1/2}-x') \right| \dx' \\
    &\leq \mathrm{TV}(\omega_\delta)h\tau \leq \frac{2h\tau}{\delta},
\end{align*}
and
\begin{align*}
    \int_0^T\int_{\mathbb{R}} J_j^n(t',x') \dx'\dt' &\leq \lambda \int_{x_{j-1/2}}^{x_{j+1/2}} \int_{\mathbb{R}} \left| \omega_\delta(x-x') - \omega_\delta(x_{j+1/2}-x') \right| \dx' \dx \\
    &\quad + \int_{t_n}^{t_{n+1}} \int_0^T \left| \omega_{\delta_0}(t-t') - \omega_{\delta_0}(t_{n+1}-t') \right| \dt' \dt \\
    &\leq \lambda\, \mathrm{TV}(\omega_\delta) \int_{x_{j-1/2}}^{x_{j+1/2}} |x-x_{j+1/2}| \dx + \mathrm{TV}(\omega_{\delta_0}) \int_{t_n}^{t_{n+1}} |t-t_{n+1}| \dt \\
    &\leq \frac{h\tau}{\delta} + \frac{\tau^2}{\delta_0},
\end{align*}
implying that
\begin{align*}
    \int_0^T\int_{\mathbb{R}} \Lambda_T(W_{\eps,h}, \Omega(\cdot,t',\cdot,x'),\rho(t',x')) \dx'\dt' \geq -\frac{2h\tau}{\delta} \sum_{n=0}^{N-1} \sum_{j\in\mathbb{Z}} K_j^n - \left( \frac{h\tau}{\delta} + \frac{\tau^2}{\delta_0} \right) \sum_{n=0}^{N-1} \sum_{j\in\mathbb{Z}} L_j^n.
\end{align*}
We then estimate
\begin{align*}
    \sum_{j\in\mathbb{Z}} K_j^n &= 2\norm{V}_{\mathrm{L}^{\infty}} \sum_{j\in\mathbb{Z}} |W_j^n-W_{j-1}^n| + \|V'\|_{\mathrm{L}^\infty} \sum_{k=0}^\infty \gamma_k^{\eps,h} \sum_{j\in\mathbb{Z}} |W_{j+k}^n-W_j^n| \\
    &= \left( 2\norm{V}_{\mathrm{L}^{\infty}} + \|V'\|_{\mathrm{L}^\infty} \sum_{k=0}^\infty k\gamma_k^{\eps,h} \right) \sum_{j\in\mathbb{Z}} |W_{j+1}^n-W_j^n| \\
    &\leq \left( 2\norm{V}_{\mathrm{L}^{\infty}} + c_\gamma \|V'\|_{\mathrm{L}^\infty} \frac{\eps}h \right) \mathrm{TV}(\rho_0),
\end{align*}
with $c_\gamma$ specified in \cref{ass:gamma-mom}, and
\begin{align*}
    \sum_{j\in\mathbb{Z}} L_j^n &= 2\norm{V}_{\mathrm{L}^{\infty}} \sum_{j\in\mathbb{Z}} |W_j^n-W_{j-1}^n| + \sum_{k=0}^\infty \left(\gamma_k^{\eps,h}-\gamma_{k+1}^{\eps,h}\right) \sum_{j\in\mathbb{Z}} \rho_{j+k}^n \left|V(W_{j+k+1}^n) -V(W_j^n)\right|.
\end{align*}
Using \cref{eq:TV-p1} in \cref{lm:use_for_tv}, we derive
\begin{align*}
    &\sum_{k=0}^\infty (\gamma_k^{\eps,h}-\gamma_{k+1}^{\eps,h}) \sum_{j\in\mathbb{Z}} \rho_{j+k}^n \left| V(W_{j+k+1}^n) - V(W_j^n) \right| \\
    &\leq \sum_{j\in\mathbb{Z}} \left( \sum_{k=0}^\infty \left(\gamma_k^{\eps,h}-\gamma_{k+1}^{\eps,h}\right) \left(\sum_{l=0}^k \rho_{j+l}^n\right) \right) \left|V(W_{j+1}^n) -V(W_j^n)\right| \\
    &= \sum_{j\in\mathbb{Z}} W_j^n \left|V(W_{j+1}^n) -V(W_j^n)\right| \\
    &\leq \|V'\|_{\mathrm{L}^\infty} \sum_{j\in\mathbb{Z}} |W_{j+1}^n-W_j^n|,
\end{align*}
and thus
\begin{align*}
    \sum_{j\in\mathbb{Z}} L_j^n \leq \left( 2\norm{V}_{\mathrm{L}^{\infty}} + \|V'\|_{\mathrm{L}^\infty} \right) \sum_{j\in\mathbb{Z}} |W_{j+1}^n-W_j^n| \leq \left( 2\norm{V}_{\mathrm{L}^{\infty}} + \|V'\|_{\mathrm{L}^\infty} \right) \mathrm{TV}(\rho_0) .
\end{align*}
Finally, we deduce the following estimate on $\Lambda_{\delta,\delta_0}$:
\begin{align*}
    \Lambda_{\delta,\delta_0}(W_{\eps,h},\rho) &= \int_0^T\int_{\mathbb{R}} \Lambda_T(W_{\eps,h}, \Omega(\cdot,t',\cdot,x'),\rho(t',x')) \dx'\dt' \\
    &\geq -\frac{2h\tau}{\delta} N \left( 2\norm{V}_{\mathrm{L}^{\infty}} + c_\gamma \|V'\|_{\mathrm{L}^\infty} \frac{\eps}h \right) \mathrm{TV}(\rho_0) \\
    &\qquad - \left( \frac{h\tau}{\delta} + \frac{\tau^2}{\delta_0} \right) N \left( 2\norm{V}_{\mathrm{L}^{\infty}} + \|V'\|_{\mathrm{L}^\infty} \right) \mathrm{TV}(\rho_0) \\
    &= - T \left( \left(\frac{6h}{\delta}+\frac{2\tau}{\delta_0}\right) \norm{V}_{\mathrm{L}^{\infty}} + \left(\frac{2c_\gamma\eps}{\delta}+\frac{h}{\delta}+\frac{\tau}{\delta_0}\right)\|V'\|_{\mathrm{L}^\infty}\right) \mathrm{TV}(\rho_0).
\end{align*}

\uline{Step 4.} Putting together the above estimates, we have
\begin{align*}
&\|W_{\eps,h}(T-,\cdot)-\rho(T,\cdot)\|_{\mathrm{L}^1} \\
&\leq  (c_\gamma\eps+h) \mathrm{TV}(\rho_0) + \left(2 \delta+\delta_0\|V\|_{\mathrm{W}^{1,\infty}}\right) \mathrm{TV}(\rho_0) + (\delta_0+\tau) \norm{V}_{\mathrm{W}^{1,\infty}} \mathrm{TV}(\rho_0) \\
&\qquad + T \left( \left(\frac{6h}{\delta}+\frac{2\tau}{\delta_0}\right) \norm{V}_{\mathrm{L}^{\infty}} + \left(\frac{2c_\gamma\eps}{\delta}+\frac{h}{\delta}+\frac{\tau}{\delta_0}\right)\|V'\|_{\mathrm{L}^\infty}\right) \mathrm{TV}(\rho_0).
\end{align*}
By taking appropriate $\delta$ and $\delta_0$ to optimize the right-hand side, we conclude that
\begin{align}\label{eq:conv_rate}
    \|W_{\eps,h}(T-,\cdot)-\rho(T,\cdot)\|_{\mathrm{L}^1} \leq K\left( \eps+h+\sqrt{\eps\, T}+\sqrt{h\,T} \right)\, \mathrm{TV}(\rho_0),
\end{align}
for all $\eps,h>0$ and $T=N\tau>0$, where the constant $K>0$ only depends on $\lambda$, $\|V\|_{\mathrm{L}^\infty}$, $\|V'\|_{\mathrm{L}^\infty}$, and $c_\gamma$.

For an arbitrary time $T>0$, let $N\in\N_+$ be chosen such that $N\tau\le T<(N+1)\tau$. Then, by the construction of $W_{\eps,h}$ in \cref{eq:W_piece_recon}, we have $W_{\eps,h}(T,\cdot)=W_{\eps,h}(N\tau,\cdot)$. A triangle inequality yields
\begin{align*}
\|W_{\eps,h}(T,\cdot)-\rho(T,\cdot)\|_{\mathrm{L}^1} \leq& \|W_{\eps,h}(N\tau-,\cdot)-\rho(N\tau,\cdot)\|_{\mathrm{L}^1} \\
&\qquad + \|W_{\eps,h}(N\tau,\cdot)-W_{\eps,h}(N\tau-,\cdot)\|_{\mathrm{L}^1} + \|\rho(T,\cdot)-\rho(N\tau,\cdot)\|_{\mathrm{L}^1},
\end{align*}
where the first term on the right-hand side is bounded by $K\left( \eps+h+\sqrt{\eps\, T}+\sqrt{h\,T} \right)\, \mathrm{TV}(\rho_0)$ using \cref{eq:conv_rate}, the second term by $\tau\cdot\norm{V}_{\mathrm{W}^{1,\infty}} \mathrm{TV}(\rho_0)$ using the temporal TV-estimate \cref{eq:TVD-t-general} in \cref{lm:TVD-t-general}, and the last term also by $\tau\cdot\norm{V}_{\mathrm{W}^{1,\infty}} \mathrm{TV}(\rho_0)$ using a temporal TV-estimate of $\rho$ (see, e.\,g., \cite[Lemma 7.5]{MR4789921}), leading to the desired conclusion.

\end{proof}

\begin{remark}[Comparison with the literature]
\label{rem:entr_lit}
Let us offer some additional comments comparing our version of Kuznetsov's argument to the ones in  \cite{MR4700412,MR4553943}.

First, we stress that \cite{MR4700412} deals with 
\begin{align}\label{eq:cle_general}
\begin{cases}
\partial_t \rho_\eps(t,x)  + \partial_x\Big(V\big(W_{\eps}[\rho_\eps](t,x))\,f(\rho_\eps(t,x))\Big)   	= 0,	& (t,x)\in  \R_+ \times \R, \\
\rho_\eps(0,x) = \rho_0(x),	&  x\in\R,
\end{cases}
\end{align}
while both \cite{MR4553943} and this work focus on a special case of \cref{eq:cle_general} where $f(\xi) \coloneqq \xi$, giving \cref{eq:cle}.

There are subtle differences between \cref{eq:cle_general} and the special case \cref{eq:cle}: when $f$ is a nonlinear function, \cref{eq:cle_general} is a nonlinear conservation law even if $V\big(W_{\eps}[\rho_\eps])$ is replaced by a given velocity field. As a result, weak solutions of \cref{eq:cle_general} are non-unique in general; then \cite{MR4700412} specifies the following entropy condition to select a physically meaningful one:
\begin{align}\label{eq:entropy_d1}
\begin{aligned}
&\int_0^\infty\int_{\R} \eta(\rho_\eps)\partial_t\phi + \psi(\rho_\eps) V\big(W_\eps[\rho_\eps]\big) \partial_x\phi \,\dx\,\dt \\
& \qquad -\int_0^\infty\int_{\R}  \big( \eta'(\rho_\eps)f(\rho_\eps) - \psi(\rho_\eps) \big) \partial_x V\big(W_\eps[\rho_\eps]\big) \phi \,\dx\,\dt + \int_\R \eta(\rho_0(x))\phi(0,x)\,\mathrm dx \geq0, \end{aligned}
\end{align}
where $(\eta,\psi)$ is an entropy-entropy flux pair with $\psi'(\xi)=\eta'(\xi)f'(\xi)$, and the convergence rate estimate of numerical schemes is via Kuznetsov-type arguments based on \cref{eq:entropy_d1}.

In contrast, the weak solutions of \cref{eq:cle} are unique (as shown in \cite{zbMATH06756308}) and there is no need for entropy conditions and \cref{eq:entropy_d1} is automatically satisfied by the unique weak solution. In \cite{MR4553943} and this work, the concern is about showing that the local entropy condition holds for the local limit of nonlocal solutions of \cref{eq:cle}, i.\,e., 
\begin{align}\label{eq:entropy_d2}
    \int_0^\infty\int_{\R}  \eta(\rho^\ast)\partial_t\phi + \psi(\rho^\ast)\partial_x\phi \,\dx\,\dt + \int_{\R} \eta(\rho_0(x))\phi(0,x)\,\mathrm dx \geq0,
\end{align}
where $(\eta,\psi)$ is an entropy-entropy flux pair with $\psi'(\xi) = \eta'(\xi)\,(V(\xi)\,\xi)'$.

In the case where $f(\xi)\coloneqq \xi$, the entropy conditions \cref{eq:entropy_d1} and \cref{eq:entropy_d2} are incompatible: one cannot directly pass a limit from \cref{eq:entropy_d1} to \cref{eq:entropy_d2} because $\partial_x V\big(W_\eps[\rho_\eps]\big)$ becomes singular as $\eps\searrow0$. 
\end{remark}

\subsection{Completion of the proof}
\label{ssec:conclusion-of-proof}

Putting together the lemmas in the previous subsections, we are ready to complete the proof of \cref{th:main-2}.

\begin{proof}[Proof of \cref{th:main-2}]
As shown in \cref{lm:limit}, the approximate solution $W_{\varepsilon,h}$ is strongly pre-compact in $\mathrm{L}^1_{\mathrm{loc}}$, thus having a limit point as $\eps,h\searrow0$. The fact that the limit point is the (unique) entropy solution of \cref{eq:cl} is proven in \cref{lem:entropy_admissibility-2}. As a consequence of the uniqueness of entropy solutions, we deduce (owing to Uryshon's subsequence principle) that the whole family (not just up to subsequences) converges. Finally, the convergence rate estimate is proven in \cref{th:rate}.
\end{proof}

\section{The exponential kernel case}
\label{sec:exp}

For the exponential kernel $\gamma=\mathds{1}_{]-\infty,0]}(\cdot)\exp(\cdot)$, the exact quadrature weights from \cref{ex:exact} and the normalized Riemann quadrature weights from \cref{ex:Riemann} coincide, with the common expression given by
\begin{align}\label{eq:quad}
    \gamma_k^{\eps,h} = \int_{-(k+1)h}^{-kh} \frac1\eps e^{\frac{z}{\eps}} \dd z = e^{-\frac{kh}\eps} \left( 1 - e^{-\frac{h}\eps} \right),
\end{align}
yielding 
\begin{align*}
    \gamma_0^{\eps,h} = 1 - e^{-\frac{h}\eps}, \quad \gamma_{k+1}^{\eps,h} = e^{-\frac{h}\eps} \gamma_k^{\eps,h} \qquad \text{for}\ k\geq0.
\end{align*}
With these quadrature weights, it follows from \cref{eq:num_nonlocal_W} that
\begin{align} \label{eq:exp-form1a}
    \frac{W_{j+1}^n - W_j^n}{h} &= \frac{e^{\frac{h}\eps} - 1}{h} (W_j^n - \rho_j^n) \\
    &\Big(\approx \frac1\eps (W_j^n - \rho_j^n) \quad \text{ provided that $h\ll\eps$}\Big), \notag 
\end{align}
which is analogous to \cref{eq:WWxq}.

In general, we may assume that the quadrature weights $\{\gamma_k^{\eps,h}\}_{k\ge0}$ form a geometric sequence
\begin{align}\label{eq:geometric-gamma}
    \gamma_k^{\eps,h} = \gamma_0^{\eps,h}\big(1-\gamma_0^{\eps,h}\big)^k \qquad \text{for}\ k\geq0,
\end{align}
where $\gamma_0^{\eps,h}$ satisfies
\begin{align}\label{eq:geometric-gamma-assm}
    0 < \gamma_0^{\eps,h} <1 \quad \text{and} \quad \frac{1-\gamma_0^{\eps,h}}{\gamma_0^{\eps,h}} \leq c\,\frac{\eps}{h} \qquad \text{for all }  \eps,h>0,
\end{align}
and the constant $c>0$ is independent of $\eps,h$.
The quadrature weights specified by \crefrange{eq:geometric-gamma}{eq:geometric-gamma-assm} satisfy \crefrange{ass:pos-mono}{ass:gamma-mom}; the condition $\frac{1-\gamma_0^{\eps,h}}{\gamma_0^{\eps,h}} \leq c\,\frac{\eps}{h}$ in \cref{eq:geometric-gamma-assm} is exactly the condition $\sum_{k=0}^{\infty} k \gamma_k^{\eps,h} \leq c_\gamma \frac{\eps}h$ in \cref{ass:gamma-mom}. Specifically, the exact quadrature weights \cref{eq:quad} satisfy \crefrange{eq:geometric-gamma}{eq:geometric-gamma-assm} with the constant $c=1$.

We supplement our numerical scheme \crefrange{eq:godunov}{eq:ini_cond_discrete} with the quadrature weights specified by \crefrange{eq:geometric-gamma}{eq:geometric-gamma-assm}.
Then \cref{eq:num_nonlocal_W} implies the following identity that relates $\rho_j^n$ to $W_j^n$ and $W_{j+1}^n$:
\begin{align} \label{eq:exp-form3}
    W_{j+1}^n - W_j^n = \frac{\gamma_0^{\eps,h}}{1-\gamma_0^{\eps,h}} (W_j^n-\rho_j^n),
\end{align}
which gives
\begin{align}\label{eq:exp-form4}
    \rho_j^n = \frac1{\gamma_0^{\eps,h}}W_j^n - \frac{1-\gamma_0^{\eps,h}}{\gamma_0^{\eps,h}}W_{j+1}^n.
\end{align}
Using \crefrange{eq:godunov}{eq:num_nonlocal_W} and \cref{eq:exp-form4}, we obtain  
\begin{align*}
    W_j^{n+1} &= \sum_{k=0}^\infty \gamma_k^{\eps,h} \rho_{j+k}^{n+1} \\
    &= W_j^n + \lambda \sum_{k=0}^\infty \gamma_k^{\eps,h} \left( \rho_{j+k-1}^n V(W_{j+k}^n) -\rho_{j+k}^n V(W_{j+k+1}^n) \right) \\
    &= W_j^n + \lambda \sum_{k=0}^\infty \big(1-\gamma_0^{\eps,h}\big)^{k+1} \left( W_{j+k+1}^n V(W_{j+k+1}^n) - W_{j+k}^n V(W_{j+k}^n) \right) \\
    &\qquad\qquad - \lambda \sum_{k=0}^\infty \big(1-\gamma_0^{\eps,h}\big)^k \left( W_{j+k}^n V(W_{j+k+1}^n) - W_{j+k-1}^n V(W_{j+k}^n) \right).
\end{align*}
Then, summation by parts yields
\begin{align*}
    &\sum_{k=0}^\infty \!\big(1-\gamma_0^{\eps,h}\big)^{k+1} \!\left( W_{j+k+1}^n V(W_{j+k+1}^n) - W_{j+k}^n V(W_{j+k}^n) \right) = \sum_{k=0}^\infty \gamma_k^{\eps,h} W_{j+k}^n V(W_{j+k}^n) - W_j^n V(W_j^n), \\
    &\sum_{k=0}^\infty \!\big(1-\gamma_0^{\eps,h}\big)^k \!\left( W_{j+k}^n V(W_{j+k+1}^n) - W_{j+k-1}^n V(W_{j+k}^n) \right) = \sum_{k=0}^\infty \gamma_k^{\eps,h} W_{j+k}^n V(W_{j+k+1}^n) - W_{j-1}^nV(W_j^n),
\end{align*}
which lead to
\begin{align}\label{eq:W-1}
 W_j^{n+1} &= W_j^n + \lambda \left( W_{j-1}^n - W_j^n \right) V(W_j^n)  - \lambda \sum_{k=0}^\infty \gamma_k^{\eps,h}  W_{j+k}^n \left( V(W_{j+k+1}^n) - V(W_{j+k}^n) \right).
\end{align}
We observe that the particular structure of the exponential kernel and its suitable discretization, as specified by \crefrange{eq:geometric-gamma}{eq:geometric-gamma-assm}, enable the derivation of the scheme \cref{eq:W-1} for $\{W_j^n\}_{j\in\mathbb{Z}}^{n\geq0}$. This scheme serves as a numerical discretization of \cref{eq:Wevo} for $W_\varepsilon$ and is conservative.

Owing to \cref{eq:W-1}, we can prove the TVD property more simply compared to \cref{lm:TVD-x-2}. 

\begin{lemma}[TVD in space]\label{lm:TVD-x}
Let us assume that \crefrange{ass:ic}{ass:V} hold, the quadrature weights satisfy \crefrange{eq:geometric-gamma}{eq:geometric-gamma-assm}, and the CFL condition \cref{ass:CFL-max} holds. 
Let $\{\rho_j^n\}_{j\in\mathbb{Z}}^{n\geq0}$ and $\{W_j^n\}_{j\in\mathbb{Z}}^{n\geq0}$ be the numerical solutions constructed with the numerical scheme \crefrange{eq:godunov}{eq:ini_cond_discrete}. Then the spatial $\mathrm{TV}$-estimate \cref{eq:TVD-x-2} holds.
\end{lemma}

\begin{proof}

Let us introduce the notation $\alpha_j^n \coloneqq \frac{V(W_{j+1}^n) - V(W_j^n)}{W_{j+1}^n - W_j^n}$ for $j\in\mathbb{Z},\,n\geq0$, and rewrite \cref{eq:W-1} as 
\begin{align*}
    W_j^{n+1} &= W_j^n + \lambda \left( W_{j-1}^n  - W_j^n \right) V(W_j^n) - \lambda \sum_{k=0}^\infty \gamma_k^{\eps,h} \alpha_{j+k}^n W_{j+k}^n \left( W_{j+k+1}^n - W_{j+k}^n \right).
\end{align*}
Then a straightforward calculation gives
\begin{align*}
    W^{n+1}_{j+1} - W^{n+1}_j =& \lambda V(W_j^n) (W_j^n-W_{j-1}^n) + \left( 1 - \lambda V(W_{j+1}^n) + \lambda \gamma_0\alpha_j^nW_j^n \right) (W_{j+1}^n-W_j^n) \\
    &\qquad + \lambda \sum_{k=1}^\infty \big(\gamma_k^{\eps,h} - \gamma_{k-1}^{\eps,h}\big) \alpha_{j+k}^n W_{j+k}^n (W_{j+k+1}^n - W_{j+k}^n),
\end{align*}
which yields
\begin{align*}
    \sum_{j \in \mathbb Z} |W^{n+1}_{j+1} - W^{n+1}_j| &\leq \sum_{j \in \mathbb Z} \Big( \lambda V(W_{j+1}^n) + 1 - \lambda V(W_{j+1}^n) +  \lambda \gamma_0\alpha_j^nW_j^n  \\ & \qquad\qquad + \lambda \sum_{k=1}^\infty \big(\gamma_k^{\eps,h}  - \gamma_{k-1}^{\eps,h}\big) \alpha_j^nW_j^n \Big) |W^n_{j+1} - W^n_j| \\
    &= \sum_{j \in \mathbb Z} |W^n_{j+1} - W^n_j|,
\end{align*}
where we have used the CFL condition \cref{ass:CFL-max} to ensure $1 - \lambda V(W_{j+1}^n) + \lambda \gamma_0\alpha_j^nW_j^n \geq0$.
Finally, by induction, we have 
\begin{align*}
    \sum_{j \in \mathbb Z} |W^n_{j+1} - W^n_j| \leq  \sum_{j \in \mathbb Z}  |W_{j+1}^0-W_j^0| \leq \sum_{j \in \mathbb Z} |\rho^0_{j+1} - \rho^0_j| \leq \mathrm{TV}(\rho_0).
\end{align*}
\end{proof}

Owing to \cref{eq:exp-form3}, we can estimate the $\mathrm{L}^1$-distance between $\rho_j^n$ and $W_j^n$.

\begin{lemma}[$\mathrm{L}^1$-deviation estimate between $\rho_j^n$ and $W_j^n$]\label{lm:limit-same}
Let us suppose that the quadrature weights satisfy \crefrange{eq:geometric-gamma}{eq:geometric-gamma-assm}. Let $\{\rho_j^n\}_{j\in\mathbb{Z}}^{n\geq0}$ and $\{W_j^n\}_{j\in\mathbb{Z}}^{n\geq0}$ be the numerical solutions constructed with the numerical scheme \crefrange{eq:godunov}{eq:ini_cond_discrete}. Then the following $\mathrm{L}^1$-estimate holds:
    \begin{align}\label{eq:L1_W_rho_num}
        \sum_{j \in \mathbb Z} |W_j^n-\rho_j^n| h \leq c\,\eps  \sum_{j \in \mathbb Z} |W^n_{j+1} - W^n_j|,
    \end{align}
    where the constant $c>0$ is as in \cref{eq:geometric-gamma-assm}, independent of $\eps,h$.
\end{lemma}

\begin{proof}
    The inequality \cref{eq:L1_W_rho_num} follows directly from \cref{eq:geometric-gamma-assm} and \cref{eq:exp-form3}.
\end{proof}

As a consequence of \cref{lm:limit-same}, we can state the counterpart of \cref{th:main-2} for the exponential kernel, which also addresses the question of the convergence of $\rho_{\eps,h}$ as $\eps,h\searrow 0$.

\begin{theorem}[Convergence (exponential kernel)]\label{th:main-1}
 Let us assume that \crefrange{ass:ic}{ass:V} hold, the quadrature weights satisfy \crefrange{eq:geometric-gamma}{eq:geometric-gamma-assm}, and the CFL condition \cref{ass:CFL-max} holds with the CFL ratio $\lambda\coloneqq \tau/h$ fixed.
 Let us consider the numerical solutions $\{\rho_j^n\}_{j\in\mathbb{Z}}^{n\geq0}$ and $\{W_j^n\}_{j\in\mathbb{Z}}^{n\geq0}$ constructed with the numerical scheme \crefrange{eq:godunov}{eq:ini_cond_discrete}, and let $\rho_{\eps,h}$ and $W_{\eps,h}$ be the piecewise constant reconstructions of $\{\rho_j^n\}_{j\in\mathbb{Z}}^{n\geq0}$ and $\{W_j^n\}_{j\in\mathbb{Z}}^{n\geq0}$, respectively, i.\,e., 
\begin{align*}
     \rho_{\eps,h} &\coloneqq  \sum_{n=0}^\infty \sum_{j\in\mathbb{Z}} \rho_j^n \cdot \mathds{1}_{[n\tau,(n+1)\tau[\times[(j-\frac12)h,(j+\frac12)h[}, \\
    W_{\eps,h} &\coloneqq \sum_{n=0}^\infty \sum_{j\in\mathbb{Z}} W_j^n \cdot \mathds{1}_{[n\tau,(n+1)\tau[\times[(j-\frac12)h,(j+\frac12)h[}.
\end{align*}
Then, as $\eps,h\searrow 0$, both approximate solutions $W_{\varepsilon, h}$ and $\rho_{\varepsilon, h}$ converge strongly in $\mathrm{L}^1_{\mathrm{loc}}$ to the unique entropy solution $\rho$ of \cref{eq:cl}.
Moreover, the following error estimates hold: 
\begin{align}
\|W_{\eps,h}(t,\cdot)-\rho(t,\cdot)\|_{\mathrm{L}^1} &\leq K\left( \eps+h+\sqrt{\eps\, t}+\sqrt{h\,t} \right) \, \mathrm{TV}(\rho_0), \quad \textnormal{for every } \eps,h > 0, \, t > 0, \label{eq:conv-W-exp} \\
\|\rho_{\eps,h}(t,\cdot)-\rho(t,\cdot)\|_{\mathrm{L}^1} &\leq K\left( \eps+h+\sqrt{\eps\, t}+\sqrt{h\,t} \right)\, \mathrm{TV}(\rho_0), \quad \textnormal{for every } \eps,h > 0, \, t > 0, \label{eq:conv-rho}
\end{align}
where the constant $K>0$ only depends on $\lambda$, $\|V\|_{\mathrm{L}^\infty}$, $\|V'\|_{\mathrm{L}^\infty}$, and $c_\gamma$ (as in \cref{eq:geometric-gamma-assm}).
\end{theorem}

\begin{proof}
The assumptions used in \cref{th:main-2} are satisfied. In particular, the conditions \crefrange{eq:geometric-gamma}{eq:geometric-gamma-assm} on quadrature weights imply \crefrange{ass:pos-mono}{ass:gamma-mom}.

We observe that the only additional claims in \cref{th:main-1} compared to \cref{th:main-2} concern the convergence of $\rho_{\varepsilon,h}$ and the convergence rate estimate in \cref{eq:conv-rho}. These follow directly from the conclusions of \cref{th:main-2} (namely, from \cref{eq:conv-W-exp}), thanks to \cref{lm:limit-same}, which holds due to the relation \cref{eq:exp-form3}. Indeed, for every $\varepsilon,h > 0$ and $t > 0$, we compute
\begin{align*}
\|\rho_{\varepsilon,h}(t,\cdot) - \rho(t,\cdot) \|_{\mathrm{L}^1}
&\le \|W_{\varepsilon,h}(t,\cdot) - \rho(t,\cdot) \|_{\mathrm{L}^1}
+ \| W_{\varepsilon,h}(t,\cdot) - \rho_{\varepsilon,h}(t,\cdot) \|_{\mathrm{L}^1} \\
&\le K\left( \varepsilon + h + \sqrt{\varepsilon\, t} + \sqrt{h\, t} \right) \mathrm{TV}(\rho_0)
+ c\varepsilon\, \mathrm{TV}(\rho_0),
\end{align*}
where we used the triangle inequality, \cref{eq:conv-W-exp}, \cref{lm:limit-same}, and the spatial TV-estimate in \cref{lm:TVD-x}.
\end{proof}

\begin{remark}[CFL conditions]\label{rk:CFL}
    The maximum principle in \cref{lem:maxm_prcp} is a prerequisite for all results in this work, and so is the CFL condition \cref{ass:CFL-max}. The spatial TV-estimate in \cref{lm:TVD-x-2} needs a stronger CFL condition stated in \cref{ass:CFL}, which is thus used in the statement of \cref{th:main-2}. 
    On the other hand, in the specific case of the exponential kernel, the CFL condition \cref{ass:CFL-max} is sufficient for the same spatial TV-estimate in \cref{lm:TVD-x}, and thus \cref{th:main-1} holds.
\end{remark}

\begin{remark}[Alternative quadrature weights] 
Alternatively, we can discretize \cref{eq:WWxq} as
\begin{align*}
    \frac{W_{j+1}^n - W_j^n}{h} = \frac1\eps (W_j^n - \rho_j^n).
\end{align*}
This discretization is equivalent to using the quadrature weights
\begin{align*}
    \gamma_k^{\eps,h} \coloneqq \frac{h}{\eps} \left( 1-\frac{h}{\eps} \right)^k,
\end{align*}
which satisfy \crefrange{eq:geometric-gamma}{eq:geometric-gamma-assm} when $h<\eps$.
The conclusion of \cref{th:main-1} then holds, provided that the spatial mesh size $h$ does not exceed the nonlocal horizon $\varepsilon$. Hence, we achieve conditional asymptotic compatibility.
\end{remark}

\subsection{Consistency with the entropy admissibility condition for \texorpdfstring{$\rho_{\eps,h}$}{rho\_eps,h}}
\label{ssec:entropy-consistency}

As demonstrated in \cref{th:main-1}, both $\rho_{\eps,h}$ and $W_{\eps,h}$ converge to the unique entropy solution of \cref{eq:cl}, with the entropy admissibility for the limit of $W_{\eps,h}$ proved using the discrete entropy inequality \crefrange{eq:discrete-entropy-2}{eq:entropy-flux-2} in \cref{lm:discrete-entropy-2}. Yet, it is useful to give a direct proof of the entropy admissibility for the limit of $\rho_{\eps,h}$, which aligns with a continuous entropy inequality for $\rho_\eps$, adapted from \cite{MR4110434,MR4283539,MR4651679} and tailored to the exponential kernel:
\begin{align}\label{eq:entropy-exp}
\begin{aligned}
    \partial_t \eta(\rho_\varepsilon) + \partial_x \left( \psi(\rho_\varepsilon) + \eta(\rho_\varepsilon) (V(W_\varepsilon) - V(\rho_\varepsilon)) + Q(W_\varepsilon) - Q(\rho_\varepsilon) \right) \le 0,
   \end{aligned}
\end{align}
where $(\eta,\psi)$ is an entropy-entropy flux pair, and $Q$ satisfies $Q'(\xi)\coloneqq P(\xi)V'(\xi)$ with $P(\xi) \coloneqq \xi\,\eta'(\xi) - \eta(\xi)$.

We outline below the key (formal) computation to obtain \cref{eq:entropy-exp}:
\begin{align*}
&\partial_t \eta(\rho_\varepsilon) + \partial_x \psi(\rho_\varepsilon)\\
&= \eta'(\rho_\varepsilon) \partial_x \big( \rho_\varepsilon (V(\rho_\varepsilon) - V(W_\varepsilon)) \big) \\
&= \eta'(\rho_\varepsilon) \partial_x \rho_\varepsilon (V(\rho_\varepsilon) - V(W_\varepsilon)) + \eta'(\rho_\varepsilon) \rho_\varepsilon \partial_x (V(\rho_\varepsilon) - V(W_\varepsilon)) \\
&= \partial_x \eta(\rho_\varepsilon) (V(\rho_\varepsilon) - V(W_\varepsilon)) + \eta'(\rho_\varepsilon) \rho_\varepsilon \partial_x (V(\rho_\varepsilon) - V(W_\varepsilon)) \\
&= \partial_x \big( \eta(\rho_\varepsilon)(V(\rho_\varepsilon) - V(W_\varepsilon)) \big) + \big( \eta'(\rho_\varepsilon) \rho_\varepsilon - \eta(\rho_\varepsilon) \big) \partial_x (V(\rho_\varepsilon) - V(W_\varepsilon)) \\
\intertext{(using \( P(\xi) = \xi\,\eta'(\xi) - \eta(\xi) \))}
&= \partial_x \big( \eta(\rho_\varepsilon)(V(\rho_\varepsilon) - V(W_\varepsilon)) \big) 
+ \big( P(\rho_\varepsilon) - P(W_\varepsilon) \big) \partial_x (V(\rho_\varepsilon) - V(W_\varepsilon)) \\
&\qquad + P(W_\varepsilon) \partial_x (V(\rho_\varepsilon) - V(W_\varepsilon)) \\
\intertext{(using \( Q'(\xi) = P(\xi) V'(\xi) \)):}
&= \partial_x \big( \eta(\rho_\varepsilon)(V(\rho_\varepsilon) - V(W_\varepsilon)) \big) 
+ \partial_x \big(Q(\rho_\varepsilon) - Q(W_\varepsilon) \big)
+ V'(W_\varepsilon) \big( P(W_\varepsilon) - P(\rho_\varepsilon) \big) \partial_x W_\varepsilon \\
&= \partial_x \big( \eta(\rho_\varepsilon)(V(\rho_\varepsilon) - V(W_\varepsilon)) + Q(\rho_\varepsilon) - Q(W_\varepsilon) \big)
+ \frac{1}{\varepsilon} V'(W_\varepsilon) \big( P(W_\varepsilon) - P(\rho_\varepsilon) \big)(W_\varepsilon - \rho_\varepsilon),
\end{align*}
where the last term in the last line is non-positive because $V'\leq0$ and $P'(\xi)=\xi\,\eta''(\xi)\ge0$ when $\xi\geq0$, yielding \cref{eq:entropy-exp}.

Taking Kru\v{z}kov's entropy function $\eta_c(\xi)\coloneqq |\xi-c|$, for any constant $c\in\R$, we have $\psi_c(\xi)=\sgn{\xi-c}\cdot(\xi V(\xi)-cV(c))$, $P_c(\xi)=c\cdot\sgn{\xi-c}$, and $Q_c(\xi)=-c|V(\xi)-V(c)|$. In this case, we can rewrite \cref{eq:entropy-exp} as 
\begin{align}\label{eq:entr-k}
    \partial_t |\rho_\varepsilon-c| + \partial_x \left( |\rho_\varepsilon-c|V(W_\varepsilon) - c|V(W_\varepsilon)-V(c)| \right) \leq 0.
\end{align}
Therefore, we can define $\Psi_c(\rho,W)\coloneqq |\rho-c|V(W) - c|V(W)-V(c)|$ as a \emph{nonlocal} entropy flux function, which satisfies $\Psi_c(\rho,\rho)=\psi_c(\rho)$, i.\,e., it is compatible with the local entropy flux function. 

With these observations in mind, in the following lemma, we prove a discrete entropy inequality that aligns with \cref{eq:entr-k}.

\begin{lemma}[Discrete entropy inequality] \label{lm:entro-disc-2}  
Let us assume that \crefrange{ass:ic}{ass:V} hold, the quadrature weights satisfy \crefrange{eq:geometric-gamma}{eq:geometric-gamma-assm}, and the CFL condition \cref{ass:CFL-max} holds. 
Let $\{\rho_j^n\}_{j\in\mathbb{Z}}^{n\geq0}$ and $\{W_j^n\}_{j\in\mathbb{Z}}^{n\geq0}$ be the numerical solutions constructed with the numerical scheme \crefrange{eq:godunov}{eq:ini_cond_discrete}. Then the following discrete entropy inequality holds: 
\begin{align}\label{eq:discrete-entropy}
    \frac{|\rho_j^{n+1}-c|-|\rho_j^n-c|}{\tau} + \frac{\Psi_c(\rho_j^n, W_{j+1}^n) - \Psi_c(\rho_{j-1}^n, W_j^n)}{h} \leq 0,
\end{align}
for all $c \in \R$ and  $\Psi_c(\rho,W)\coloneqq |\rho-c|V(W) - c|V(W)-V(c)|$.
\end{lemma}

\begin{proof}
Noting that $0\leq\rho_j^n,W_j^n\leq1$ for all $j,n$ (from \cref{lem:maxm_prcp} and \cref{rk:W}), it suffices to show \cref{eq:discrete-entropy} for $c\in(0,1)$.
From \cref{eq:godunov}, we write
\begin{align*}
\rho_j^{n+1} = \rho_j^n+\lambda \left( (\rho_{j-1}^n-c) V(W_j^n) - (\rho_j^n-c) V(W_{j+1}^n) + c \left(V(W_j^n)-V(W_{j+1}^n)\right) \right),
\end{align*}
which implies
\begin{align*}
|\rho_j^{n+1}-c| &\leq \lambda V(W_j^n) |\rho_{j-1}^n-c|  + \left| \left(1-\lambda V(W_{j+1}^n)\right) (\rho_j^n-c)  + \lambda c \left(V(W_j^n)-V(W_{j+1}^n)\right) \right|.
\end{align*}
We denote $R\coloneqq (1-\lambda V(W_{j+1}^n)) (\rho_j^n-c)  + \lambda c (V(W_j^n)-V(W_{j+1}^n))$ and discuss the following cases. 

\begin{enumerate}[label={\uline{Case \arabic*.}}]
    \item \label{it:c1} If 
$W_{j+1}^n>W_j^n>c$, we have
\[ \left|V(W_j^n)-V(W_{j+1}^n)\right| = V(W_j^n)-V(W_{j+1}^n) = \left|V(W_{j+1}^n)-V(c)\right| - \left|V(W_j^n)-V(c)\right|, \]
yielding
\begin{align}\label{eq:R_estm}
    |R| \leq \left(1-\lambda V(W_{j+1}^n)\right) |\rho_j^n-c| + \lambda c \left(\left|V(W_{j+1}^n)-V(c)\right| - \left|V(W_j^n)-V(c)\right|\right).
\end{align}
\item \label{it:c2} If $W_{j+1}^n<W_j^n \leq c$, we have
\[ \left|V(W_j^n)-V(W_{j+1}^n)\right| = V(W_{j+1}^n) - V(W_j^n) = \left|V(W_{j+1}^n)-V(c)\right| - \left|V(W_j^n)-V(c)\right|, \]
which also yields \cref{eq:R_estm}. 
\item \label{it:c3} If $W_j^n>c$ and $W_{j+1}^n\leq W_j^n$,  the identity \cref{eq:exp-form3} implies $\rho_j^n\geq W_j^n>c$, and consequently
\[ \left|V(W_j^n)-V(c)\right| \leq \left|V(\rho_j^n)-V(c)\right| \leq \|V'\|_{\mathrm{L}^\infty} |\rho_j^n-c|. \]
Let $R_1\coloneqq \left(1-\lambda V(W_{j+1}^n)\right) (\rho_j^n-c)  + \lambda c \left(V(W_j^n)-V(c)\right)$, using the CFL condition \cref{ass:CFL-max}, we have
\begin{align*}
    \lambda c\left|V(W_j^n)-V(c)\right| \leq \left(1-\lambda V(W_{j+1}^n)\right) |\rho_j^n-c|,
\end{align*}
thus $\sgn{R_1}=\sgn{\rho_j^n-c}=1$ and
\begin{align*}
    |R| &\leq R_1  + \lambda c \left|V(W_{j+1}^n)-V(c)\right| \\
    &= \left(1-\lambda V(W_{j+1}^n)\right) (\rho_j^n-c)  + \lambda c \left(V(W_j^n)-V(c)\right) + \lambda c \left|V(W_{j+1}^n)-V(c)\right| \\
    &= \left(1-\lambda V(W_{j+1}^n)\right) |\rho_j^n-c| - \lambda c \left|V(W_j^n)-V(c)\right| + \lambda c \left|V(W_{j+1}^n)-V(c)\right|; 
\end{align*}
so we still obtain \cref{eq:R_estm}.

\item \label{it:c4} If $W_j^n \leq c$ and $W_{j+1}^n\geq W_j^n$, by similar arguments as in \cref{it:c3}, we obtain $\rho_j^n\leq W_j^n \leq c$ and $\sgn{R_1}=\sgn{\rho_j^n-c}=-1$, giving
\begin{align*}
    |R| &\leq -R_1  + \lambda c \left|V(W_{j+1}^n)-V(c)\right| \\
    &= -\left(1-\lambda V(W_{j+1}^n)\right) (\rho_j^n-c)  -\lambda c \left(V(W_j^n)-V(c)\right) + \lambda c \left|V(W_{j+1}^n)-V(c)\right| \\
    &= \left(1-\lambda V(W_{j+1}^n)\right) |\rho_j^n-c| - \lambda c \left|V(W_j^n)-V(c)\right| + \lambda c \left|V(W_{j+1}^n)-V(c)\right|;
\end{align*}
so we still obtain \cref{eq:R_estm}.
\end{enumerate}
In summary, the estimate \cref{eq:R_estm} holds in all cases, and it gives
\begin{align*}
    |\rho_j^{n+1}-c| & \leq |\rho_j^n-c| + \lambda \Big( |\rho_{j-1}^n-c| V(W_j^n) - |\rho_j^n-c|V(W_{j+1}^n) \\ &\qquad\qquad - c \left(\left|V(W_j^n)-V(c)\right| - \left|V(W_{j+1}^n)-V(c)\right|\right) \Big) \\
    &= |\rho_j^n-c| + \lambda \big(\Psi_c(\rho_{j-1}^n, W_j^n) - \Psi_c(\rho_j^n, W_{j+1}^n)\big).
\end{align*}
Therefore, we deduce the discrete entropy inequality \cref{eq:discrete-entropy}.
\end{proof}

\subsection{\texorpdfstring{$\mathrm{L}^1$}{L-1}-contraction}
\label{ssec:L1-contr}
This subsection is devoted to the $\mathrm{L}^1$-contraction property of $W_\eps$, the nonlocal impact specified in \cref{eq:W} based on \cref{eq:cle}, motivated by its critical role in scalar conservation laws. For a scalar conservation law with a (locally) Lipschitz continuous flux function $f:\R \to \R$,
\begin{align*}
    \partial_t u +\partial_xf(u)=0,  \qquad (t,x) \in \R_+\times \R,
\end{align*}
considering Kru\v{z}kov's entropy-entropy flux pair $(\eta_c,\psi_c)$, where $\eta_c(\xi)\coloneqq|\xi-c|$ and $\psi_c(\xi)\coloneqq\sgn{\xi-c}\cdot(f(\xi)-f(c))$, which yields 
\begin{align*}
    \partial_t |u-c| + \partial_x \left( \sgn{u-c}(f(u)-f(c)) \right) \leq0,
\end{align*}
in the sense of distributions in $\R_+\times \R$.
Then by employing the ``doubling of variables''  technique, in \cite{zbMATH03341462}, Kru\v{z}kov obtained 
\begin{align*}
    \partial_t |u-\tilde{u}| + \partial_x \left( \sgn{u-\tilde{u}}(f(u)-f(\tilde{u})) \right) \leq0,
\end{align*}
for any pair of entropy solutions $u$ and $\tilde{u}$, which then gives
\begin{align*}
    \frac{\mathrm d}{\mathrm dt}\int_\mathbb{R} |u-\tilde{u}|\,\mathrm dx \leq 0.
\end{align*}
This $\mathrm{L}^1$-contraction property was also noticed by Volpert in \cite{zbMATH03269216} for BV initial data. We refer to \cite{MR2103702} for further discussions. From this property, a comparison principle and a priori bounds for the $\mathrm{L}^1$, $\mathrm{L}^{\infty}$, and $\mathrm{BV}$ norms of the solutions follow.\footnote{~Let us recall a relevant result from \cite{zbMATH01222044}: 
Let us suppose that $T: \mathrm{L}^1\left(\mathbb{R}\right) \rightarrow \mathrm{L}^1\left(\mathbb{R}\right)$ satisfies the following conditions:
\begin{enumerate}[label={\arabic*.}]
    \item \label{it:L1.1} for all $u \in \mathrm{L}^1\left(\mathbb{R}\right)$,  we have $\int_{\mathbb{R}} T u=\int_{\mathbb{R}} u$;
    \item \label{it:L1.2}  for all $u, \, v \in \mathrm{L}^1\left(\mathbb{R}\right)$, with $u \geq v$ a.\,e., we have $T u \geq T v$ a.\,e.;
    \item \label{it:L1.3} for all $h \in \mathbb{R}$ and all $u \in \mathrm{L}^1\left(\mathbb{R}\right)$, we have $T(u(\cdot-h))=(T u)(\cdot-h)$.
\end{enumerate} 
Then, for all $u \in \mathrm{L}^1\left(\mathbb{R}\right)$, the following maximum and minimum principles hold: 
$$
\esssup T u \leq \operatorname{ess} \sup u \quad \text { and } \quad \essinf T u \geq \operatorname{ess} \inf u.
$$
By Crandall--Tartar's lemma (see \cite{zbMATH03703495}), conditions \cref{it:L1.1}  and \cref{it:L1.2}  on $\mathrm{L}^1(\Omega)$ of any measure space $(\Omega, \mathrm{d} \mu)$ are equivalent to \cref{it:L1.1} and
\begin{enumerate}
    \item[2'.] for all $u,\, v \in \mathrm{L}^1(\Omega)$, we have $\|T u-T v\|_{\mathrm{L}^1(\Omega)} \leq\|u-v\|_{\mathrm{L}^1(\Omega)}$.
\end{enumerate}
Thus, on $\mathrm{L}^1(\Omega)$, non-expansive mappings that preserve the integral are the same as order-preserving mappings that preserve the integral.}

For the nonlocal conservation law \cref{eq:cle}, the violation of the TVD property for $\rho_\eps$ implies the loss of the $\mathrm{L}^1$-contraction property. In contrast, under the exponential kernel, we can establish the $\mathrm{L}^1$-contraction property for $W_\eps$, which satisfies the evolution equation \cref{eq:Wevo}, improving the existing $\mathrm{L}^1$-stability estimates for \cref{eq:cle} in the literature. We also note that, for a different nonlocal conservation law inspired by \cref{eq:cle}, an $\mathrm{L}^1$-contraction estimate was proven in \cite{coclite2025upwind}.

\begin{theorem}[$\mathrm{L}^1$-contraction for $W_\eps$]\label{th:L1.contr}
Let us assume that \cref{ass:V} holds, the nonlocal kernel $\gamma \coloneqq \mathds{1}_{]-\infty,0]}(\cdot) \exp(\cdot)$, and the initial data $\rho_0,\, \widetilde{\rho}_0$ satisfy \cref{ass:ic} with $\rho_0-\widetilde{\rho}_0\in\mathrm{L}^1(\R)$. Let $\rho_\eps, \, \widetilde{\rho}_\varepsilon$ be the solutions of \cref{eq:cle}, with the nonlocal impacts $W_\eps,\, \widetilde{W}_\varepsilon$ defined as in \cref{eq:W}.
Then the following $\mathrm{L}^1$-contraction property holds: 
\[
\|W_\varepsilon(t,\cdot) - \widetilde{W}_\varepsilon(t,\cdot)\|_{\mathrm{L}^1} \le \|W_\eps(0,\cdot) - \widetilde{W}_\eps(0,\cdot)\|_{\mathrm{L}^1}.
\]
\end{theorem}

\begin{proof}
We take functions $Q$ and $R$ that satisfy $Q'(\xi)\coloneqq\xi\, V'(\xi)$ and $R'(\xi)\coloneqq V(\xi)$.
Noting that $W_\eps,\widetilde{W}_\eps \in \mathrm{W}^{1,\infty}(\R_+\times\R)$ and using \cref{eq:Wevo}, the direct analogue of \cref{eq:cle} and \cref{eq:W}, we obtain
\begin{align*}
    &\frac{\dd}{\dd t} \int_{\mathbb{R}} \left|W_\eps(t,x)-\widetilde{W}_\eps(t,x)\right|\dx \\
    &= \int_{\mathbb{R}} \sgn{W_\eps(t,x)-\widetilde{W}_\eps(t,x)} \partial_t \left( W_\eps(t,x)-\widetilde{W}_\eps(t,x) \right) \dx \\
    &= -\int_{\mathbb{R}} \sgn{W_\eps(t,x)-\widetilde{W}_\eps(t,x)} \partial_x \left( R(W_\eps(t,x)) - R(\widetilde{W}_\eps(t,x)) \right) \dx \\
    &\qquad -\frac1\eps \int_{\mathbb{R}} \sgn{W_\eps(t,x)-\widetilde{W}_\eps(t,x)} \int_x^\infty \exp\left(\frac{x-y}{\eps}\right) \partial_y \left( Q(W_\eps(t,y)) - Q(\widetilde{W}_\eps(t,y)) \right) \dd y \dx \\
    &\eqqcolon I_1+I_2.
\end{align*}
For the first term, we have\footnote{~Here, we observe that $\operatorname{sign}(v)' = \delta_{\{v=0\}}v'$ holds (in the sense of distributions) for any smooth function $v$ (cf.~\cite[Lemma 2]{MR542510} and \cite[Remark 7.1]{MR4789921}). To apply this to our setting, we may employ approximations (by regularizing the initial data) for $W_\eps,\,\widetilde{W}_\eps$ as in \cite[Theorem 3.1]{MR4651679} or \cite[Lemma 3.1]{MR4855163}.}
\begin{align*}
    I_1 = \int_{\mathbb{R}} 2\delta_{\left\{W_\eps(t,x)-\widetilde{W}_\eps(t,x) = 0\right\}} \left( R(W_\eps(t,x)) - R(\widetilde{W}_\eps(t,x)) \right) \partial_x\left( W_\eps(t,x) - \widetilde{W}_\eps(t,x) \right) \dx = 0.
\end{align*}
For the second term, we use Fubini's theorem and integration by parts (noticing that the boundary terms vanish as $W_\eps,\,\widetilde{W}_\eps \in \mathrm{W}^{1,\infty}(\R_+\times\R)$) to deduce  
\begin{align*}
    I_2 &= -\frac1\eps \int_{\mathbb{R}} \partial_y \left( Q(W_\eps(t,y)) - Q(\widetilde{W}_\eps(t,y)) \right) \int_{-\infty}^y \exp\left(\frac{x-y}{\eps}\right) \sgn{W_\eps(t,x)-\widetilde{W}_\eps(t,x)} \dx \dd y \\
    &= \frac1\eps \int_{\mathbb{R}} \left( Q(W_\eps(t,y)) - Q(\widetilde{W}_\eps(t,y)) \right) \sgn{W_\eps(t,y)-\widetilde{W}_\eps(t,y)} \dd y \\
    &\qquad  - \frac1{\eps^2} \int_{\mathbb{R}} \left( Q(W_\eps(t,y)) - Q(\widetilde{W}_\eps(t,y)) \right) \int_{-\infty}^y \exp\left(\frac{x-y}{\eps}\right) \sgn{W_\eps(t,x)-\widetilde{W}_\eps(t,x)} \dx \dd y \\
    &\eqqcolon I_{21} + I_{22}.
\end{align*}
Noting that $Q$ is a decreasing function, we have
\begin{align*}
    I_{21} = - \frac1\eps \int_{\mathbb{R}} \left| Q(W_\eps(t,y)) - Q(\widetilde{W}_\eps(t,y)) \right| \dd y;
\end{align*}
on the other hand,
\begin{align*}
    |I_{22}| &\leq \frac1{\eps^2} \int_{\mathbb{R}} \left| Q(W_\eps(t,y)) - Q(\widetilde{W}_\eps(t,y)) \right| \int_{-\infty}^y \exp\left(\frac{x-y}{\eps}\right) \dx \dd y \\
    &= \frac1\eps \int_{\mathbb{R}} \left| Q(W_\eps(t,y)) - Q(\widetilde{W}_\eps(t,y)) \right| \dd y.
\end{align*}
Therefore, $I_2=I_{21}+I_{22}\leq0$ and we conclude that
\begin{align*}
    \frac{\dd}{\dd t} \int_{\mathbb{R}} \left|W_\eps(t,x)-\widetilde{W}_\eps(t,x)\right|\dx \leq 0.
\end{align*}
\end{proof}

\section{Numerical experiments}
\label{sec:nums}

We supplement the paper with a series of numerical experiments that illustrate our theoretical findings and suggest further conjectures on the approximate solutions produced by the numerical scheme \crefrange{eq:godunov}{eq:ini_cond_discrete}.
These experiments, organized in \cref{ssec:num-convergence} to demonstrate the asymptotically compatible convergence rates (cf.~\cref{th:main-2} and \cref{th:main-1}) and \cref{ssec:num-stability} to investigate stability properties including the TVD property (cf.~\cref{lm:TVD-x-2} and \cref{lm:TVD-x}) and the entropy condition (cf.~\cref{lm:discrete-entropy-2} and \cref{lm:entro-disc-2}), rely on the following settings.

We use three representative initial data, detailed as follows:  
\begin{itemize}
    \item \emph{Riemann shock:}
    \begin{align}\label{eq:ini_riemann_shock}
        \rho_0(x) \coloneqq 0.7\,\mathds{1}_{[0,\infty[}(x);
    \end{align}
    \item \emph{Riemann rarefaction:}
    \begin{align}\label{eq:ini_riemann_rarefaction}
        \rho_0(x) \coloneqq 0.65\,\mathds{1}_{]-\infty,0]}(x) + 0.35\,\mathds{1}_{[0,\infty[}(x);
    \end{align}
    \item \emph{Bell-shaped:} 
    \begin{align}\label{eq:ini_bellshape}
	\rho_0(x) \coloneqq 0.4 + 0.4 \, \exp\left(-100x^2\right).
    \end{align}
\end{itemize}
We also adopt the \textit{Greenshields velocity function}~\cite{greenshields1935study}, $V(\xi) \coloneqq 1 -\xi$, unless otherwise indicated.
Furthermore, we use the following nonlocal kernels:
\begin{itemize}
    \item \emph{exponential kernel:} 
    \begin{align}
        \gamma(z) &\coloneqq \displaystyle \exp(z) \, \mathds{1}_{]-\infty,0]}(z);
        \label{eq:gamma-expo} 
    \end{align}
    
    \item \emph{linear kernel:} 
    \begin{align}
        \gamma(z) &\coloneqq \displaystyle 2(z+1) \, \mathds{1}_{]-1,0]}(z);
        \label{eq:gamma-lin} 
    \end{align}
    
    \item \emph{constant kernel:} 
    \begin{align}
        \gamma(z) &\coloneqq \displaystyle \mathds{1}_{]-1,0]}(z);
        \label{eq:gamma-const} 
    \end{align}
\end{itemize}
discretized with their exact quadrature weights from \cref{ex:exact} unless otherwise specified.
Along with these, we fix the CFL ratio $\lambda=0.25$ and restrict the scheme's implementation to a finite computational domain and use constant extensions outside.

\subsection{Convergence rates}
\label{ssec:num-convergence}

In the following experiments, we examine the asymptotically compatible convergence rates of the approximate solution $W_{\eps,h}$, produced by the scheme \crefrange{eq:godunov}{eq:ini_cond_discrete}, to the entropy solution $\rho$ of \cref{eq:cl} as $\eps,h\searrow0$. We also evaluate the impact of nonlocal kernels, quadrature weights, and velocity functions on the convergence rates. Further, we investigate the convergence from $\rho_{\eps,h}$ to $\rho$.

\begin{example}
First of all, we offer a visual comparison between the approximate solution $W_{\eps,h}$ (with $\eps=5\times10^{-3}$ and $h=10^{-3}$) and the entropy solution $\rho$ of \cref{eq:cl}, using the linear kernel \cref{eq:gamma-lin} and initial data \crefrange{eq:ini_riemann_shock}{eq:ini_bellshape}. Snapshots of $W_{\eps,h}$ at times $t=0,0.5,1$ and of $\rho$ at $t=1$ (computed on a finer mesh) are shown in \cref{fig:exp_0}.
\end{example}

From \cref{fig:exp_0}, we see that $W_{\eps,h}$ closely approximates $\rho$ at $t=1$ across all initial data with the small $\eps$ and $h$, and the difference appears as a slight smoothing effect on $\rho$. Moreover, the temporal evolution of $W_{\eps,h}$ mimics the dynamics of \cref{eq:cl} that include a moving shock wave for the Riemann initial data \cref{eq:ini_riemann_shock}, a centered rarefaction wave for the Riemann initial data \cref{eq:ini_riemann_rarefaction}, and the formation of a shock for the smooth bell-shaped initial data \cref{eq:ini_bellshape}.

\begin{figure}[htbp]
\centering
	\begin{subfigure}{.32\textwidth}
	\includegraphics[width=\textwidth]{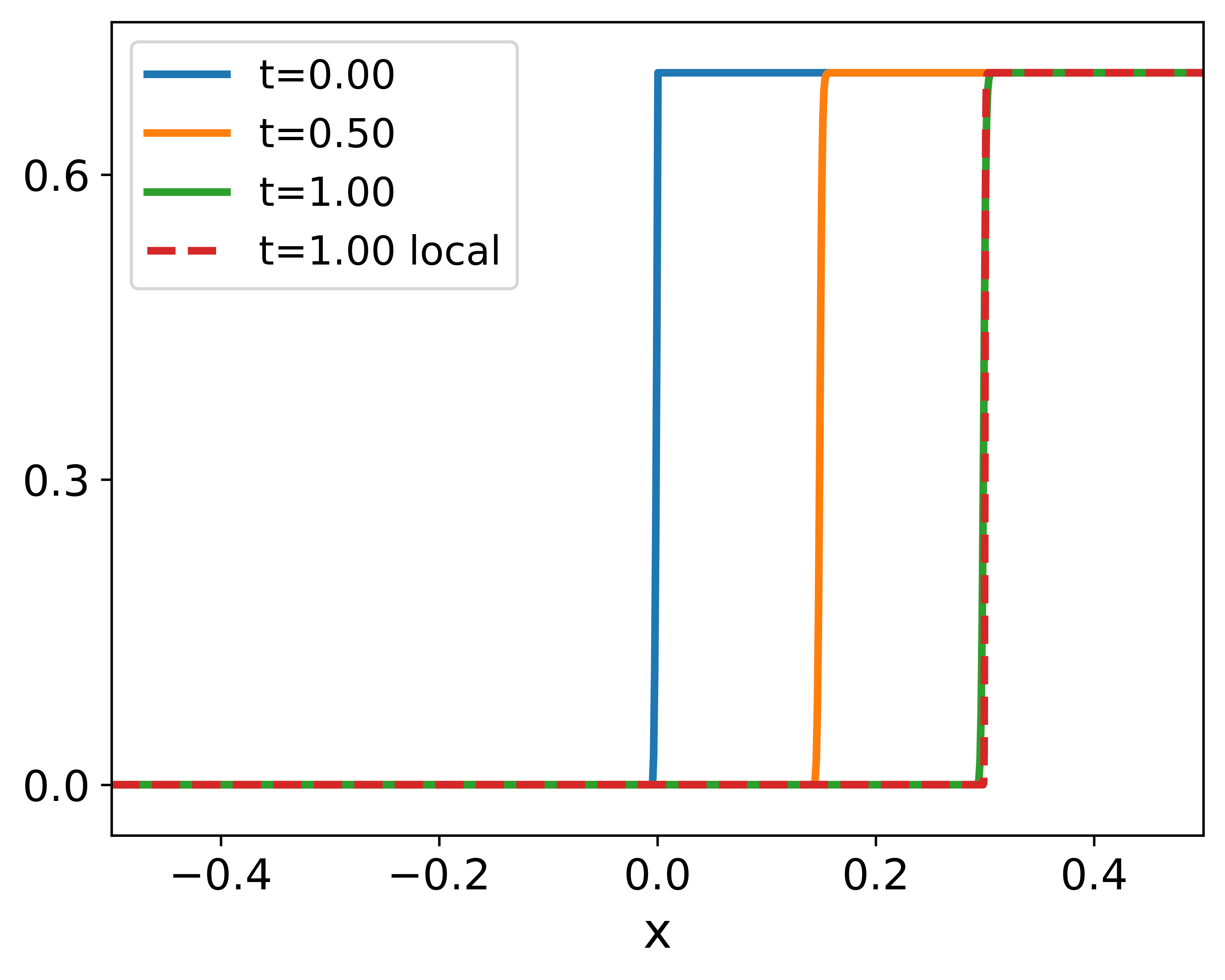}
	\end{subfigure}
	\begin{subfigure}{.32\textwidth}
	\includegraphics[width=\textwidth]{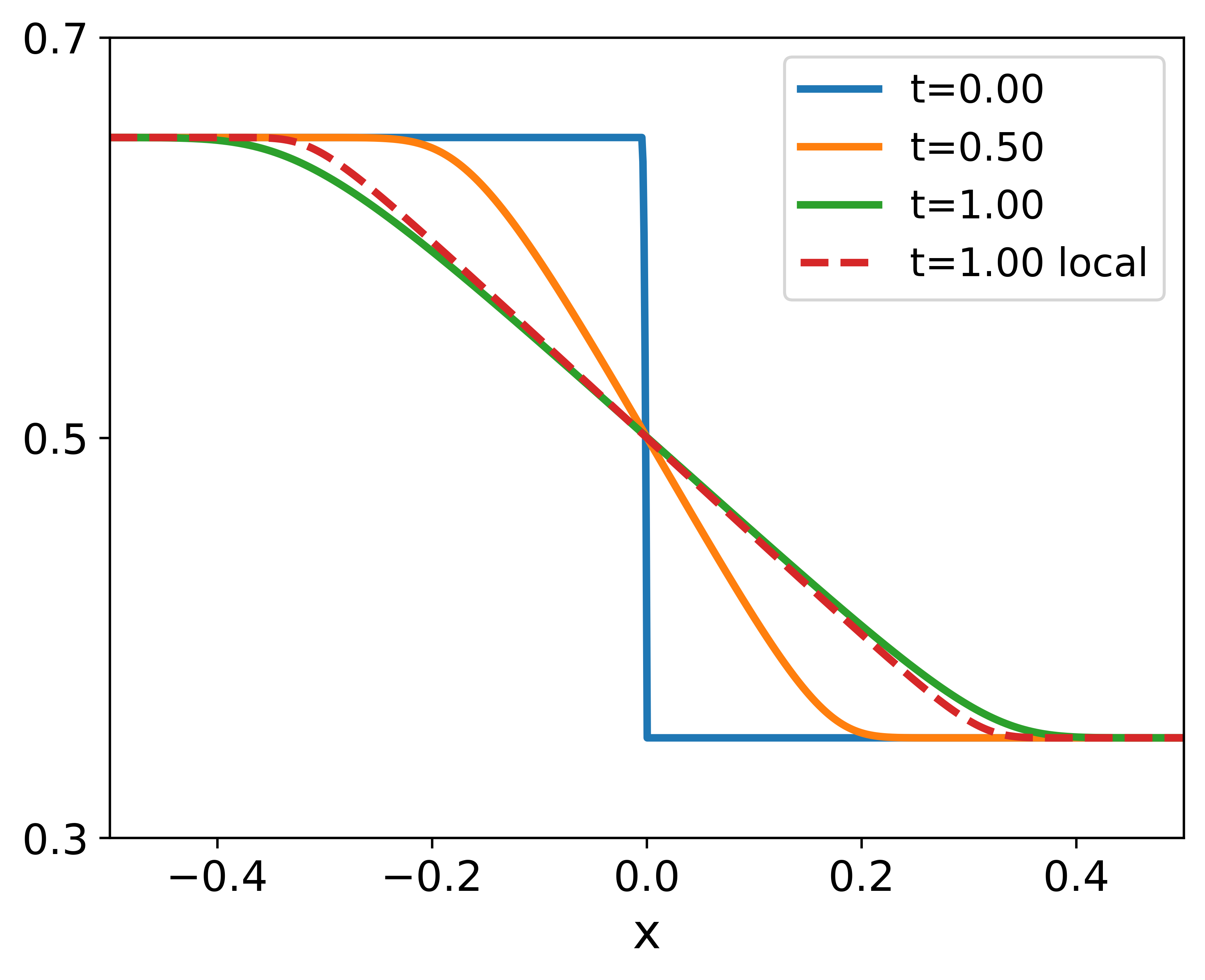}
	\end{subfigure}
    \begin{subfigure}{.32\textwidth}
	\includegraphics[width=\textwidth]{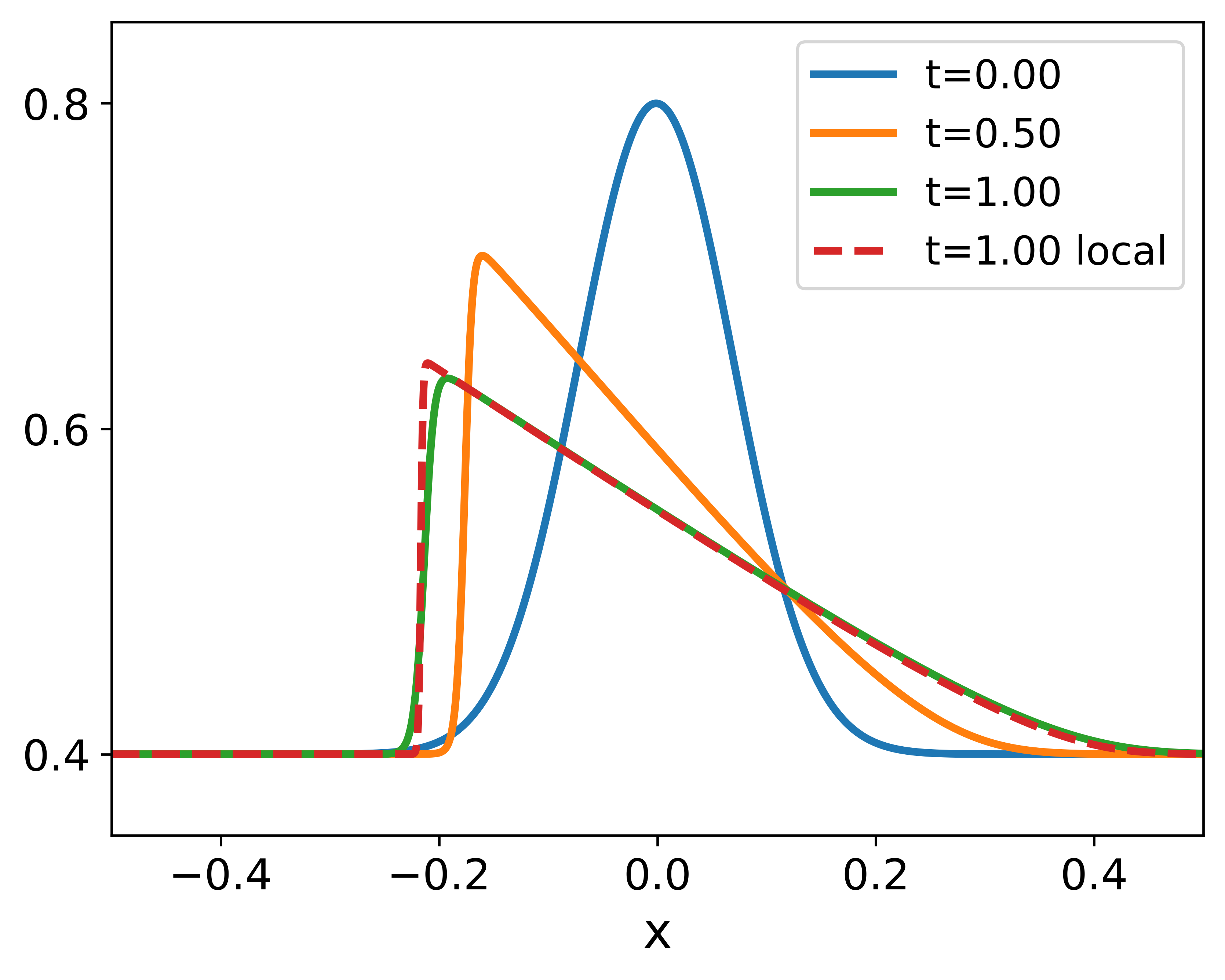}
	\end{subfigure}
	\caption{Snapshots of $W_{\eps,h}$ at times $t=0,0.5,1$ and of $\rho$ at $t=1$ for the Riemann shock initial data \cref{eq:ini_riemann_shock} (\textsc{left}), the Riemann rarefaction initial data \cref{eq:ini_riemann_rarefaction} (\textsc{middle}), and the bell-shaped initial data \cref{eq:ini_bellshape} (\textsc{right}).}
    \label{fig:exp_0}
\end{figure}

\begin{example}
Next, we examine the convergence rate from $W_{\eps,h}$ to $\rho$ using the initial data \crefrange{eq:ini_riemann_shock}{eq:ini_bellshape} and nonlocal kernels \crefrange{eq:gamma-expo}{eq:gamma-const}, and along the following limiting paths:
\begin{itemize}
    \item $\eps=h\searrow 0$: in this case, with the linear kernel \cref{eq:gamma-lin} or the constant kernel \cref{eq:gamma-const}, the scheme simplifies to a monotone three-point scheme solving \cref{eq:cl} (see \cref{rk:hfixed});
    \item $\eps=5h\searrow 0$;
    \item $\eps=\sqrt{h}\searrow 0$: in this case, $\frac{\eps}h=\frac1{\sqrt{h}}\nearrow\infty$.
\end{itemize}
In \cref{fig:exp_1}, we plot the $\mathrm{L}^1$ error $\norm{W_{\eps,h}(t,\cdot)-\rho(t,\cdot)}_{\mathrm{L}^1}$ at $t=1$ against $h^{-1}$ on a log-log scale, with $\rho$ computed on a finer mesh, and include reference lines with slopes $-1$, $-\frac12$, and $-\frac14$ to indicate expected convergence rates.

Additionally, we evaluate the impact of nonlinear velocity functions:
\begin{itemize}
    \item Krystek's velocity function~\cite{krystek1980syntetyczny}: $V(\xi)\coloneqq (1-\xi)^4$;
    \item Underwood's velocity function~\cite{underwood1961speed}: $V(\xi)\coloneqq \exp(-\xi)$;
\end{itemize}
applied to the initial data \crefrange{eq:ini_riemann_shock}{eq:ini_bellshape} with the linear kernel \cref{eq:gamma-lin}. In \cref{fig:exp_3}, we present convergence plots similar to those in \cref{fig:exp_1}.
\end{example}

The left panel of \cref{fig:exp_1} reveals that, for a moving shock wave in $\rho$, the error of $W_{\varepsilon,h}$ decays at a rate $h$ for $\varepsilon=h$ and $\varepsilon=5h$, and at a rate $\sqrt{h}$ for $\varepsilon=\sqrt{h}$, indicating a potential convergence rate of $\eps+h$, which exceeds the $\sqrt{\varepsilon}+\sqrt{h}$ estimate in \cref{th:main-2}. The middle panel shows that, for a centered rarefaction wave in $\rho$, the error decay lies between $h$ and $\sqrt{h}$ for $\varepsilon=h$ and $\varepsilon=5h$, and between $\sqrt{h}$ and $h^{\frac14}$ for $\varepsilon=\sqrt{h}$, suggesting a rate between $\varepsilon+h$ and $\sqrt{\varepsilon}+\sqrt{h}$. The right panel exhibits similar trends for a shock formed from the bell-shaped initial data. These findings align with \cref{th:main-2}, with observed convergence rates lying between $\varepsilon+h$ and $\sqrt{\varepsilon}+\sqrt{h}$. Remarkably, the non-convex constant kernel \cref{eq:gamma-const} leads to quadrature weights violating the assumption \cref{ass:convexity}, falling outside the scope of \cref{th:main-2}, yet its conclusion remains valid. The convergence reaches a rate of $h$ along $\varepsilon=h$ in shock and bell-shaped cases, but declines slightly below $h$ for a centered rarefaction wave in $\rho$, consistent with discussions in \cref{rk:hfixed} and \cref{rk:local-convergence-references}.

Moreover, from \cref{fig:exp_3} we observe that, with nonlinear velocity functions $V(\xi)=(1-\xi)^4$ and $V(\xi)=\exp(-\xi)$, the convergence rates remain consistent, reinforcing the robustness of these convergence rates in velocity models.

\begin{figure}[htbp]
\centering
    \begin{subfigure}{.32\textwidth}
	\includegraphics[width=\textwidth]{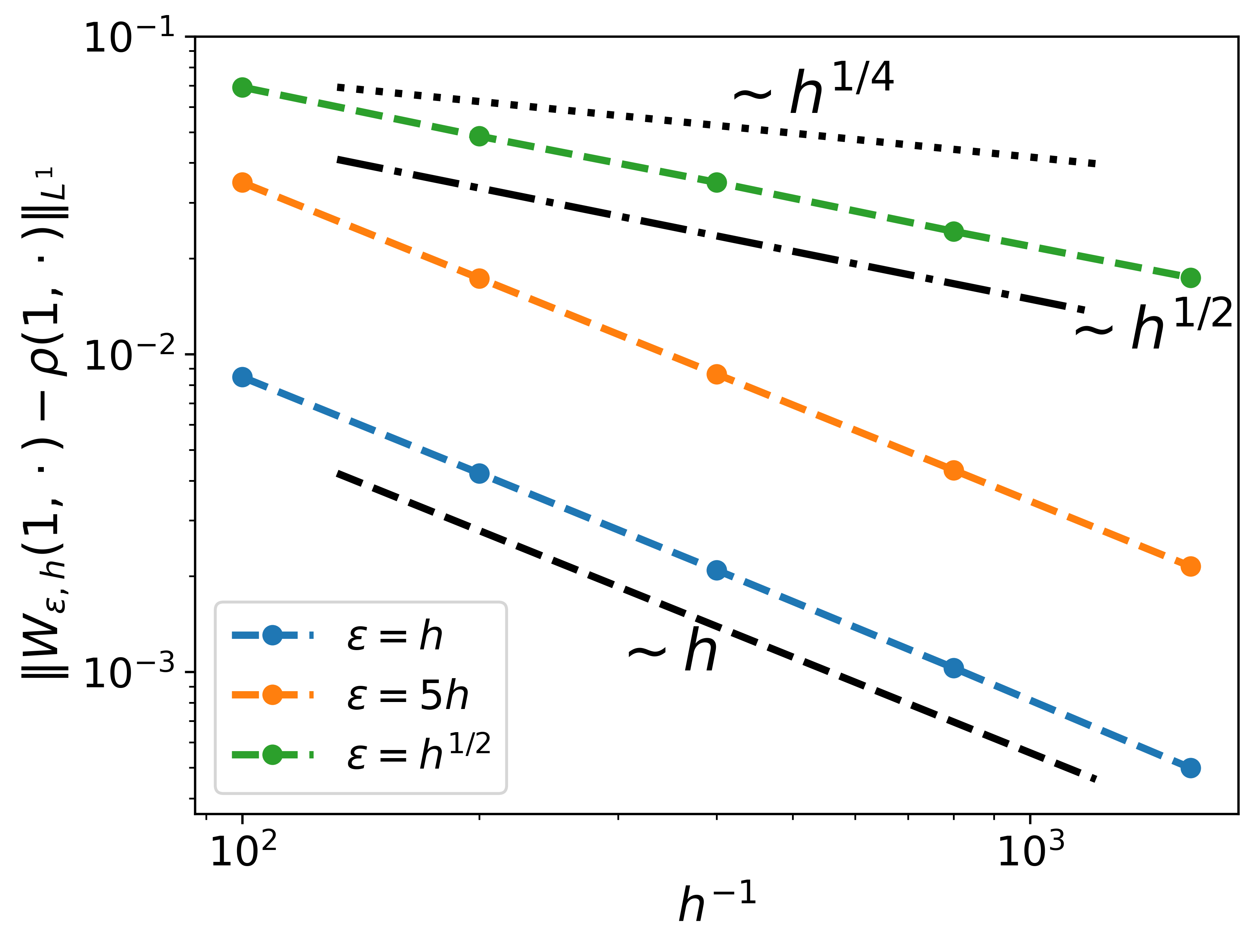}
	\end{subfigure}
	\begin{subfigure}{.32\textwidth}
	\includegraphics[width=\textwidth]{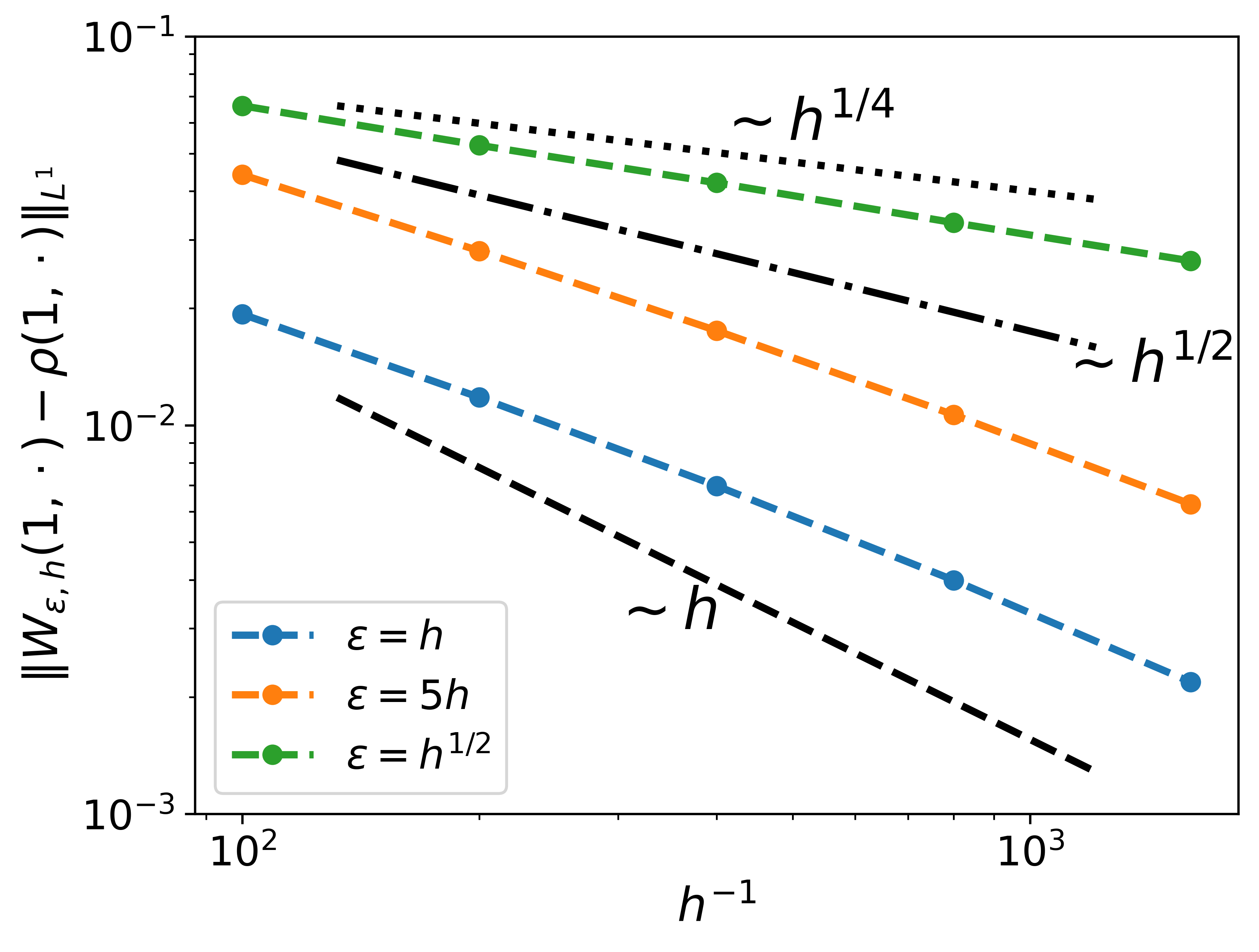}
	\end{subfigure}
    \begin{subfigure}{.32\textwidth}
	\includegraphics[width=\textwidth]{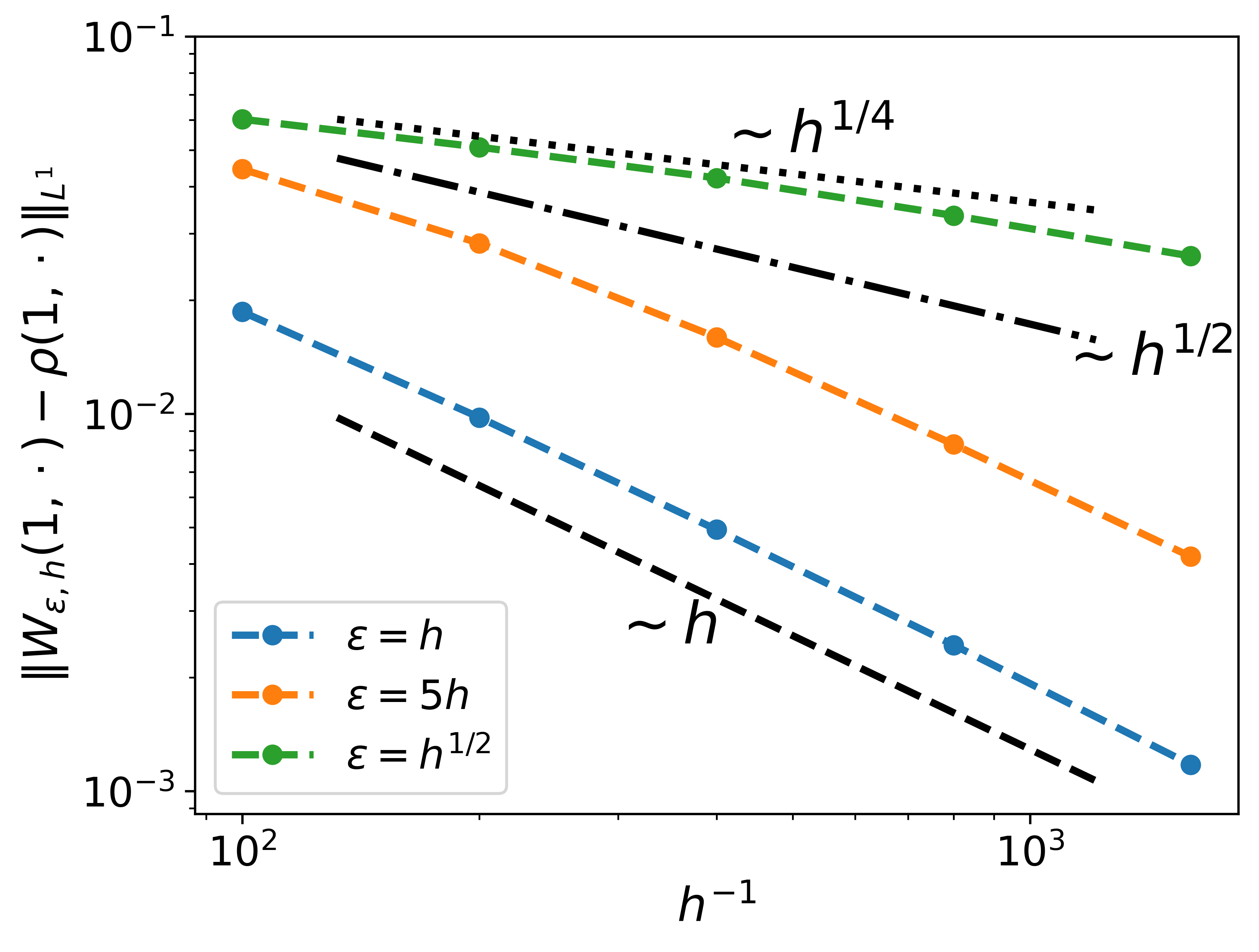}
	\end{subfigure}
	\begin{subfigure}{.32\textwidth}
	\includegraphics[width=\textwidth]{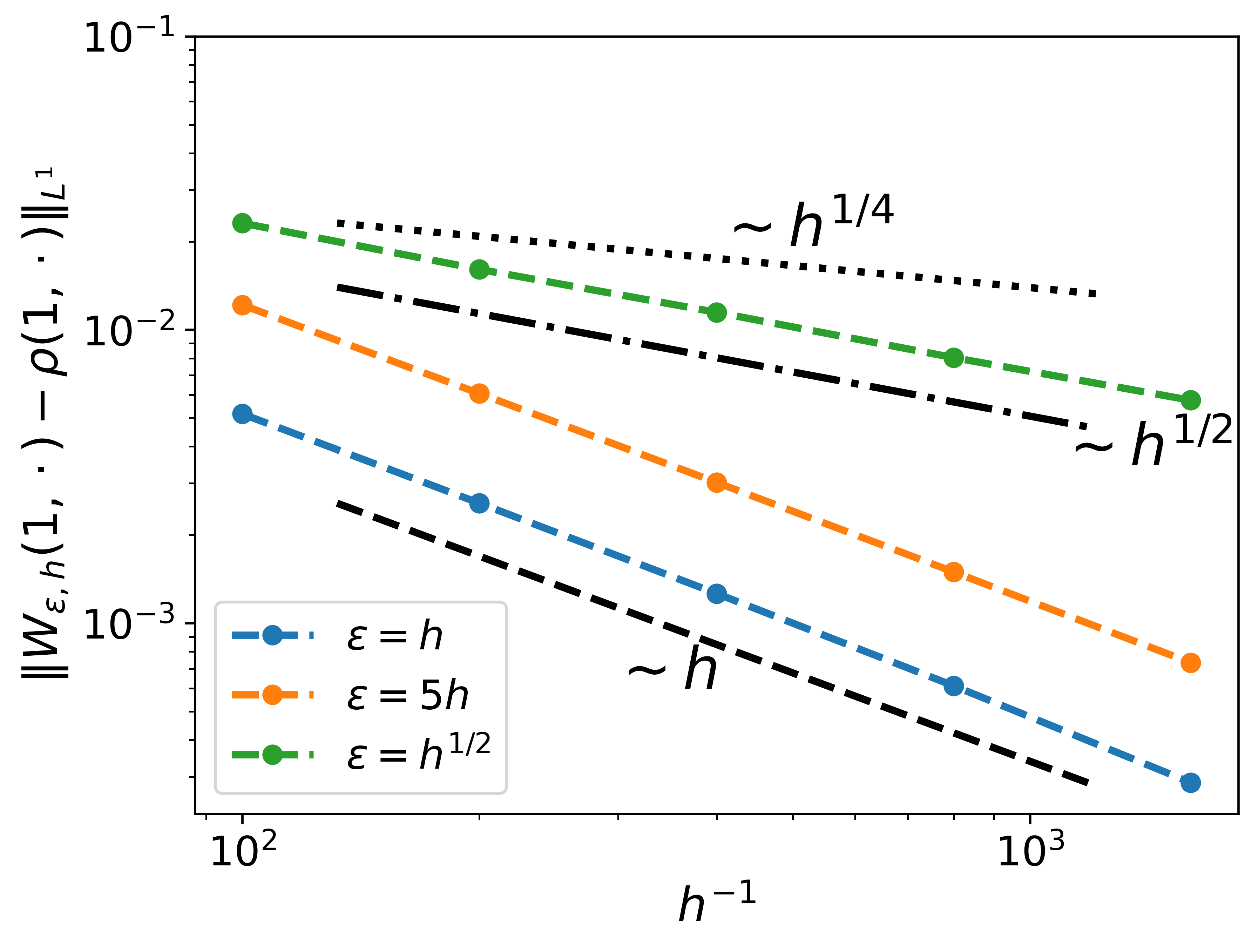}
	\end{subfigure}
	\begin{subfigure}{.32\textwidth}
	\includegraphics[width=\textwidth]{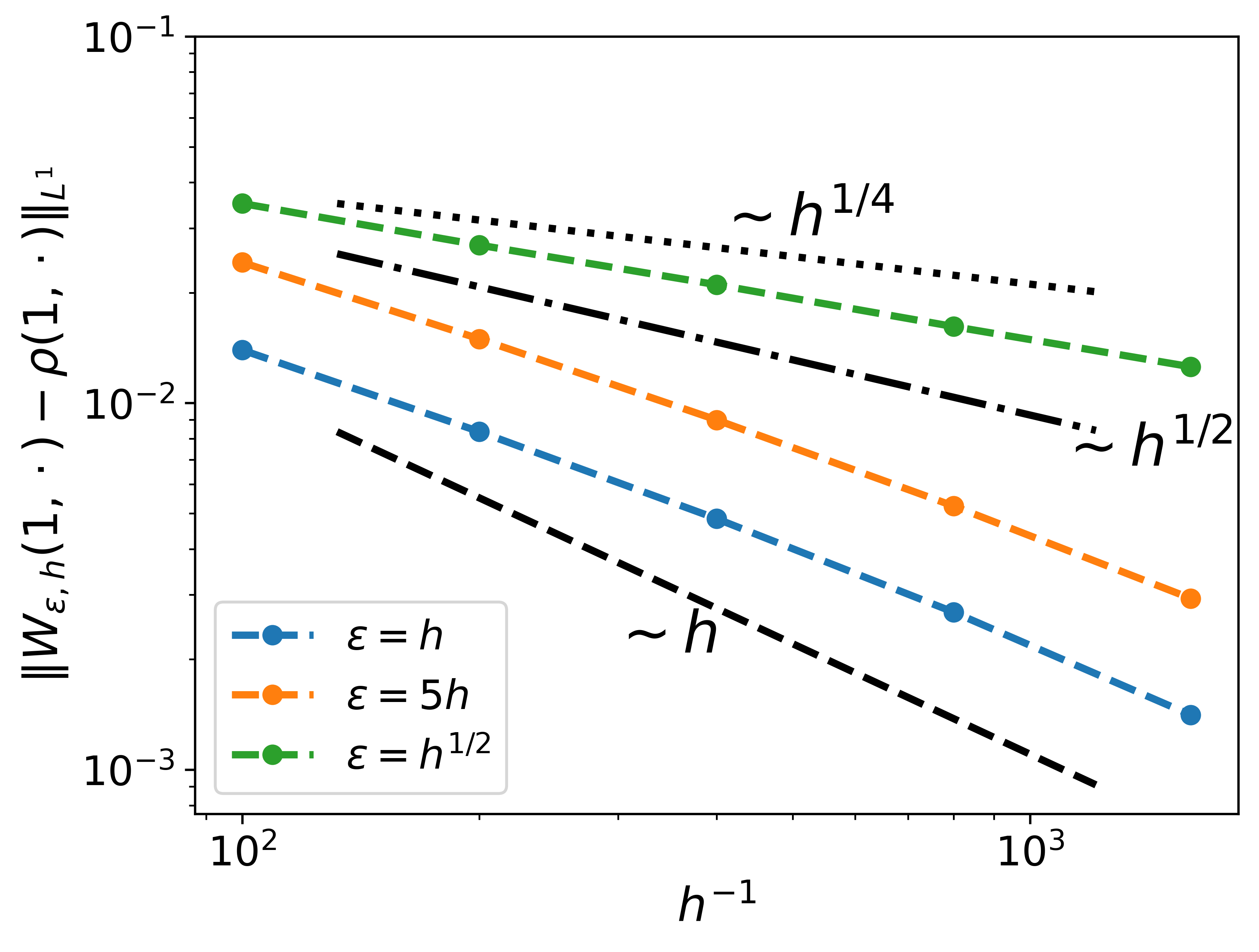}
	\end{subfigure}
    \begin{subfigure}{.32\textwidth}
	\includegraphics[width=\textwidth]{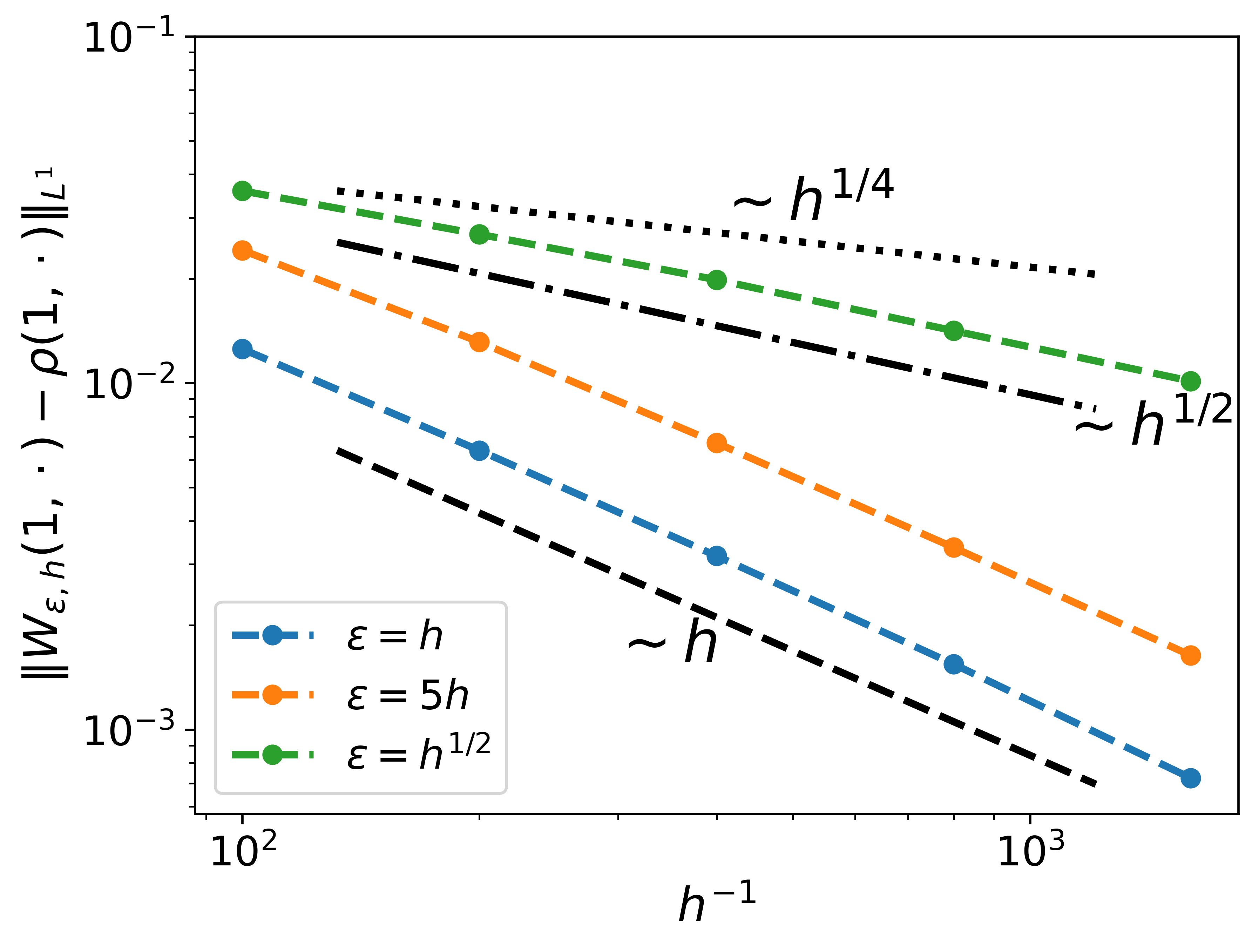}
	\end{subfigure}
	\begin{subfigure}{.32\textwidth}
	\includegraphics[width=\textwidth]{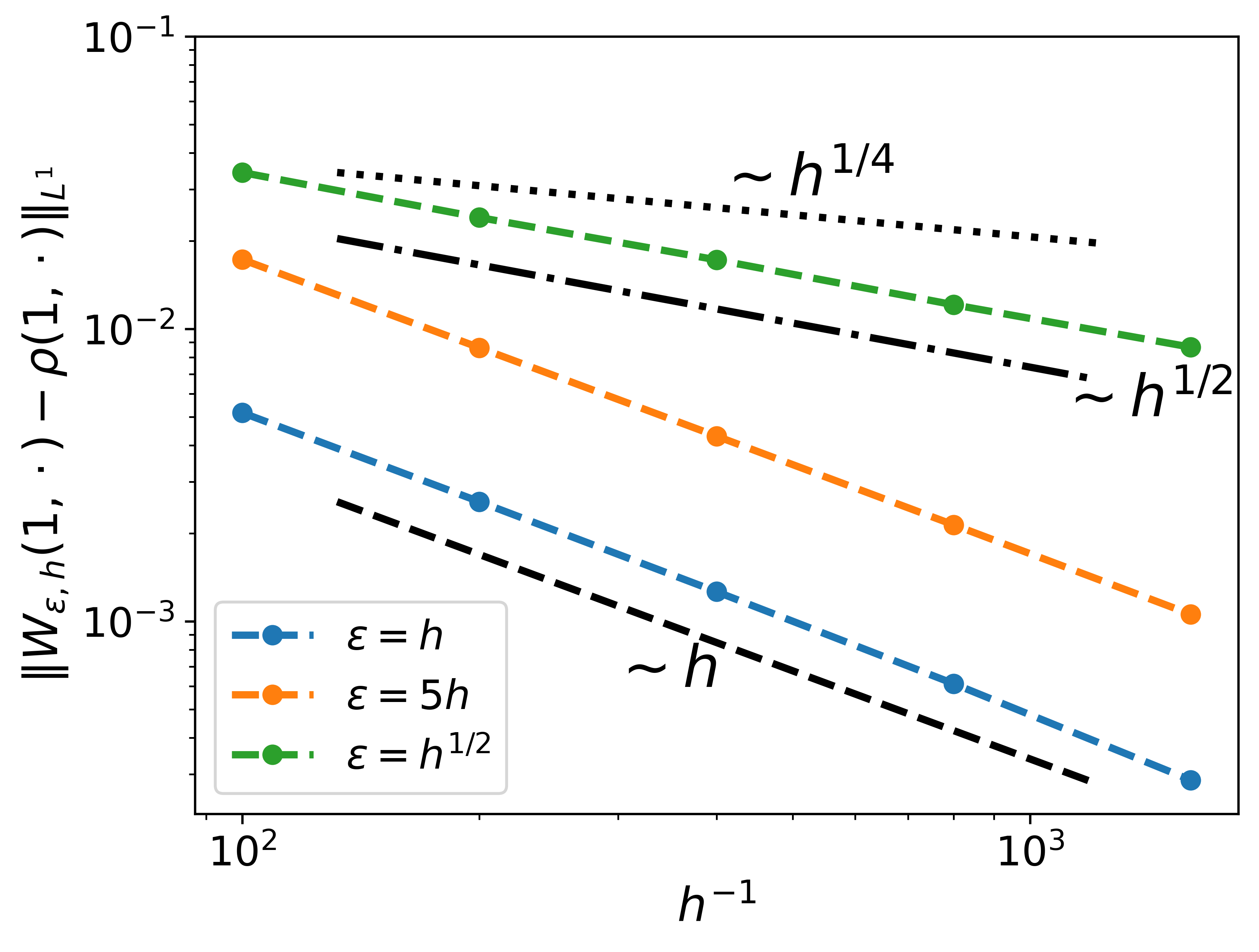}
	\end{subfigure}
	\begin{subfigure}{.32\textwidth}
	\includegraphics[width=\textwidth]{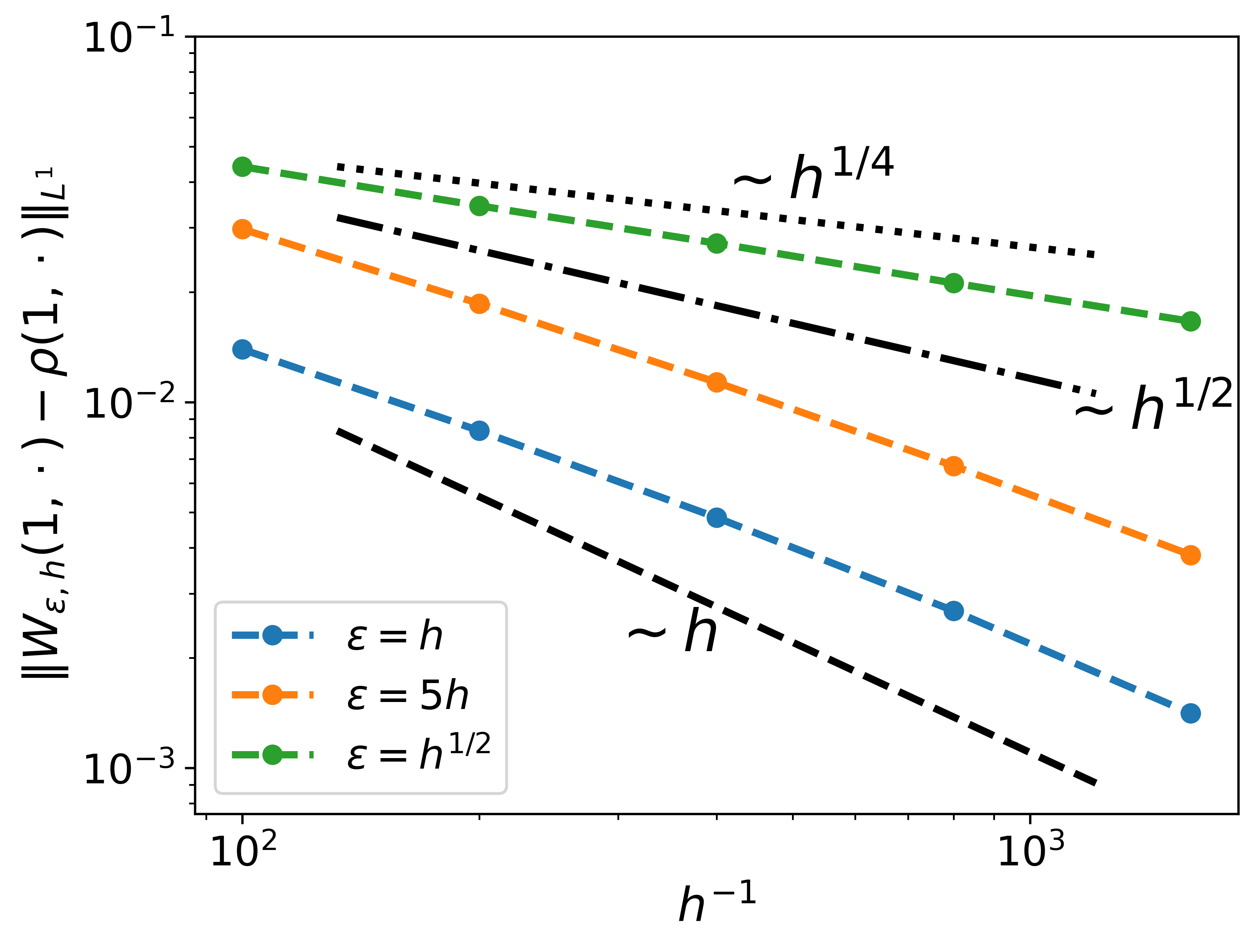}
	\end{subfigure}
    \begin{subfigure}{.32\textwidth}
	\includegraphics[width=\textwidth]{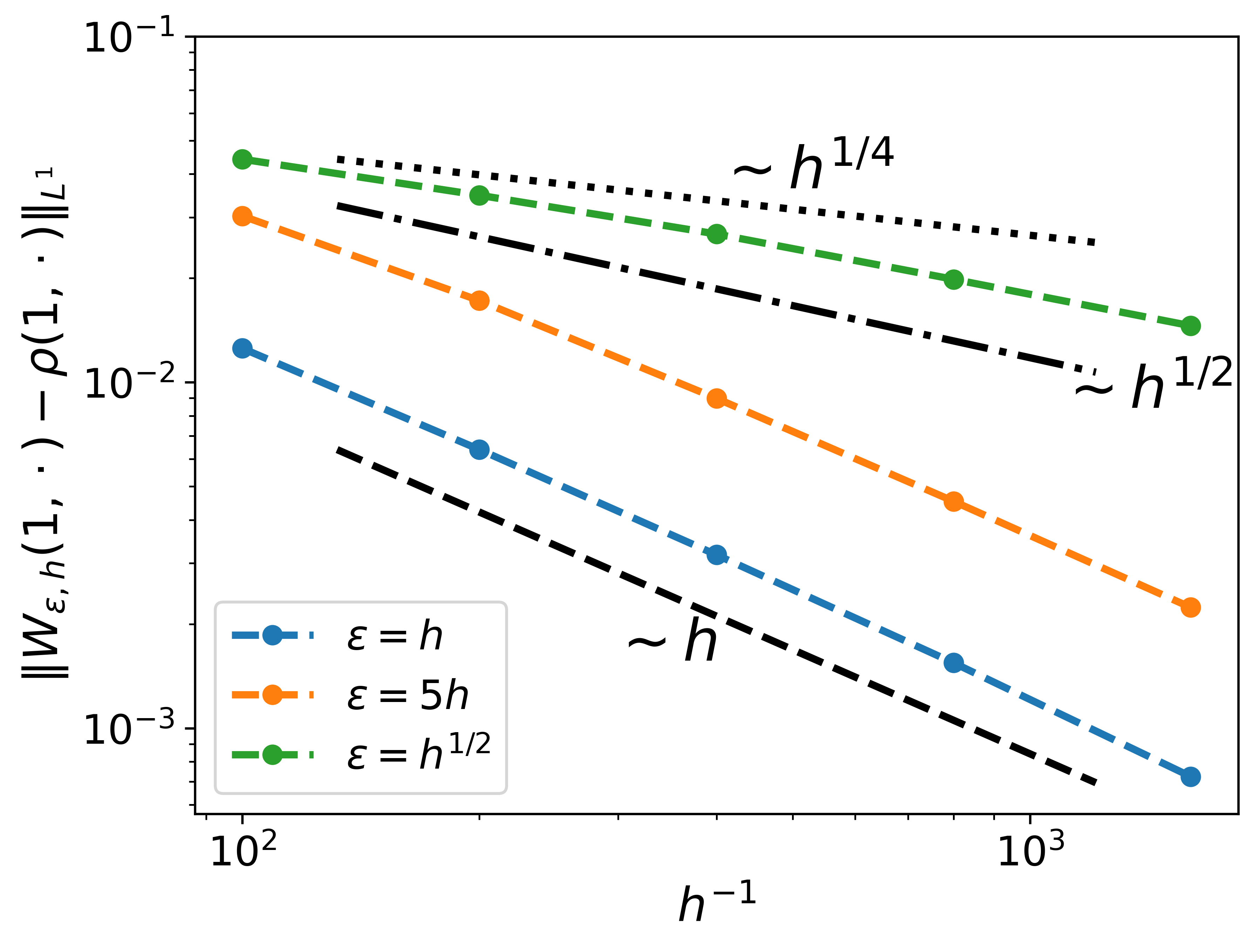}
	\end{subfigure}
	\caption{Convergence from $W_{\eps,h}$ to $\rho$ for the Riemann shock initial data \cref{eq:ini_riemann_shock} (\textsc{left}), the Riemann rarefaction initial data  \cref{eq:ini_riemann_rarefaction} (\textsc{middle}), and the bell-shaped initial data \cref{eq:ini_bellshape} (\textsc{right}) with the exponential kernel \cref{eq:gamma-expo} (\textsc{top}), the linear kernel \cref{eq:gamma-lin} (\textsc{middle}), and the constant kernel \cref{eq:gamma-const} (\textsc{bottom}).}
    \label{fig:exp_1}
\end{figure}

\begin{figure}[htbp]
\centering
	\begin{subfigure}{.32\textwidth}
	\includegraphics[width=\textwidth]{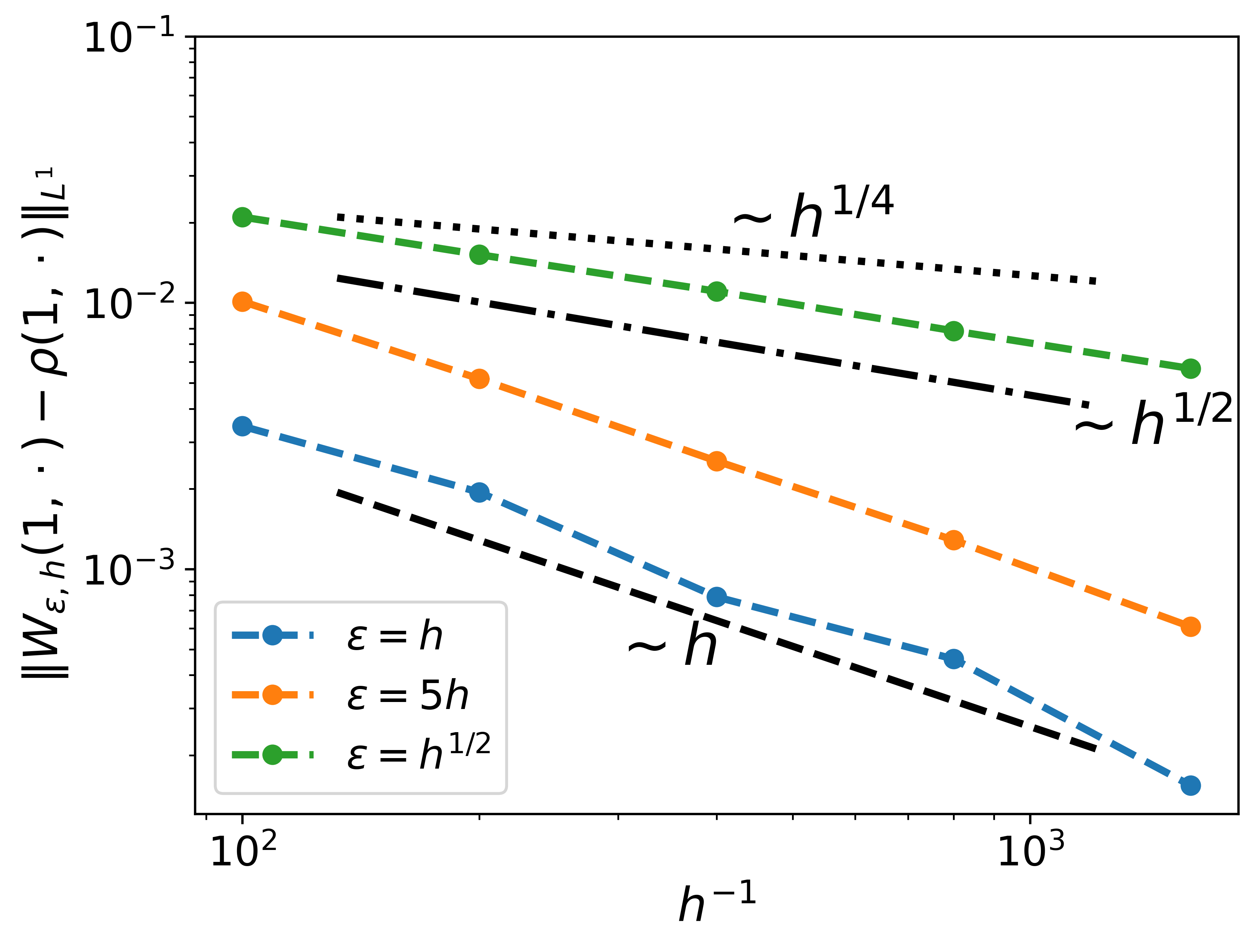}
	\end{subfigure}
	\begin{subfigure}{.32\textwidth}
	\includegraphics[width=\textwidth]{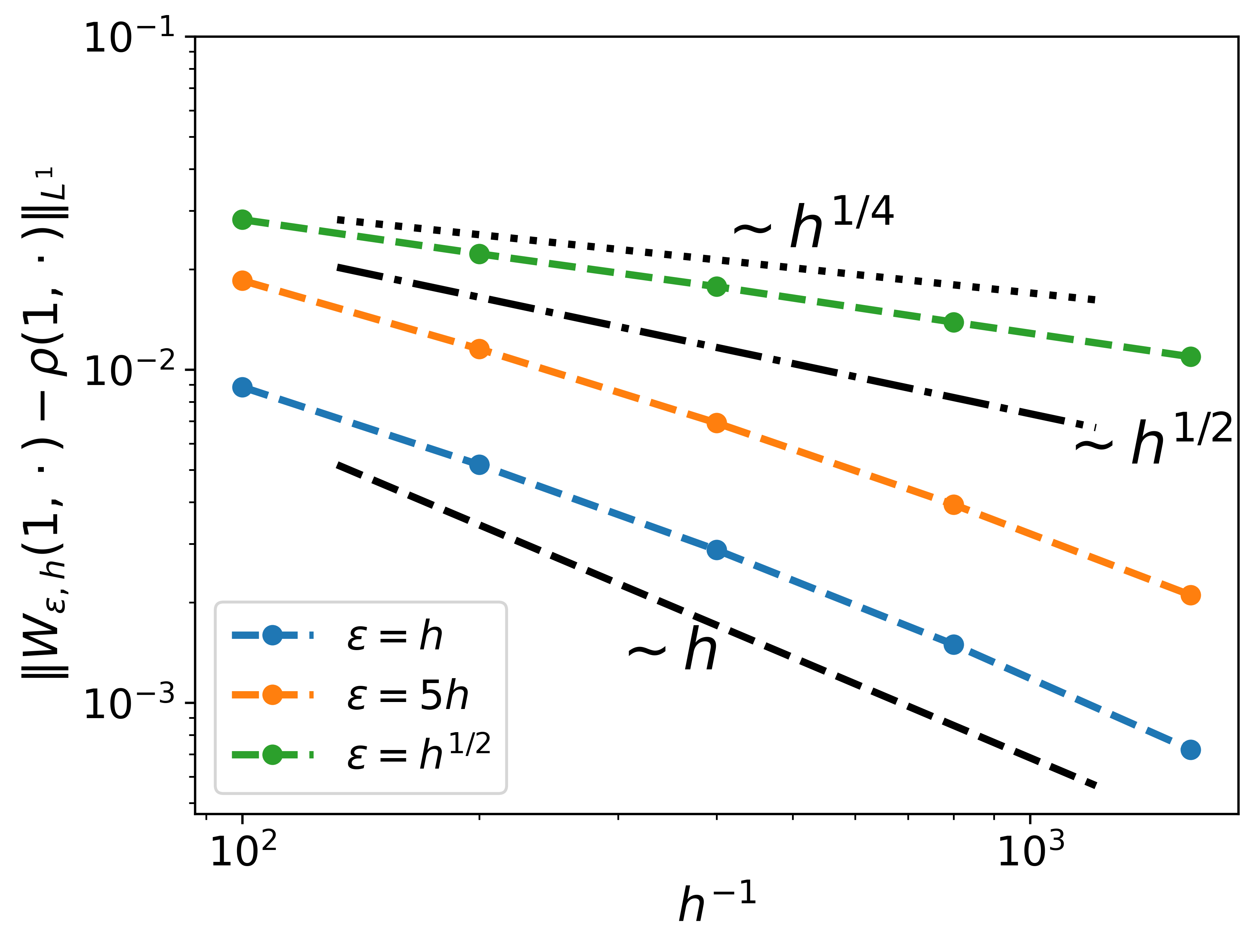}
	\end{subfigure}
	\begin{subfigure}{.32\textwidth}
	\includegraphics[width=\textwidth]{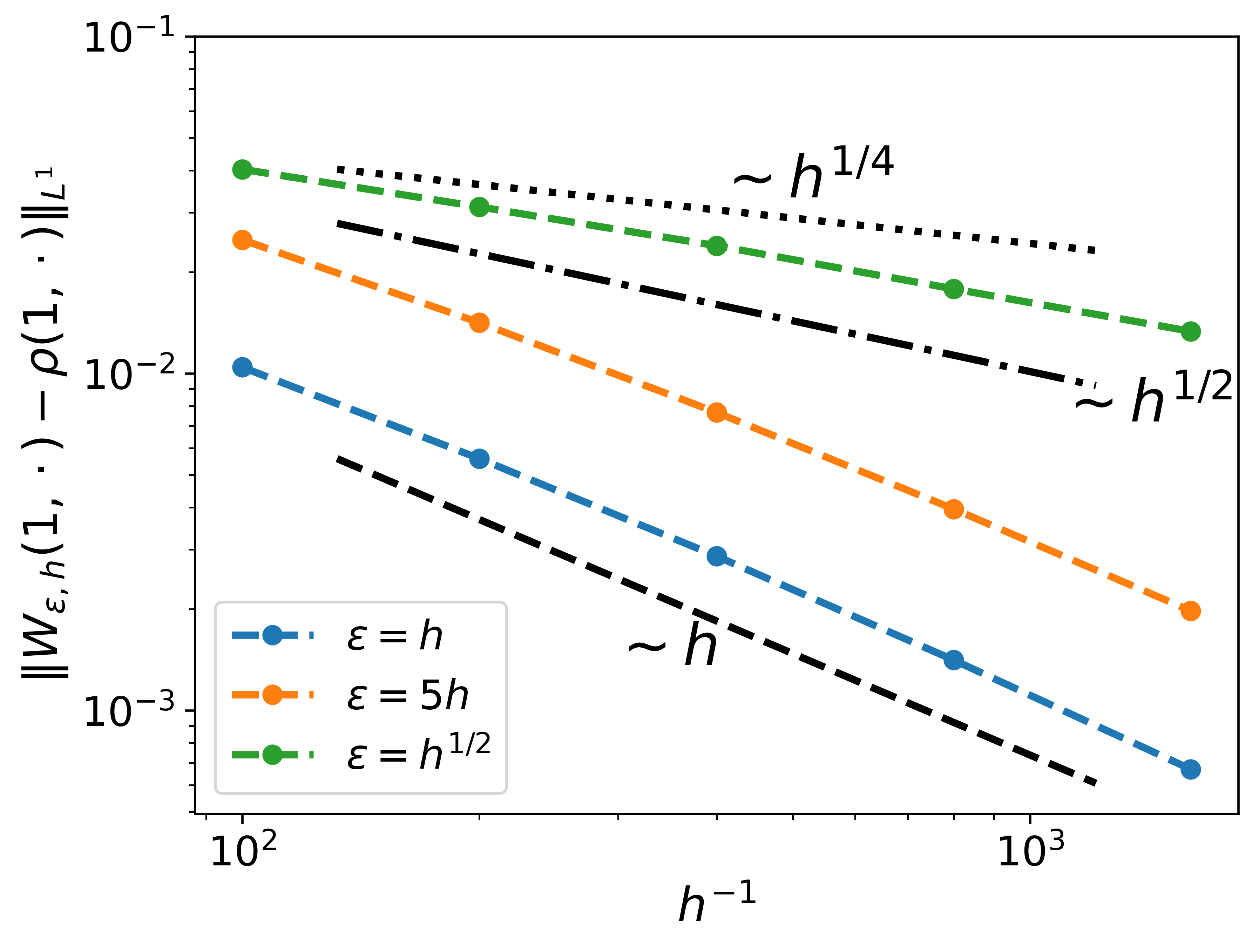}
	\end{subfigure}
	\begin{subfigure}{.32\textwidth}
	\includegraphics[width=\textwidth]{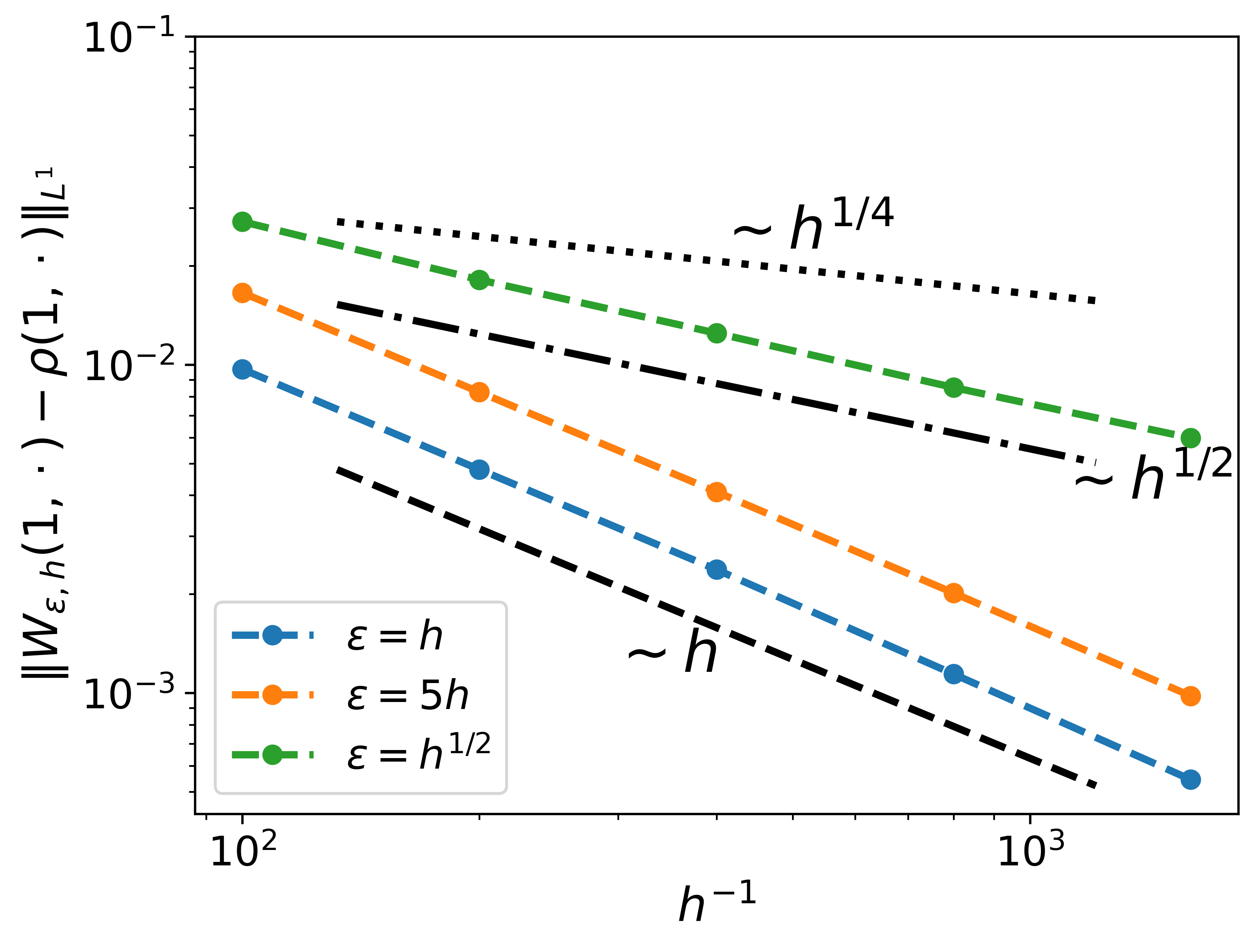}
	\end{subfigure}
	\begin{subfigure}{.32\textwidth}
	\includegraphics[width=\textwidth]{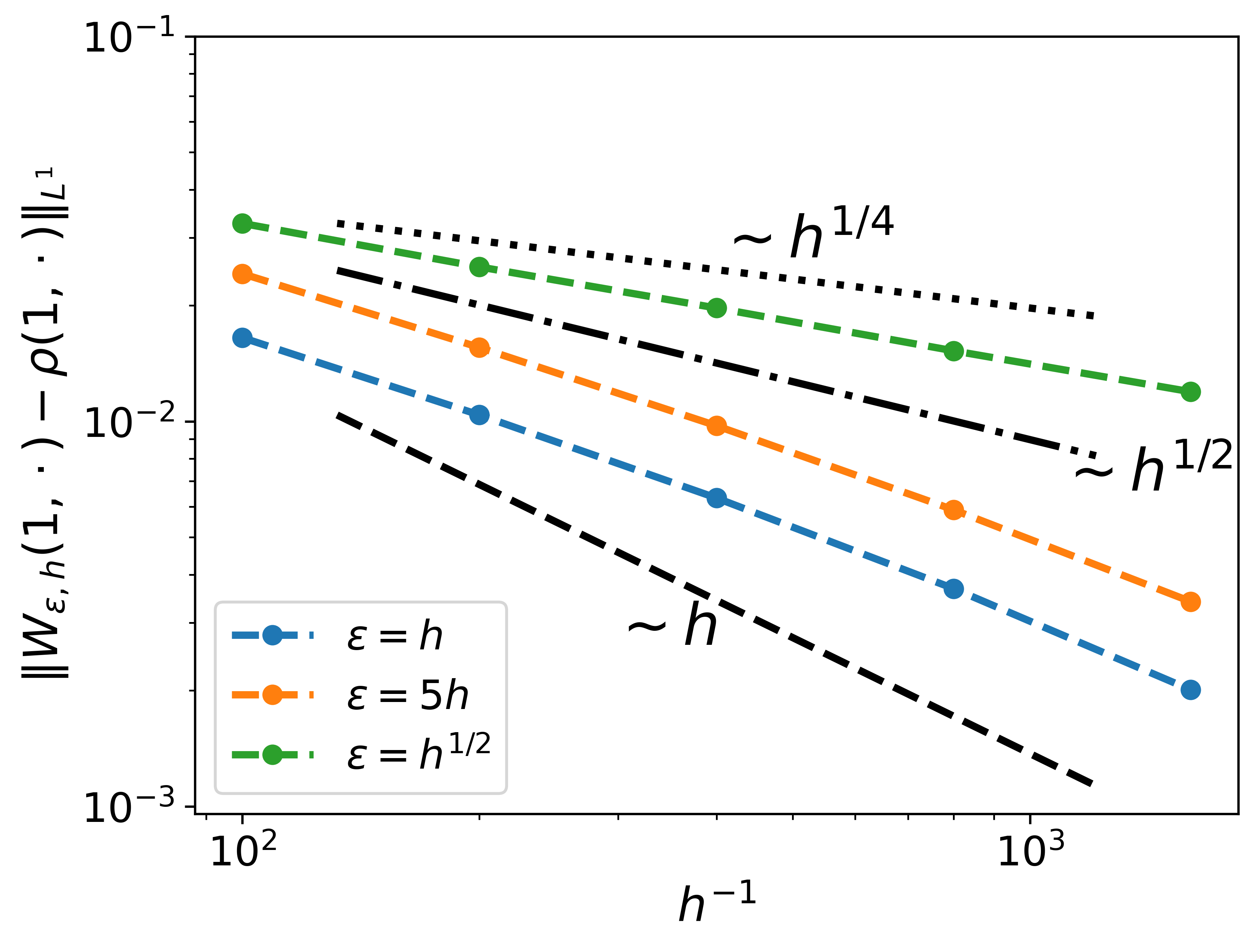}
	\end{subfigure}
	\begin{subfigure}{.32\textwidth}
	\includegraphics[width=\textwidth]{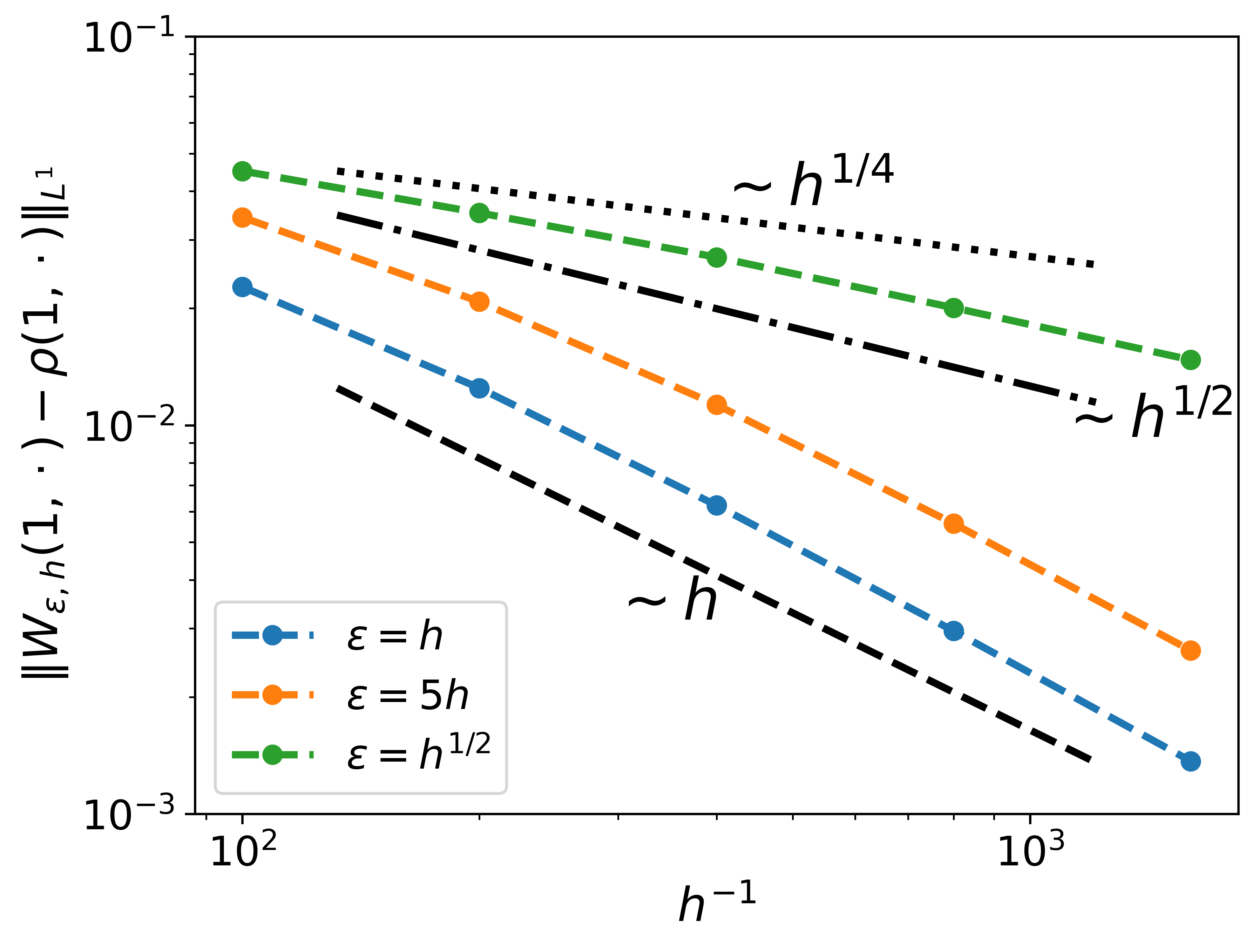}
	\end{subfigure}
	\caption{Convergence from $W_{\eps,h}$ to $\rho$ for the Riemann shock initial data \cref{eq:ini_riemann_shock} (\textsc{left}), the Riemann rarefaction initial data \cref{eq:ini_riemann_rarefaction} (\textsc{middle}), and the bell-shaped initial data \cref{eq:ini_bellshape} (\textsc{right}) with $V(\xi)=(1-\xi)^4$ (\textsc{top}) and $V(\xi)=\exp(-\xi)$ (\textsc{bottom}) for the velocity function.}
    \label{fig:exp_3}
\end{figure}

\begin{example}
In this experiment, we evaluate the impact of alternative quadrature weights beyond the exact ones.
Specifically, we use the linear kernel \cref{eq:gamma-lin} and assess the Riemann quadrature weights $\tilde{\gamma}^{\eps,h}_k$ and the normalized Riemann quadrature weights $\gamma^{\eps,h}_k$ defined in \cref{ex:Riemann}, where the former sums to $1+\frac{h}\eps$ and the latter to 1.
In \cref{fig:exp_2}, we present convergence plots akin to \cref{fig:exp_1} across the initial data \crefrange{eq:ini_riemann_shock}{eq:ini_bellshape} for these weights. Since the Riemann quadrature weights $\tilde{\gamma}^{\eps,h}_k$ fail the normalization condition \cref{ass:normalize} and may lead $W_{\eps,h}$ to violate the maximum principle, we adopt $V(\xi)\coloneqq(1-\xi)^+$ to prevent negative velocities.
\end{example}

The top row of \cref{fig:exp_2} reveals that, for Riemann quadrature weights $\tilde{\gamma}^{\varepsilon,h}_k$, the error of $W_{\varepsilon,h}$ stagnates along $\varepsilon=h$ and $\varepsilon=5h$, as $\sum_{k=0}^\infty \tilde{\gamma}^{\varepsilon,h}_k$ equals 2 and 1.2, respectively, failing the normalization condition \cref{ass:normalize} and leading to persistent overestimates of the nonlocal quantity $W$ in \cref{eq:num_nonlocal_W}, thus preventing convergence. For $\varepsilon=\sqrt{h}$ with $\sum_{k=0}^\infty \tilde{\gamma}^{\varepsilon,h}_k=1+\sqrt{h}$, the violation of the normalization condition with magnitude $\sqrt{h}$ decays to zero as $h\searrow0$, leading to an error decay rate of $\sqrt{h}$ for $W_{\varepsilon,h}$ as $h\searrow0$. The bottom row shows that, for normalized Riemann quadrature weights $\gamma^{\varepsilon,h}_k$, convergence rates align with those in \cref{fig:exp_1} using exact quadrature weights, confirming the sufficiency of the normalization condition. These findings align with \cite{MR4742183}, underscoring the critical role of the normalization condition in ensuring asymptotic compatibility. Moreover, they indicate that conditional asymptotic compatibility may hold if the violation of the normalization condition vanishes along specific limiting paths.

\begin{figure}[htbp]
\centering
    \begin{subfigure}{.32\textwidth}
	\includegraphics[width=\textwidth]{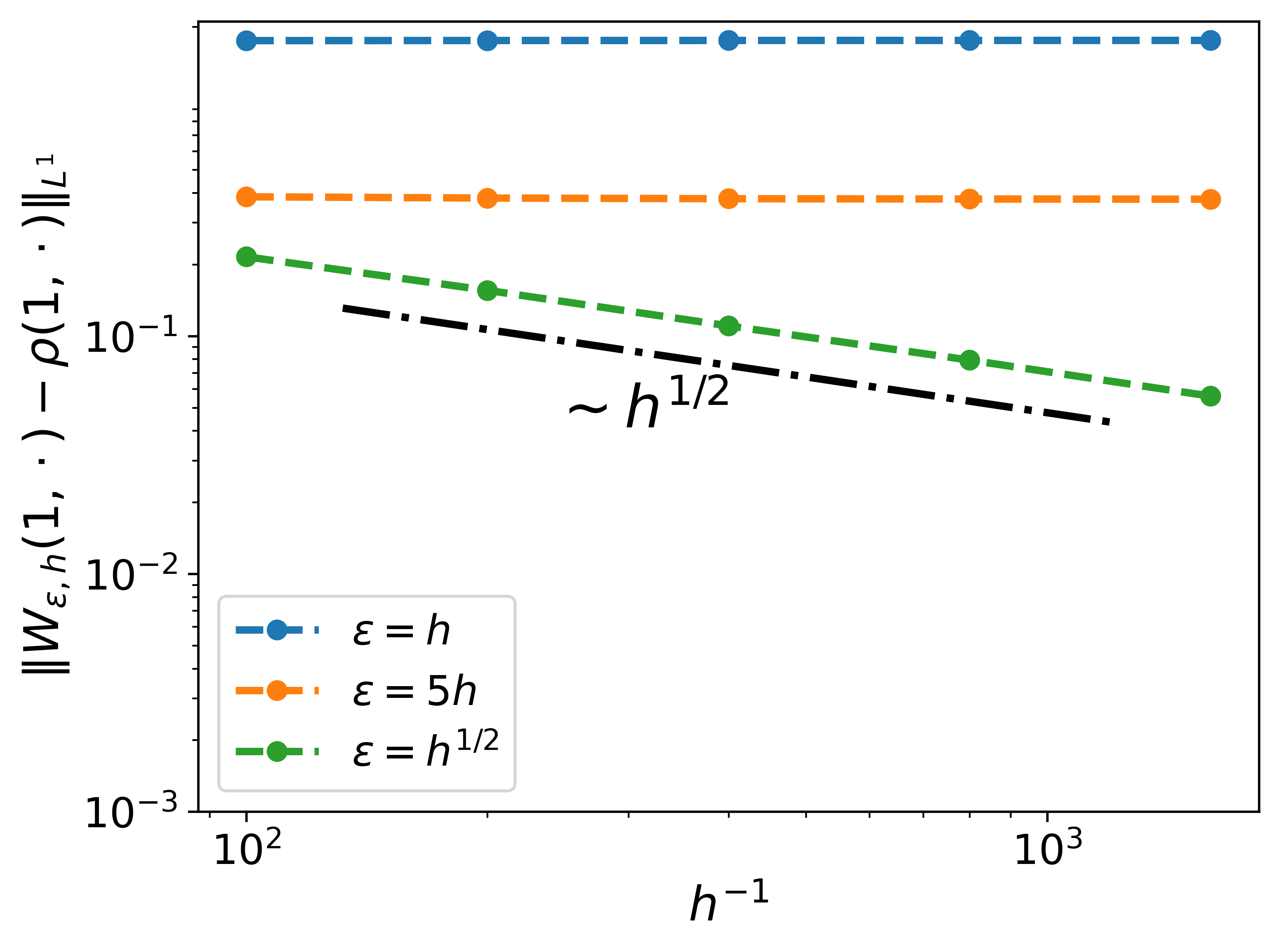}
	\end{subfigure}
	\begin{subfigure}{.32\textwidth}
	\includegraphics[width=\textwidth]{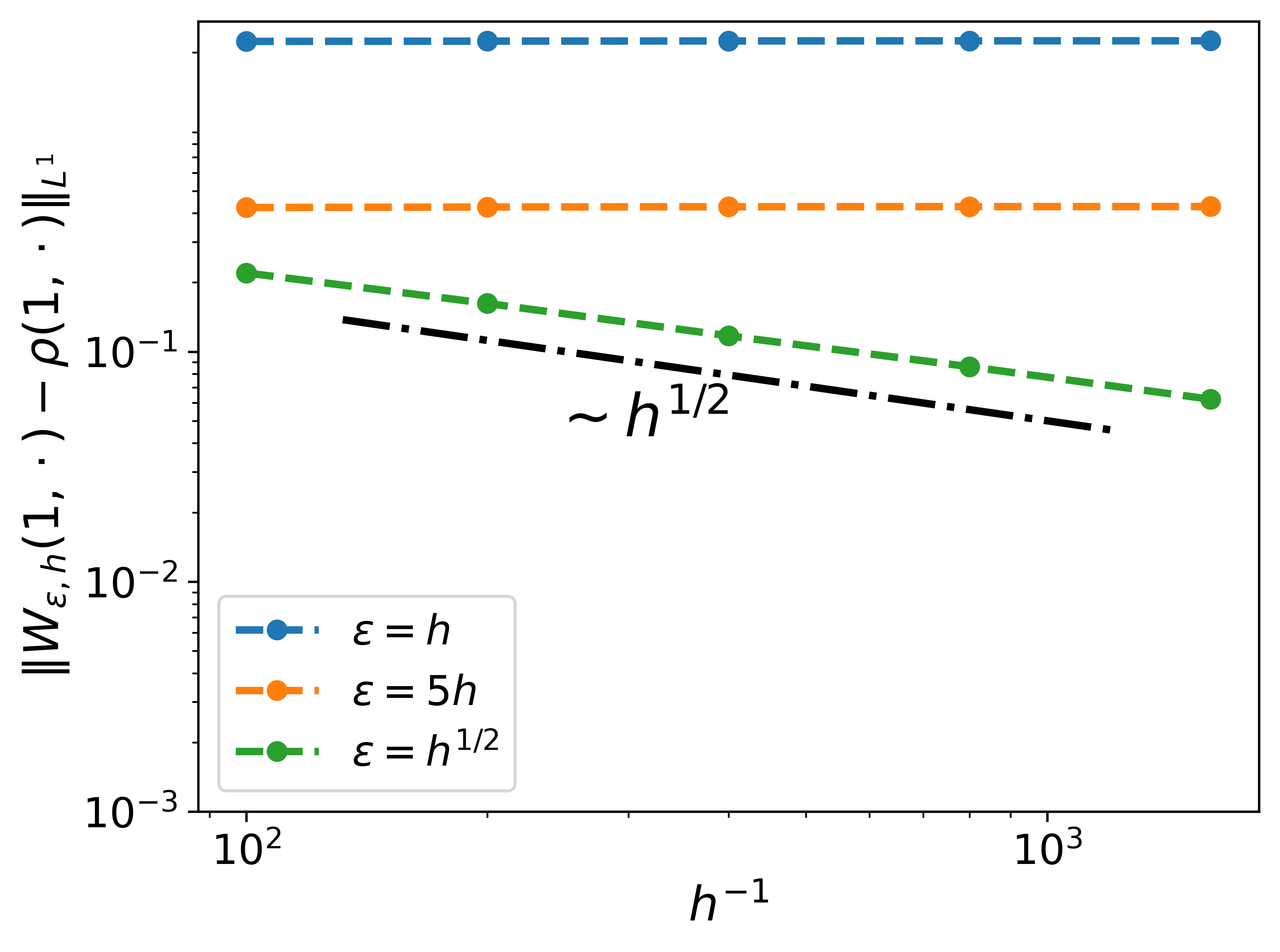}
	\end{subfigure}
    \begin{subfigure}{.32\textwidth}
	\includegraphics[width=\textwidth]{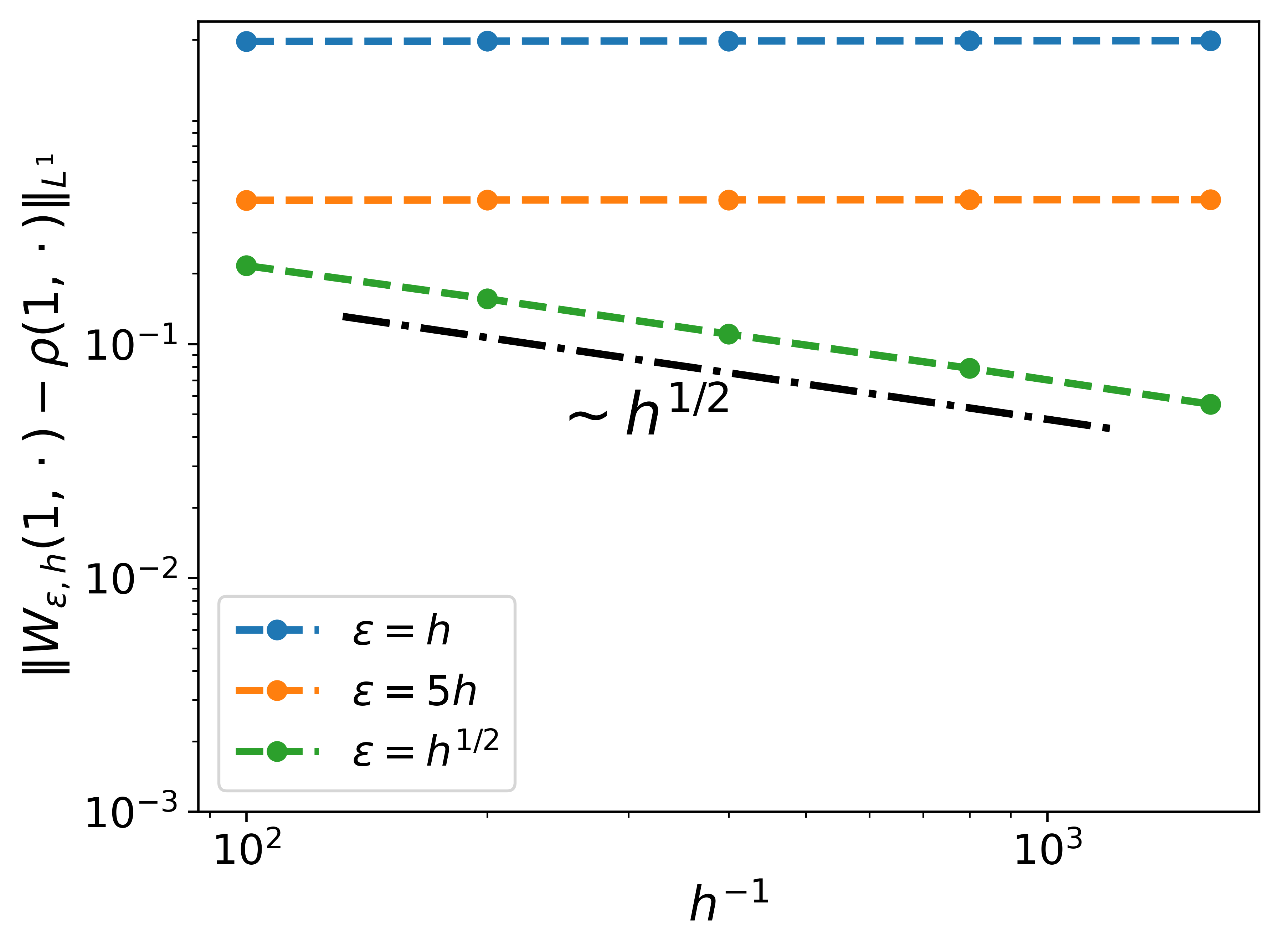}
	\end{subfigure}
    \begin{subfigure}{.32\textwidth}
	\includegraphics[width=\textwidth]{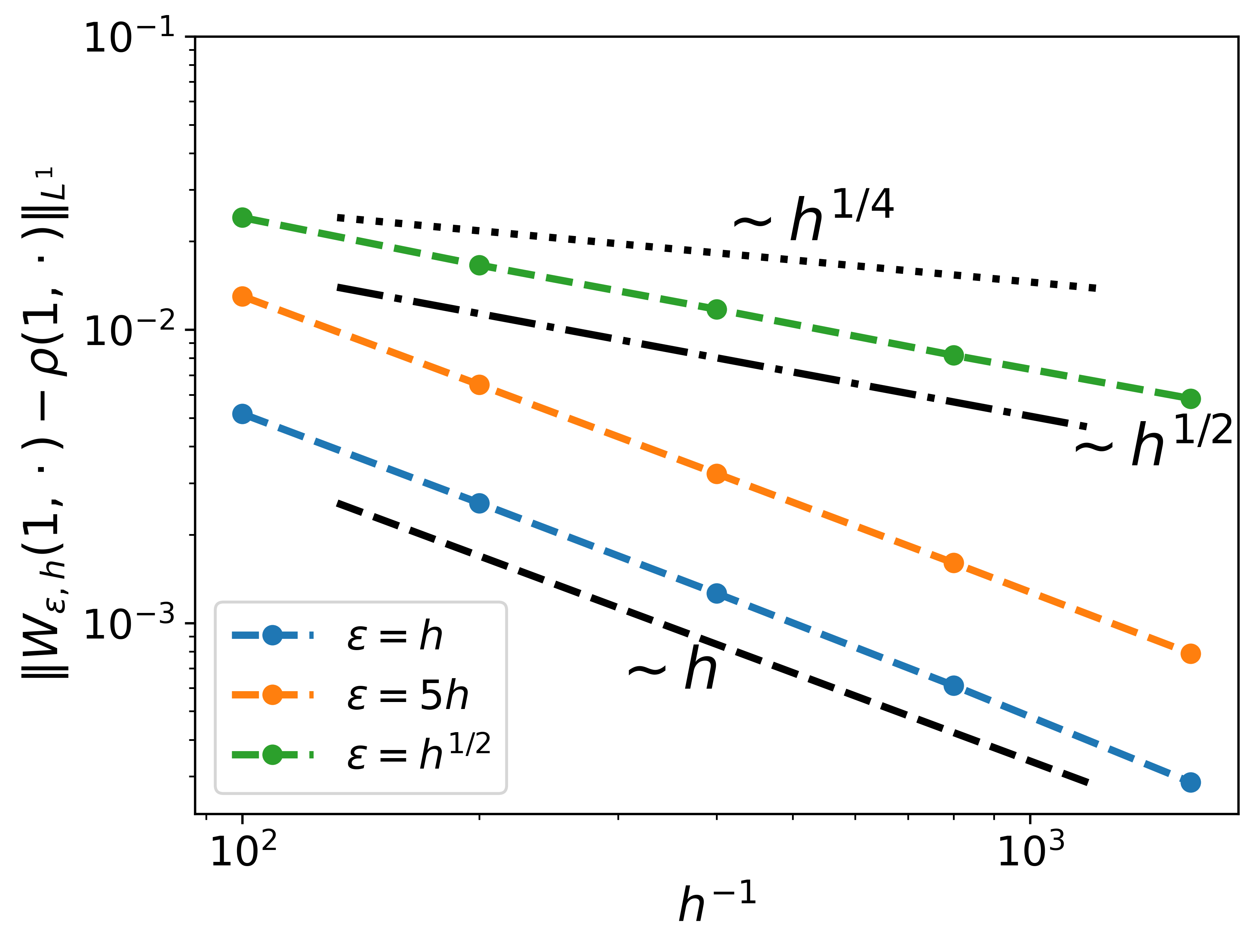}
	\end{subfigure}
	\begin{subfigure}{.32\textwidth}
	\includegraphics[width=\textwidth]{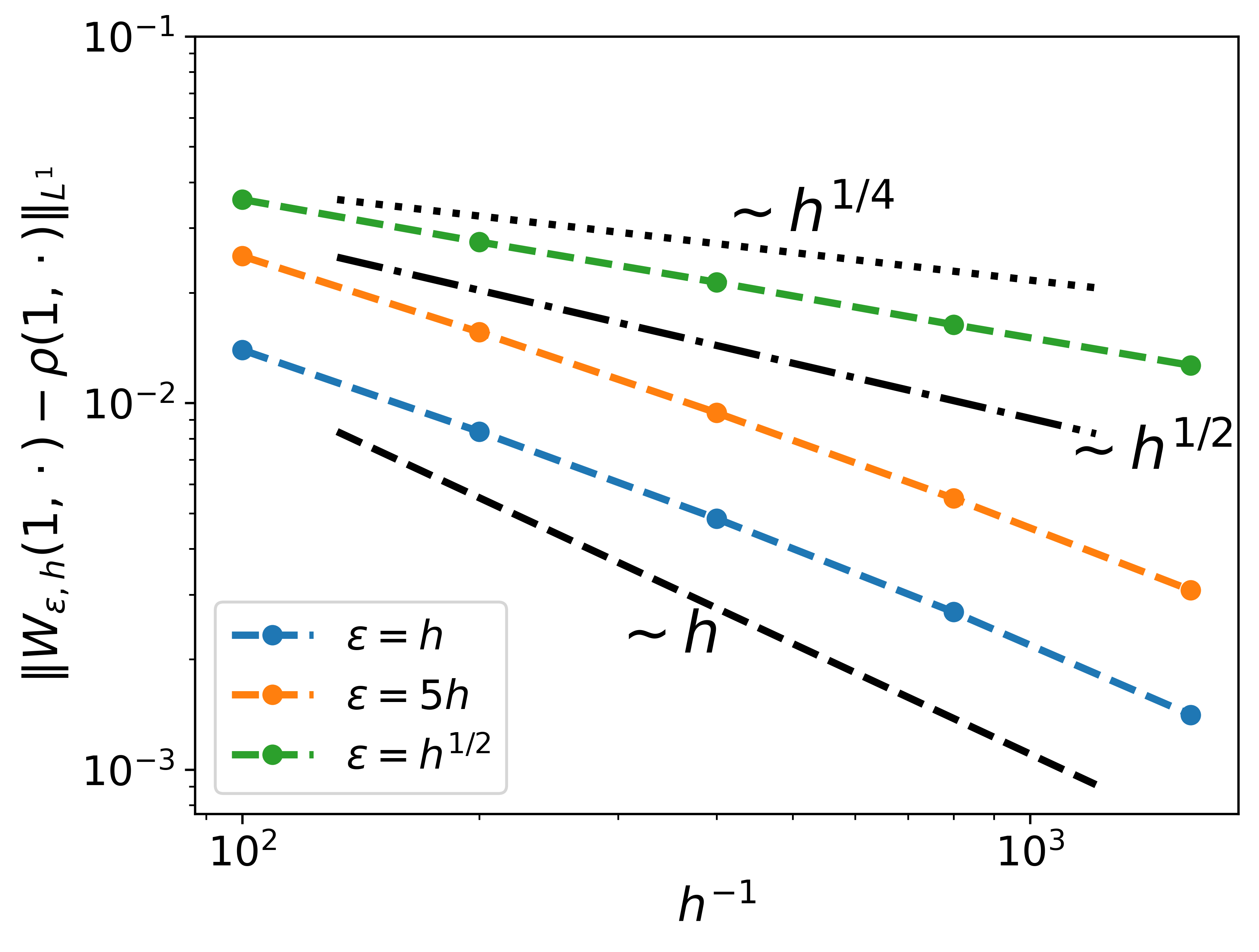}
	\end{subfigure}
    \begin{subfigure}{.32\textwidth}
	\includegraphics[width=\textwidth]{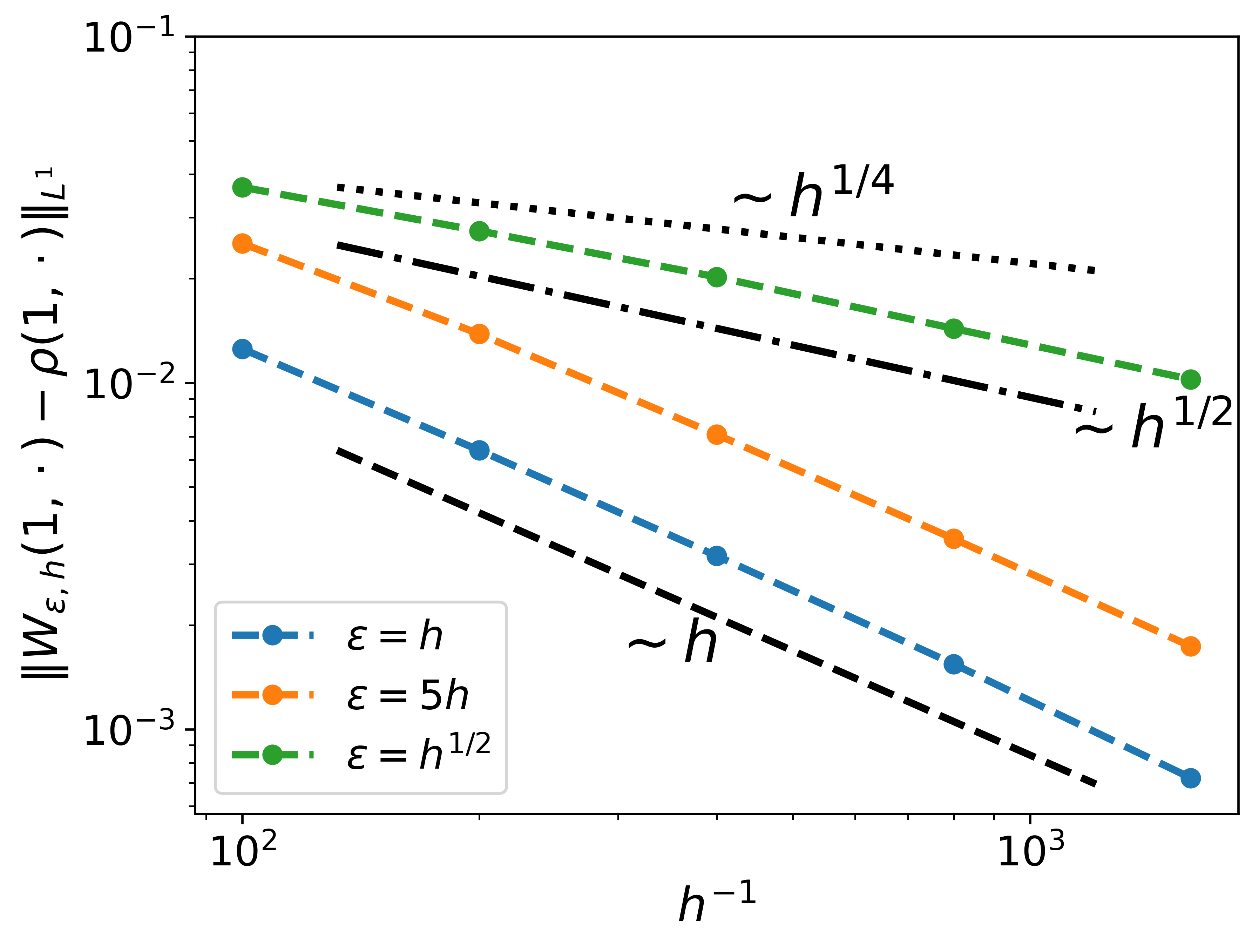}
	\end{subfigure}
	\caption{Convergence from $W_{\eps,h}$ to $\rho$ for the Riemann shock initial data \cref{eq:ini_riemann_shock} (\textsc{left}), the Riemann rarefaction initial data \cref{eq:ini_riemann_rarefaction} (\textsc{middle}), and the bell-shaped initial data \cref{eq:ini_bellshape} (\textsc{right}) with the Riemann quadrature weights $\tilde{\gamma}^{\eps,h}_k$ (\textsc{top}) and the normalized Riemann quadrature weights $\gamma^{\eps,h}_k$ (\textsc{bottom}) defined in \cref{ex:Riemann}.}
    \label{fig:exp_2}
\end{figure}

\begin{example}
Finally, we examine the convergence from $\rho_{\eps,h}$ to $\rho$ using the initial data \crefrange{eq:ini_riemann_shock}{eq:ini_bellshape} and nonlocal kernels \crefrange{eq:gamma-expo}{eq:gamma-const}. In \cref{fig:exp_4}, we present convergence plots akin to \cref{fig:exp_1}.
\end{example}

The convergence behaviors in \cref{fig:exp_4} mirror those in \cref{fig:exp_1}, suggesting that the convergence result in \cref{th:main-2} may apply to $\rho_{\varepsilon,h}$ for a broader range of kernels, despite the theoretical result applying only to the exponential kernel (cf.~\cref{th:main-1}). Moreover, with the Riemann shock initial data \cref{eq:ini_riemann_shock} and along $\eps=\sqrt{h}$, the convergence rate from $\rho_{\eps,h}$ to $\rho$ exceeds $\sqrt{h}$, implying convergence rates possibly surpassing $\eps+h$ in certain regimes despite the scheme's first-order nature.

We defer the investigation of diverse convergence behaviors across different initial data, nonlocal kernels, limiting paths, and transitions from $W_{\varepsilon,h}$ to $\rho_{\varepsilon,h}$ to future research.

\begin{figure}[htbp]
\centering
    \begin{subfigure}{.32\textwidth}
	\includegraphics[width=\textwidth]{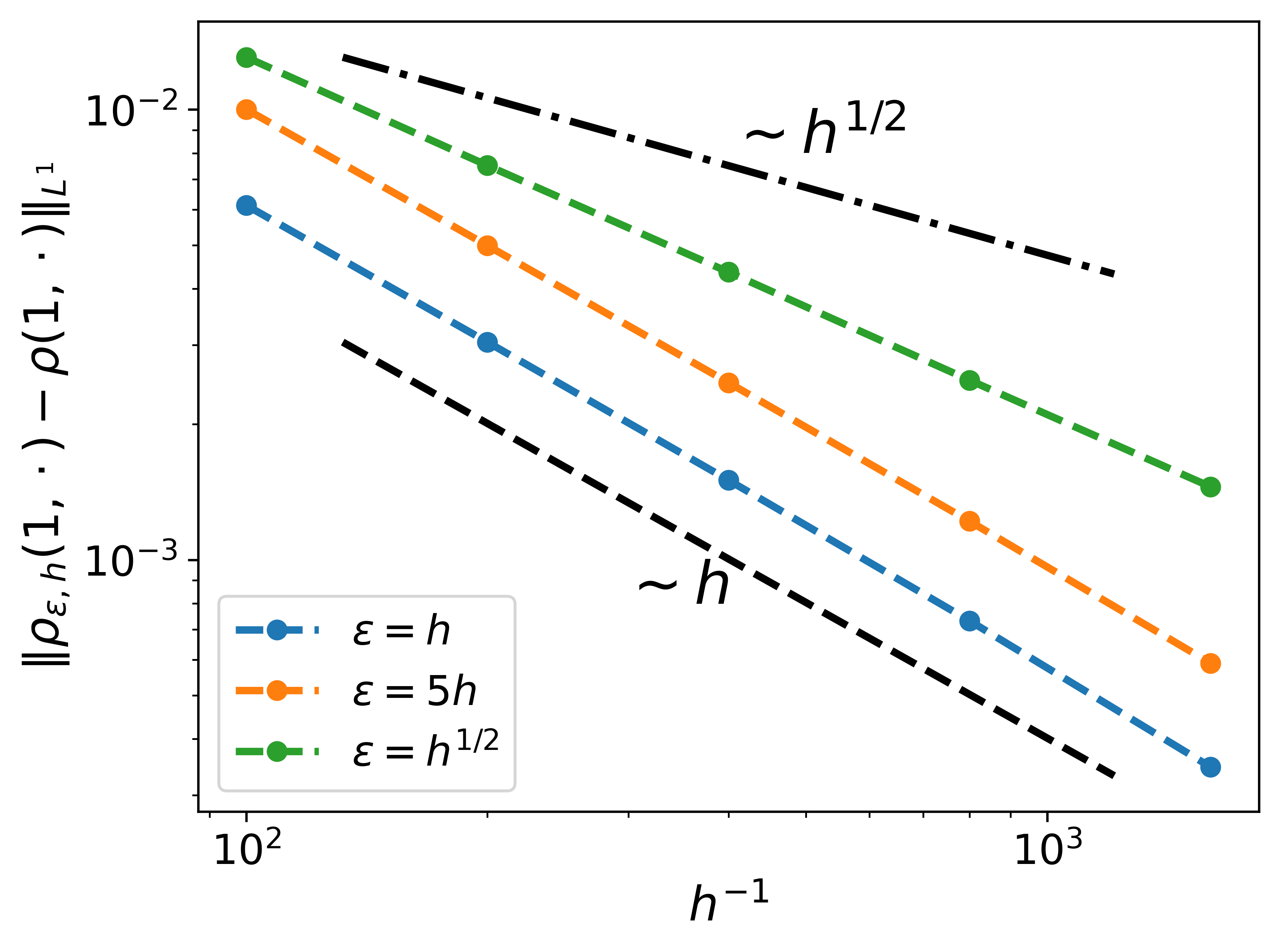}
	\end{subfigure}
	\begin{subfigure}{.32\textwidth}
	\includegraphics[width=\textwidth]{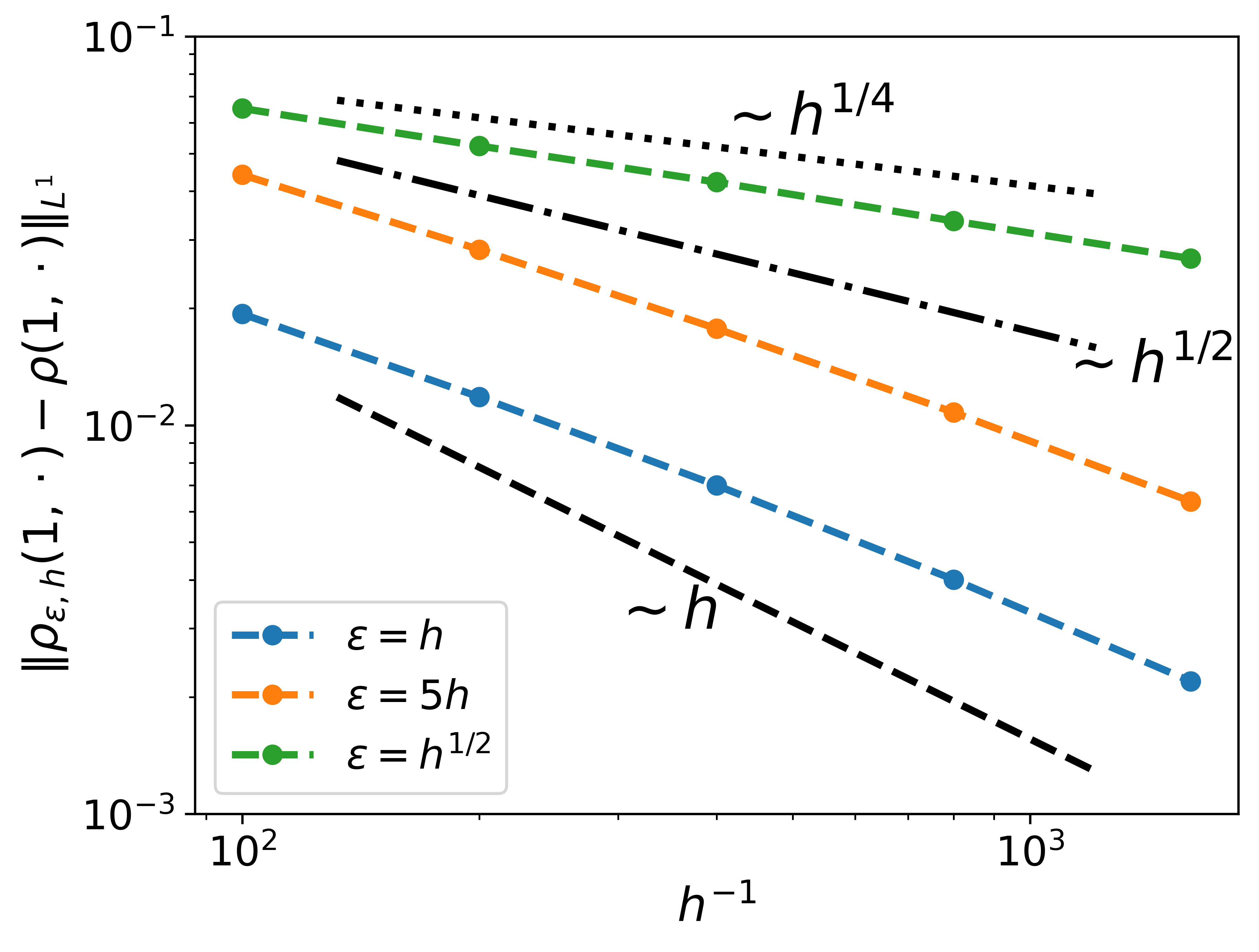}
	\end{subfigure}
    \begin{subfigure}{.32\textwidth}
	\includegraphics[width=\textwidth]{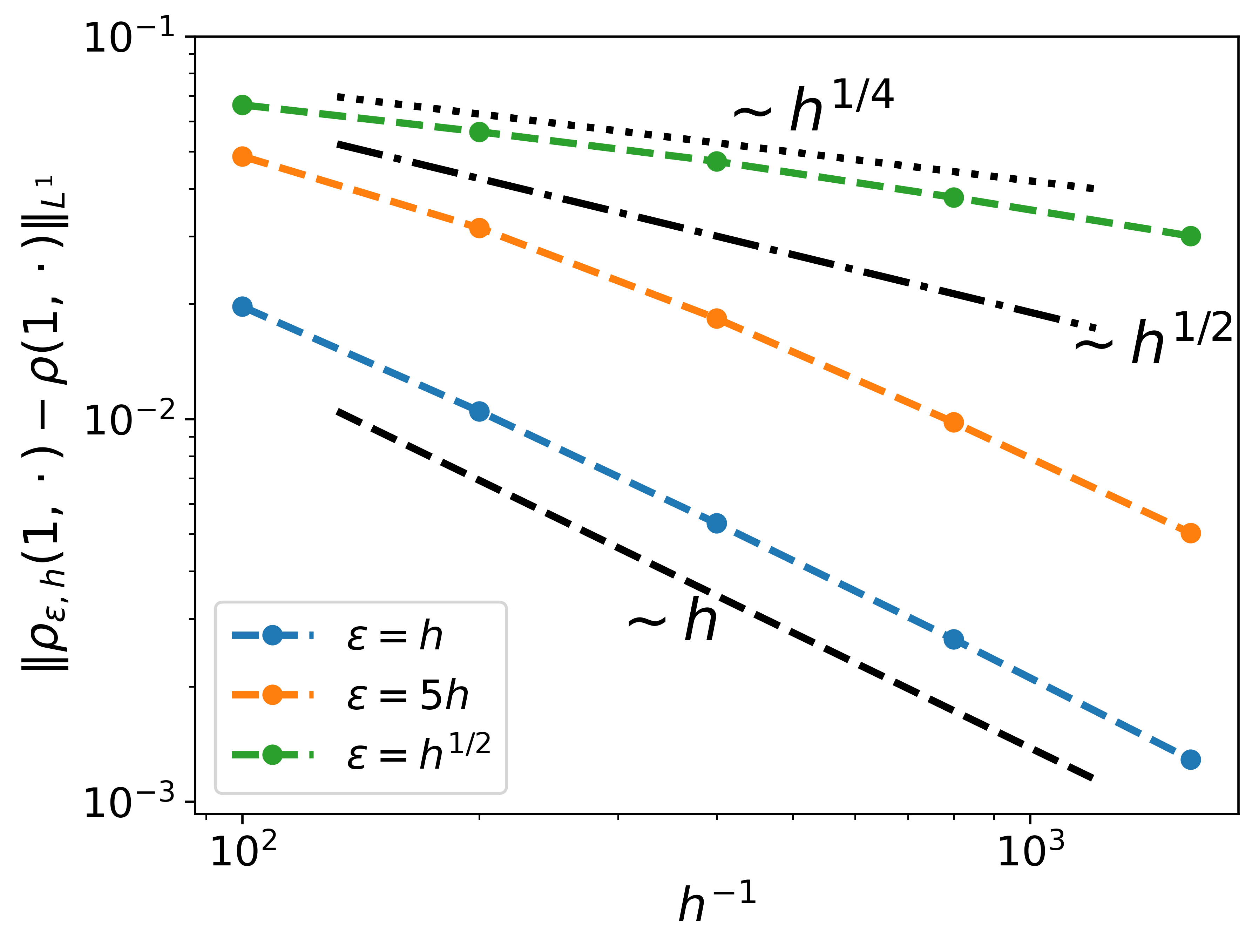}
	\end{subfigure}
	\begin{subfigure}{.32\textwidth}
	\includegraphics[width=\textwidth]{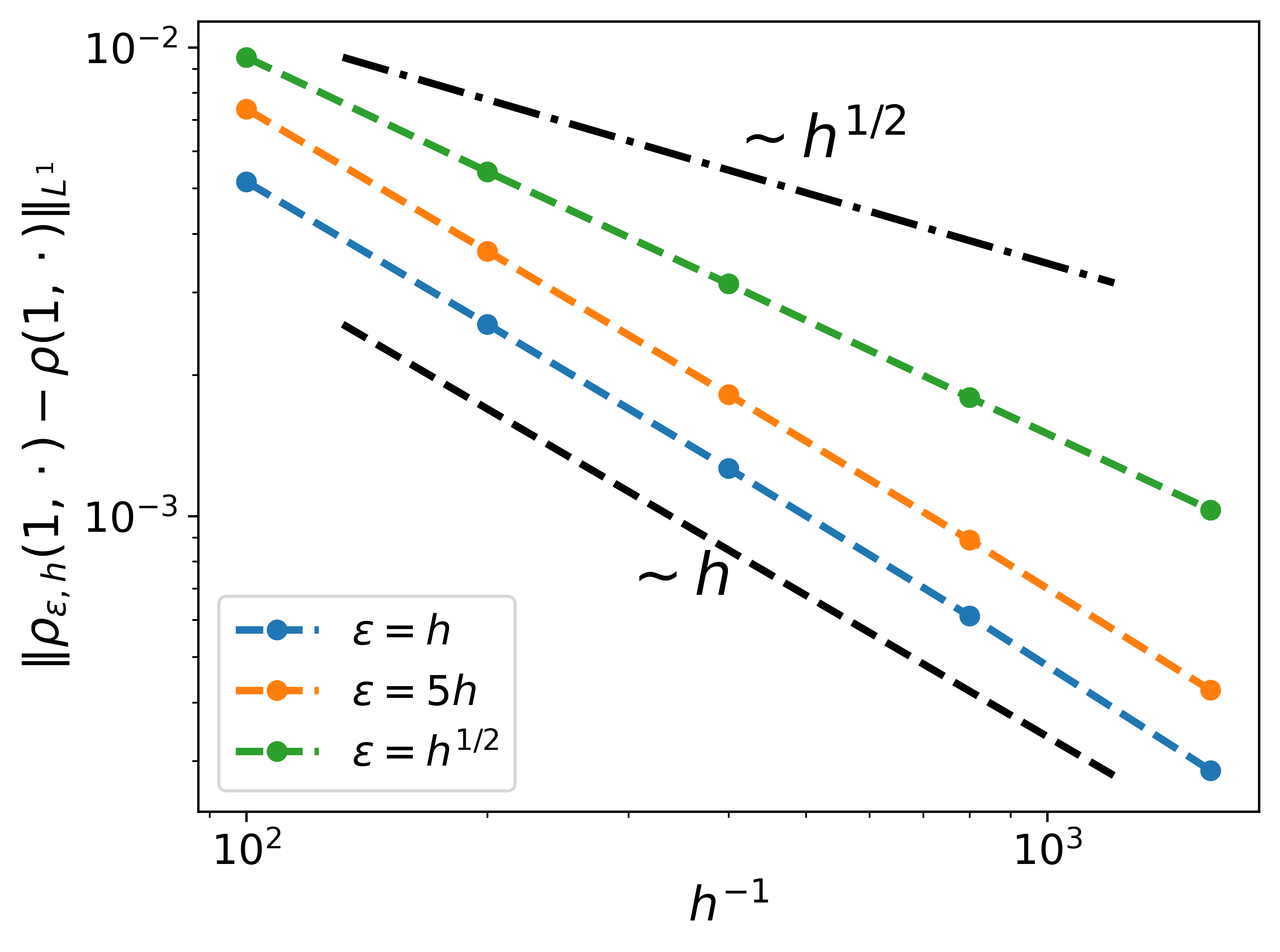}
	\end{subfigure}
	\begin{subfigure}{.32\textwidth}
	\includegraphics[width=\textwidth]{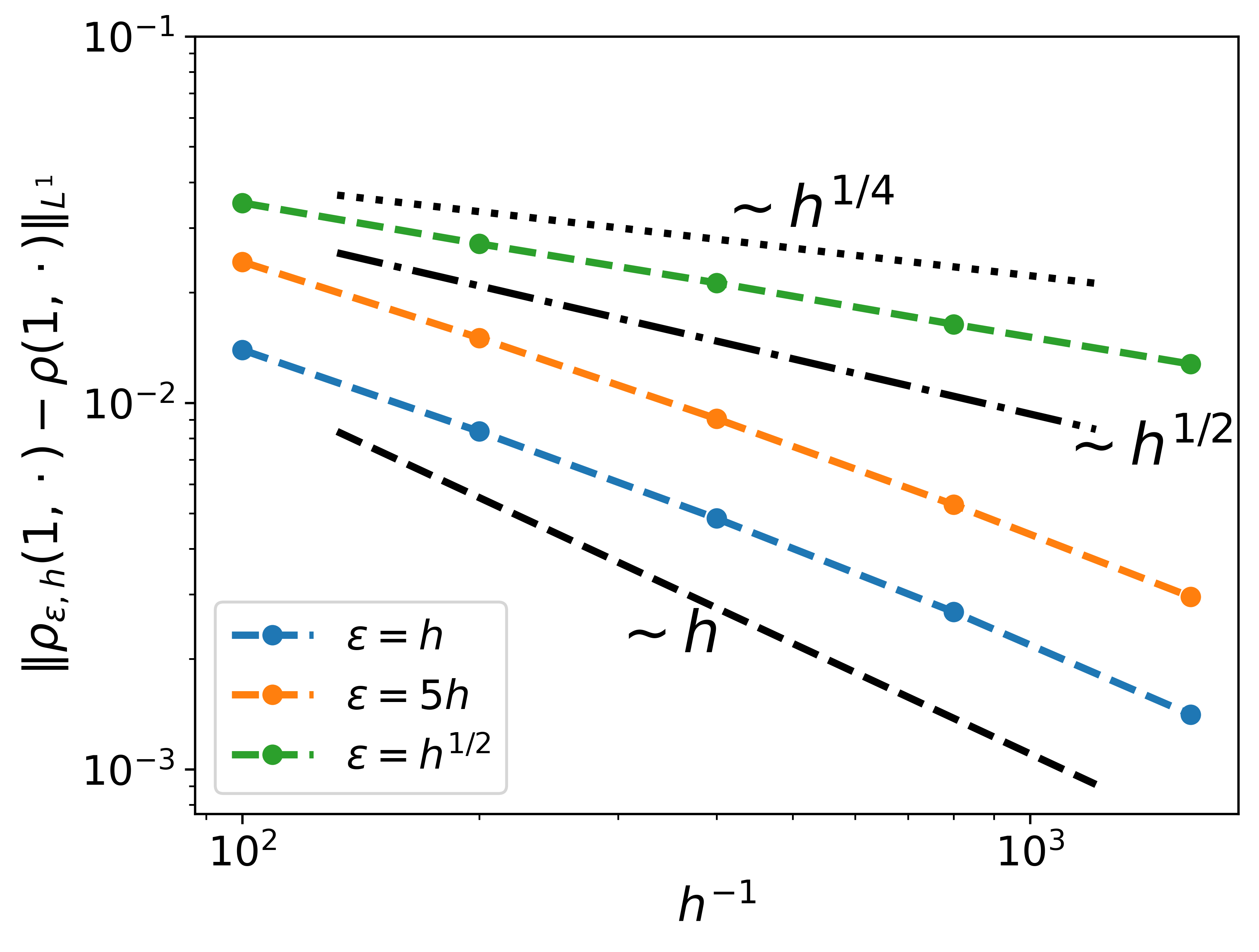}
	\end{subfigure}
    \begin{subfigure}{.32\textwidth}
	\includegraphics[width=\textwidth]{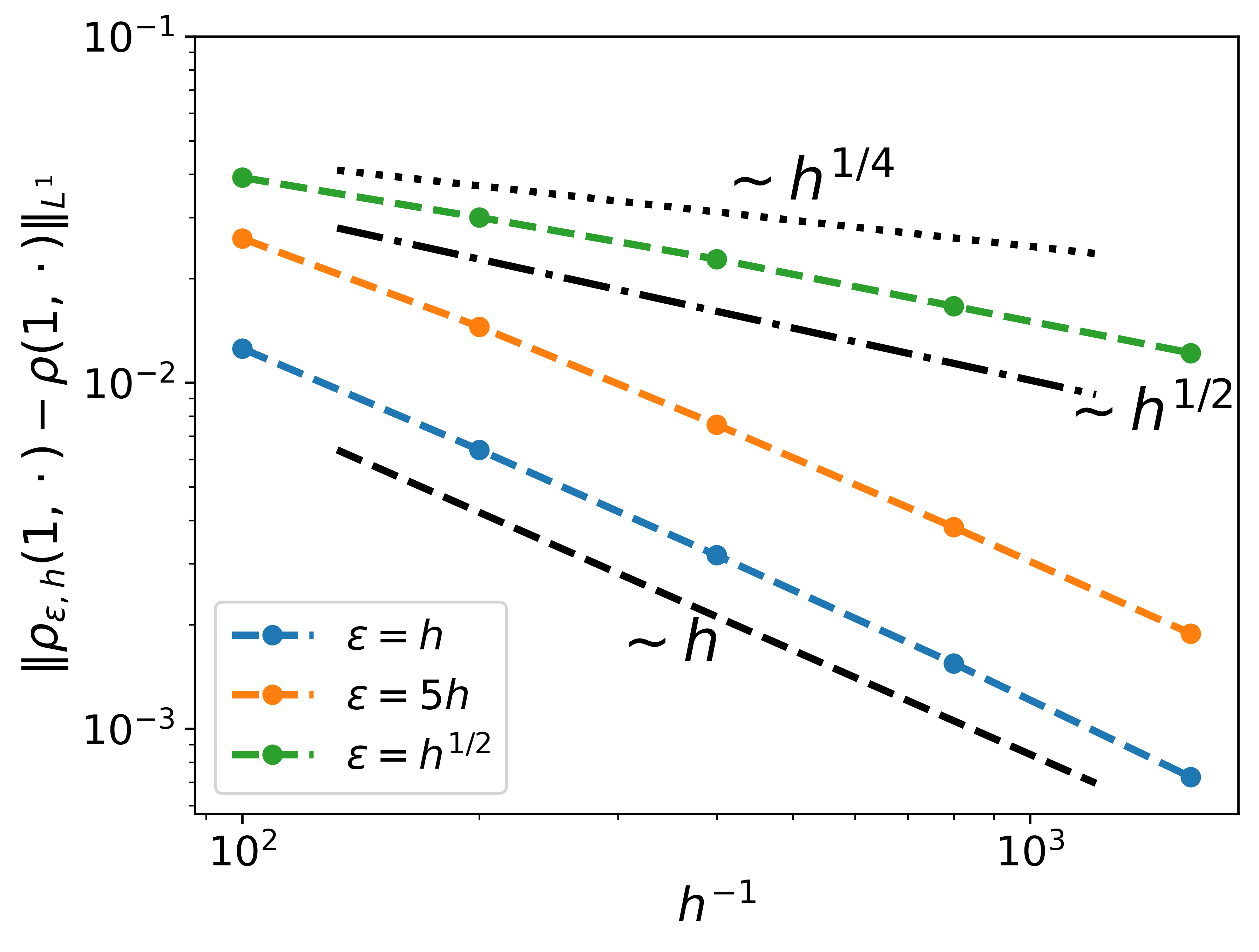}
	\end{subfigure}
	\begin{subfigure}{.32\textwidth}
	\includegraphics[width=\textwidth]{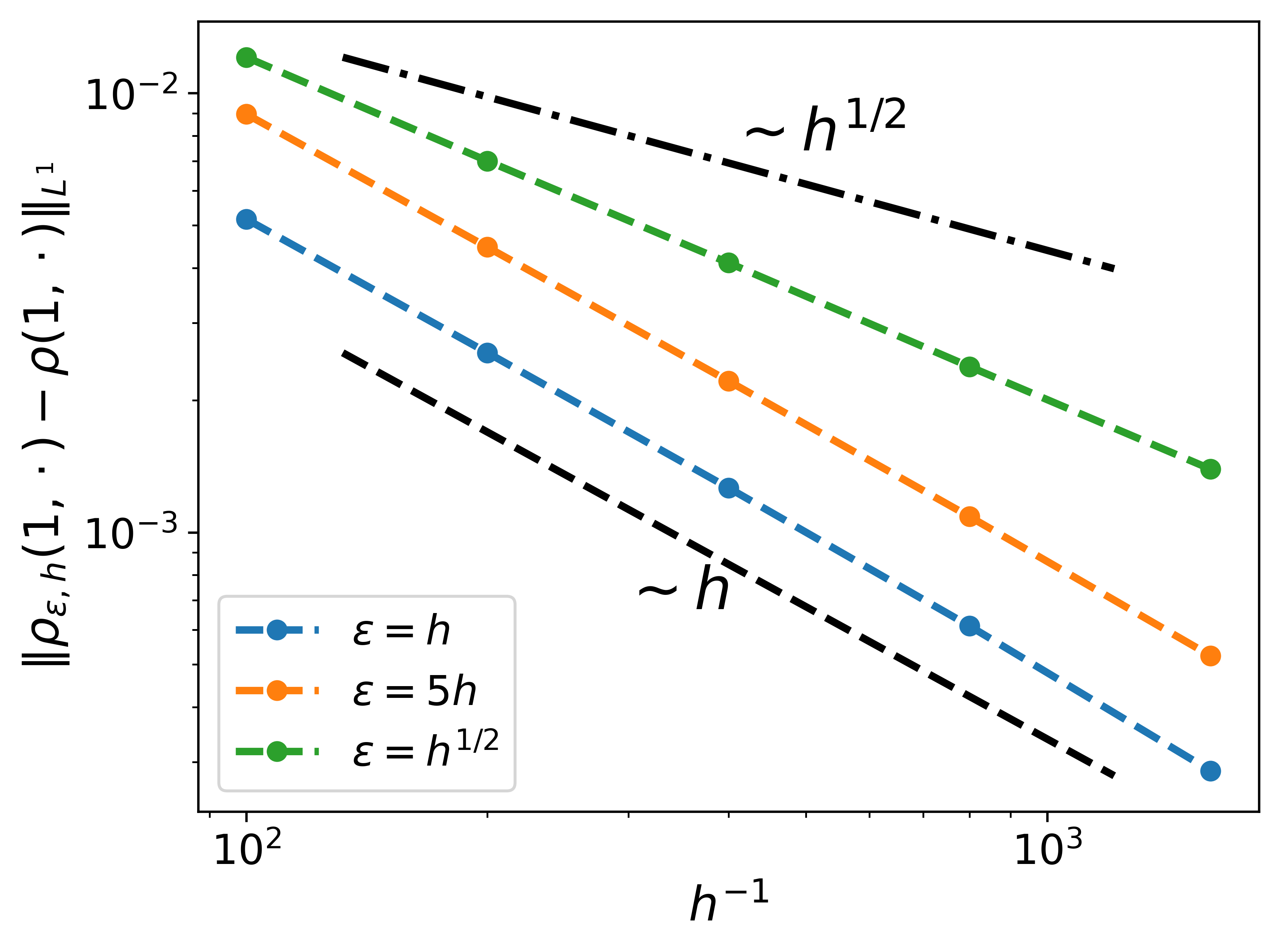}
	\end{subfigure}
	\begin{subfigure}{.32\textwidth}
	\includegraphics[width=\textwidth]{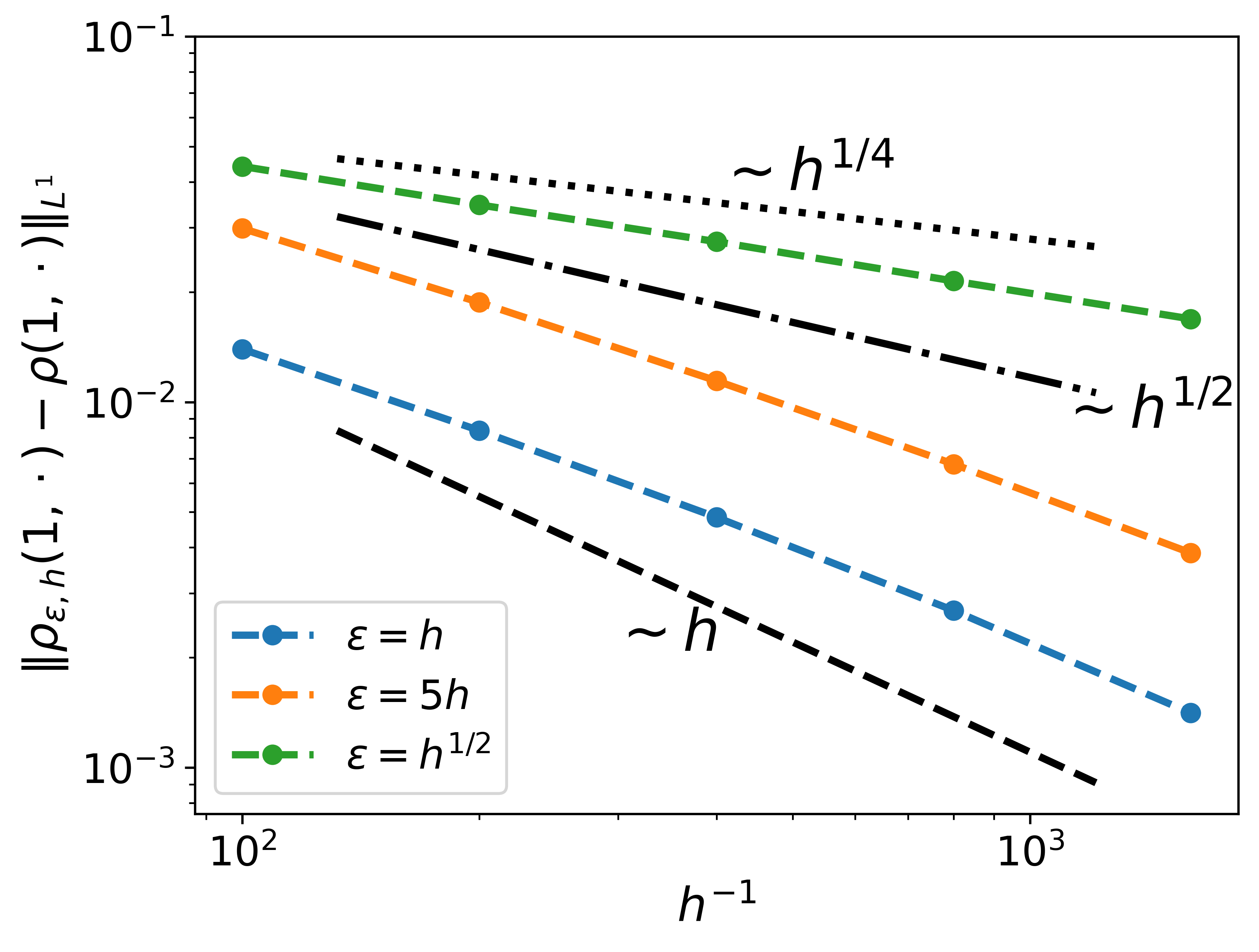}
	\end{subfigure}
    \begin{subfigure}{.32\textwidth}
	\includegraphics[width=\textwidth]{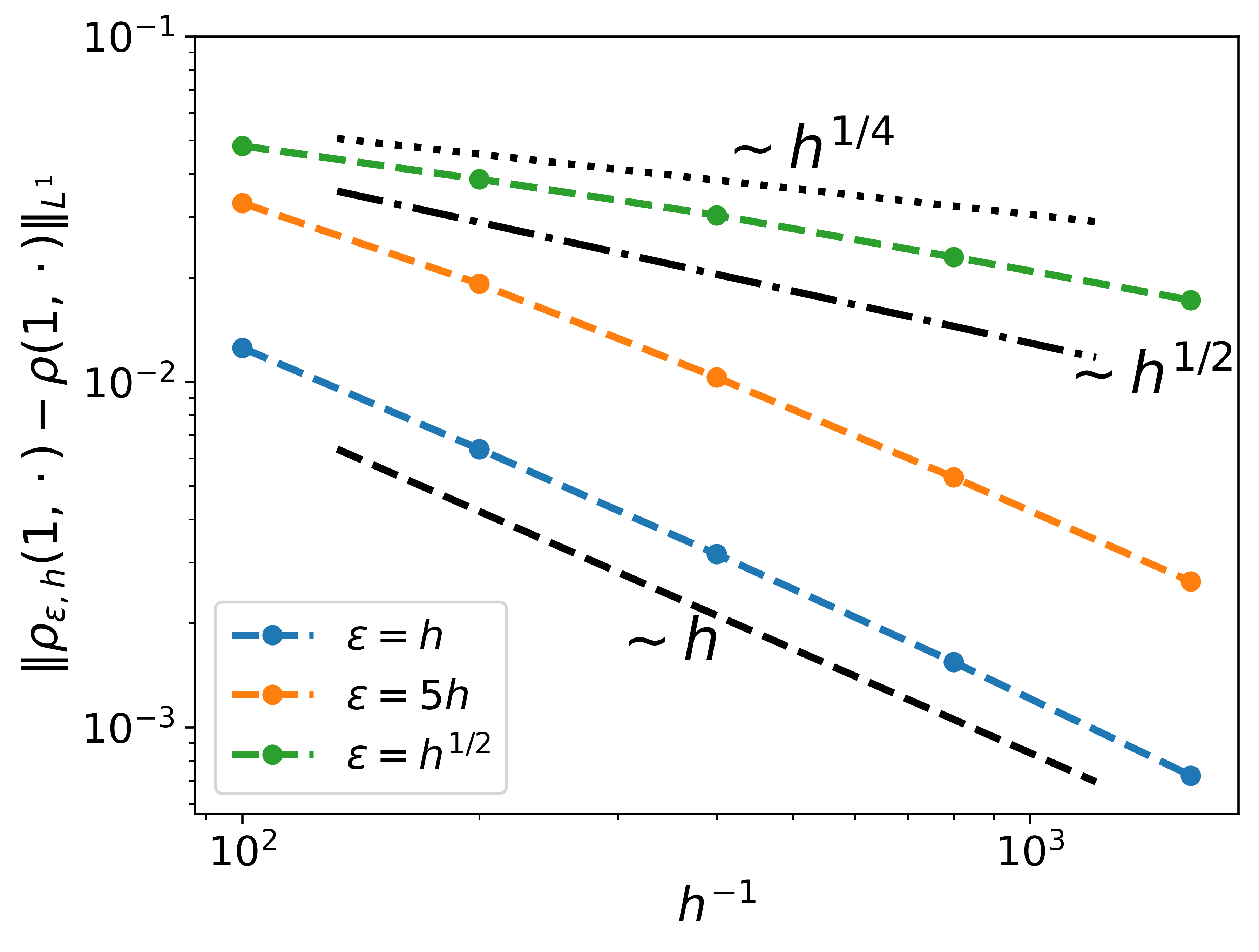}
	\end{subfigure}
	\caption{Convergence from $\rho_{\eps,h}$ to $\rho$ for the Riemann shock initial data \cref{eq:ini_riemann_shock} (\textsc{left}), the Riemann rarefaction initial data  \cref{eq:ini_riemann_rarefaction} (\textsc{middle}), and the bell-shaped initial data \cref{eq:ini_bellshape} (\textsc{right}) with the exponential kernel \cref{eq:gamma-expo} (\textsc{top}), the linear kernel \cref{eq:gamma-lin} (\textsc{middle}), and the constant kernel \cref{eq:gamma-const} (\textsc{bottom}).}
    \label{fig:exp_4}
\end{figure}

\subsection{Stability properties}
\label{ssec:num-stability}

In the following experiments, we explore the stability of numerical solutions, examining their TVD property and entropy condition with respect to Kru\v{z}kov's entropies, while also shedding light on the stability of their continuous counterparts with specific initial data.

We begin by addressing the TVD property. It is established in \cite{MR4300935} that the piecewise constant initial data
\begin{align}\label{eq:tvinc_ini}
    \rho_0^\delta(x) \coloneqq 0.5\,\mathds{1}_{]-\delta,-\delta/2[}(x) + \mathds{1}_{[0,\infty[}(x),
\end{align}
where $\delta\in]0,\eps]$, induces total variation increase in the solution $\rho_\eps$ of  \cref{eq:cle}. In the following experiment, we investigate numerical solutions using this initial data, providing visual insight into the total variation increase mechanism of $\rho_{\varepsilon,h}$ while confirming the TVD property of $W_{\varepsilon,h}$ across specific $\varepsilon,\delta$ values with a fixed mesh size $h$.
\begin{example}
We set the mesh size $h=2\times10^{-3}$. Initially, we use $\eps=0.2$, $\delta=0.2$ in the initial data \cref{eq:tvinc_ini}, and nonlocal kernels \crefrange{eq:gamma-expo}{eq:gamma-const}. In the top row of \cref{fig:exp_S1}, we present snapshots of $\rho_{\eps,h}$ at times $t=0,0.1,0.4,1.6$. Subsequently, we apply a sequence of $\eps$ values, $\eps=0.2,0.1,0.05$, with $\delta=\eps$ in the initial data \cref{eq:tvinc_ini} and the same nonlocal kernels. In the bottom row of \cref{fig:exp_S1}, we plot the total variation of $\rho_{\eps,h}$ and $W_{\eps,h}$ versus time $t\in[0,1.6]$, using solid and dashed lines, respectively.
\end{example}

The top row of \cref{fig:exp_S1} shows that the evolution of $\rho_{\varepsilon,h}$ is consistent across all nonlocal kernels: the initial shock at $x=0$ remains stationary, and the initial rectangular profile for $x\in]-\delta,-\delta/2[$ develops into a right-moving wave with a growing peak, which then merges with the stationary shock, resulting in a standing shock on the negative axis. The bottom row indicates that, for all cases, the total variation of $\rho_{\varepsilon,h}$ increases initially for $t<t_\varepsilon$ before decreasing to $1$ for $t>t_\varepsilon$, where $t_\varepsilon$, the time of wave-shock connection at $x=0$, decreases as $\varepsilon=\delta$ diminishes. In contrast, for all cases, the total variation of $W_{\varepsilon,h}$ is non-increasing over $t\in[0,1.6]$, remaining constant at $1$ for the exponential and constant kernels, and decreasing rapidly to this value for the linear kernel. These findings support and extend the analytical insights in \cite{MR4300935} on the total variation increase of $\rho_\eps$ with initial data \cref{eq:tvinc_ini}. They also suggest conjectures on the TVD property of $W_{\varepsilon,h}$ with a broader range of kernels, including the non-convex constant kernel, and the total variation bounded (TVB) property, which is weaker than TVD, may hold for $\rho_{\varepsilon,h}$.

\begin{figure}[htbp]
\centering
	\begin{subfigure}{.32\textwidth}
	\includegraphics[width=\textwidth]{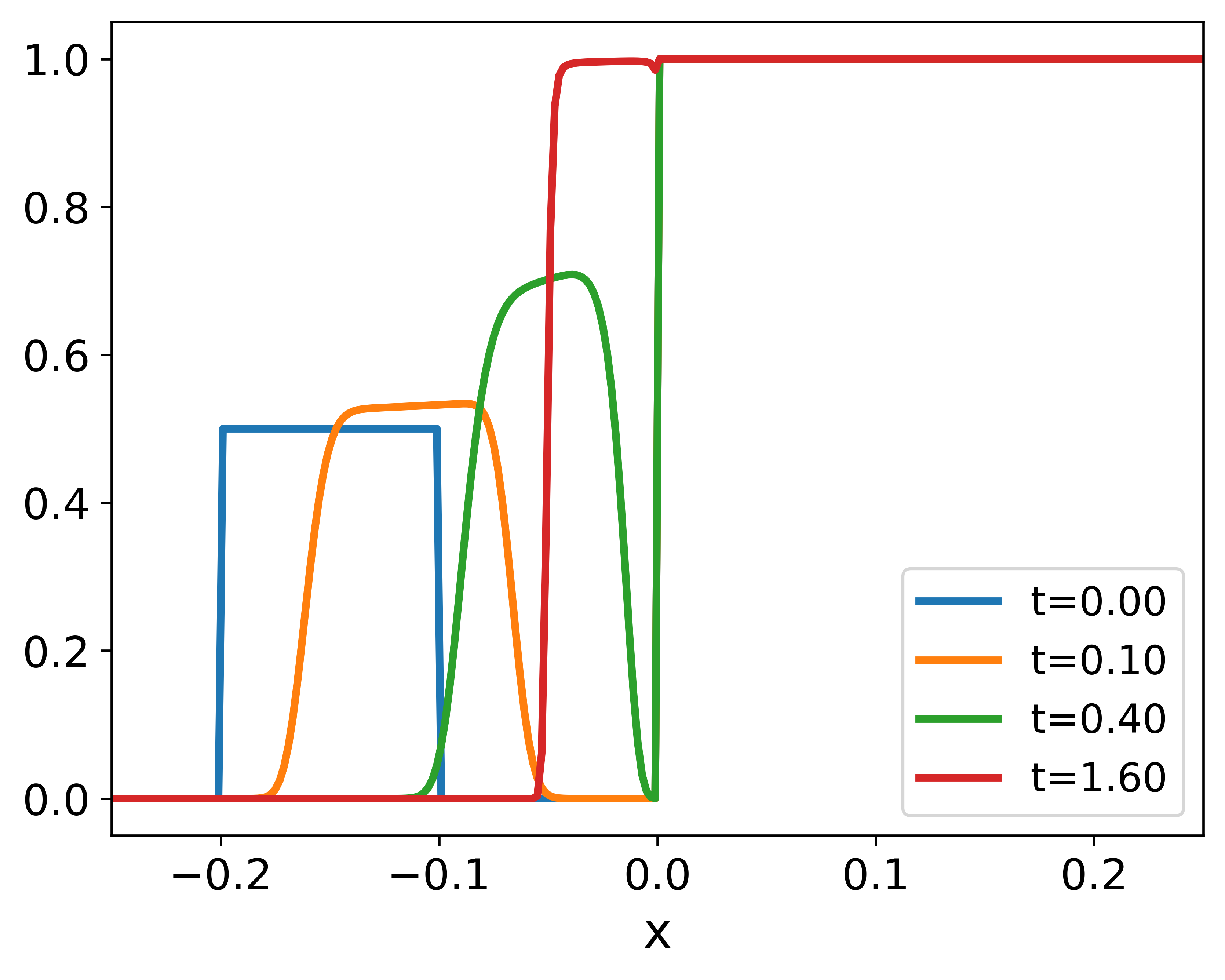}
	\end{subfigure}
	\begin{subfigure}{.32\textwidth}
	\includegraphics[width=\textwidth]{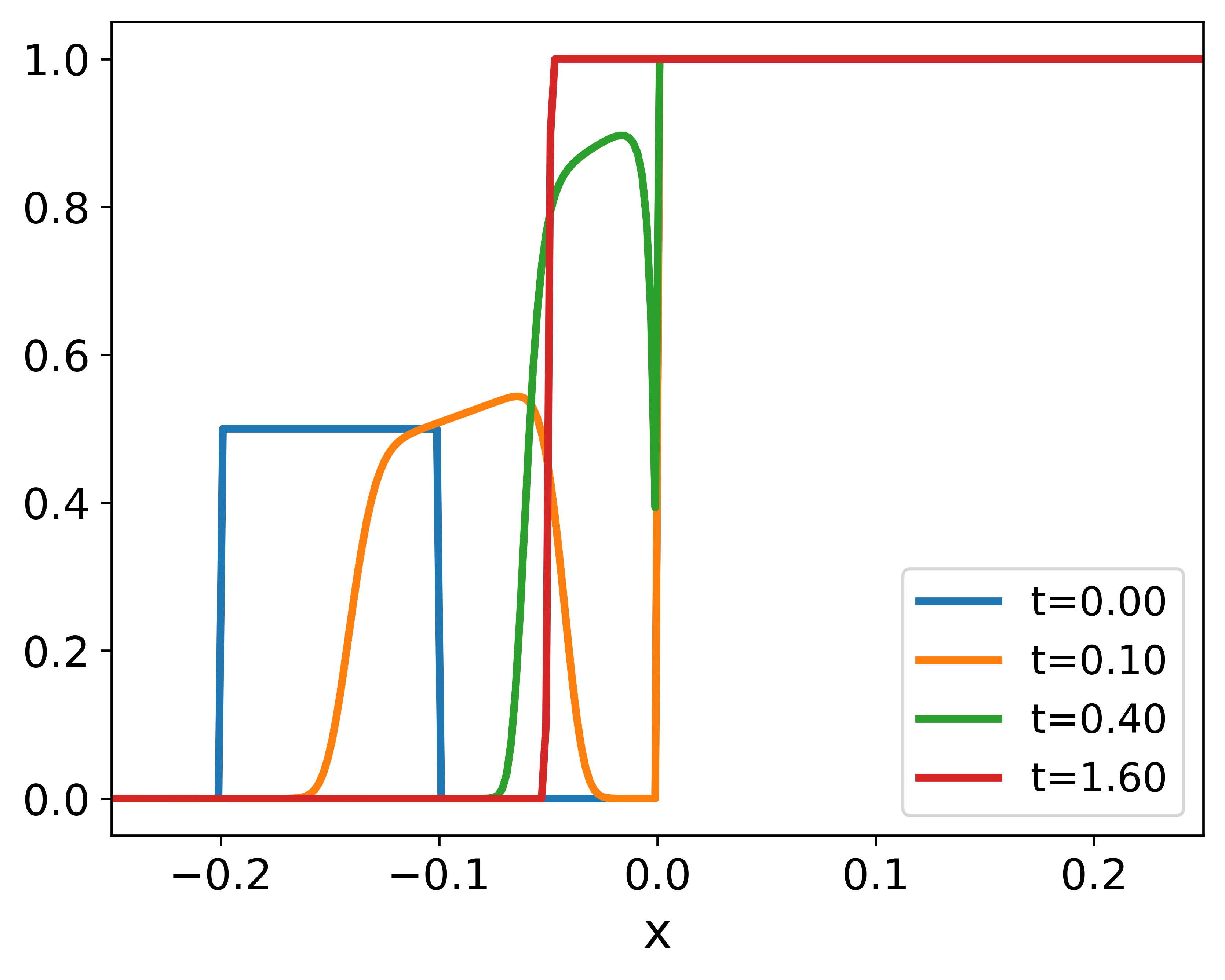}
	\end{subfigure}
	\begin{subfigure}{.32\textwidth}
	\includegraphics[width=\textwidth]{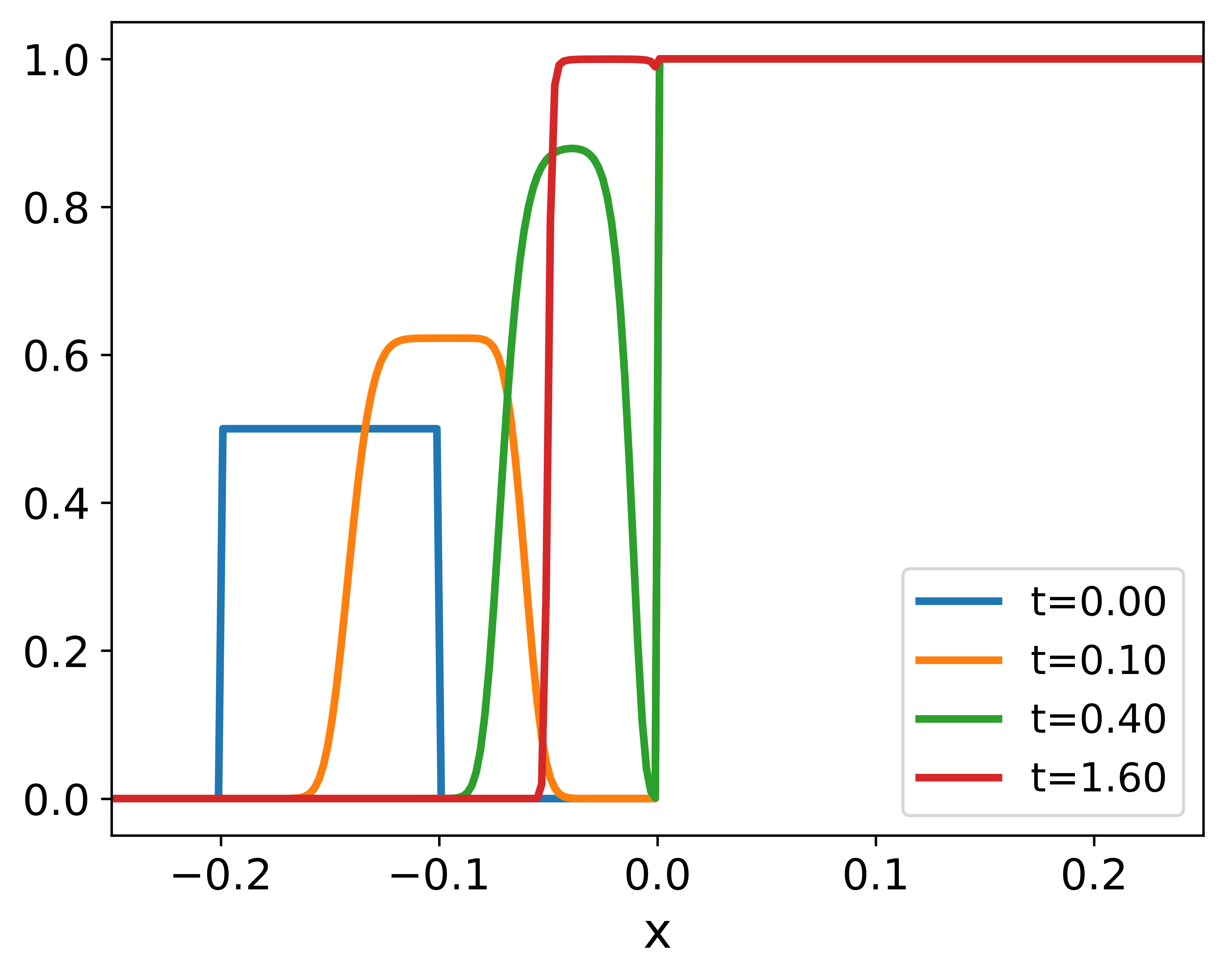}
	\end{subfigure}
	\begin{subfigure}{.32\textwidth}
	\includegraphics[width=\textwidth]{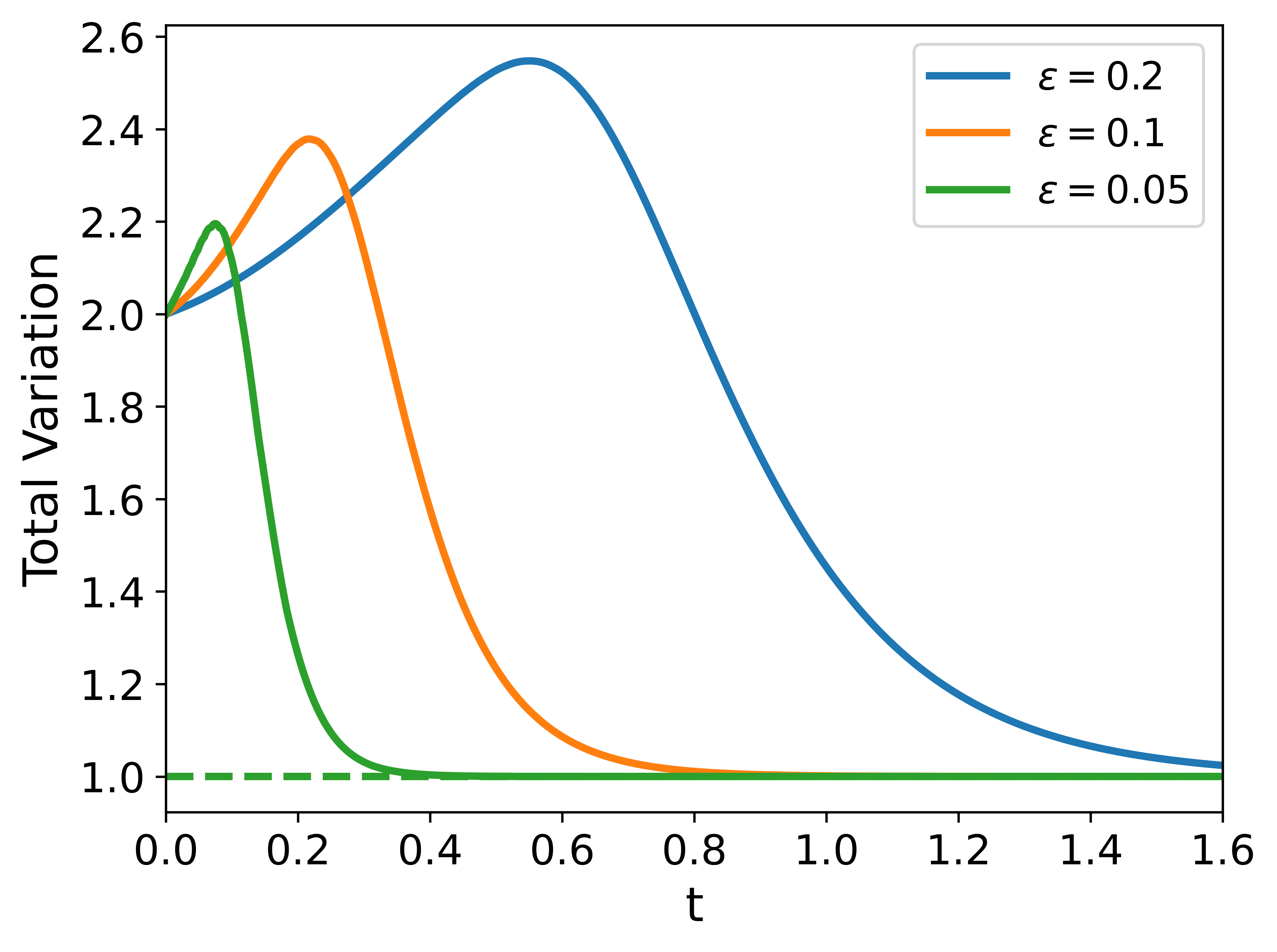}
	\end{subfigure}
	\begin{subfigure}{.32\textwidth}
	\includegraphics[width=\textwidth]{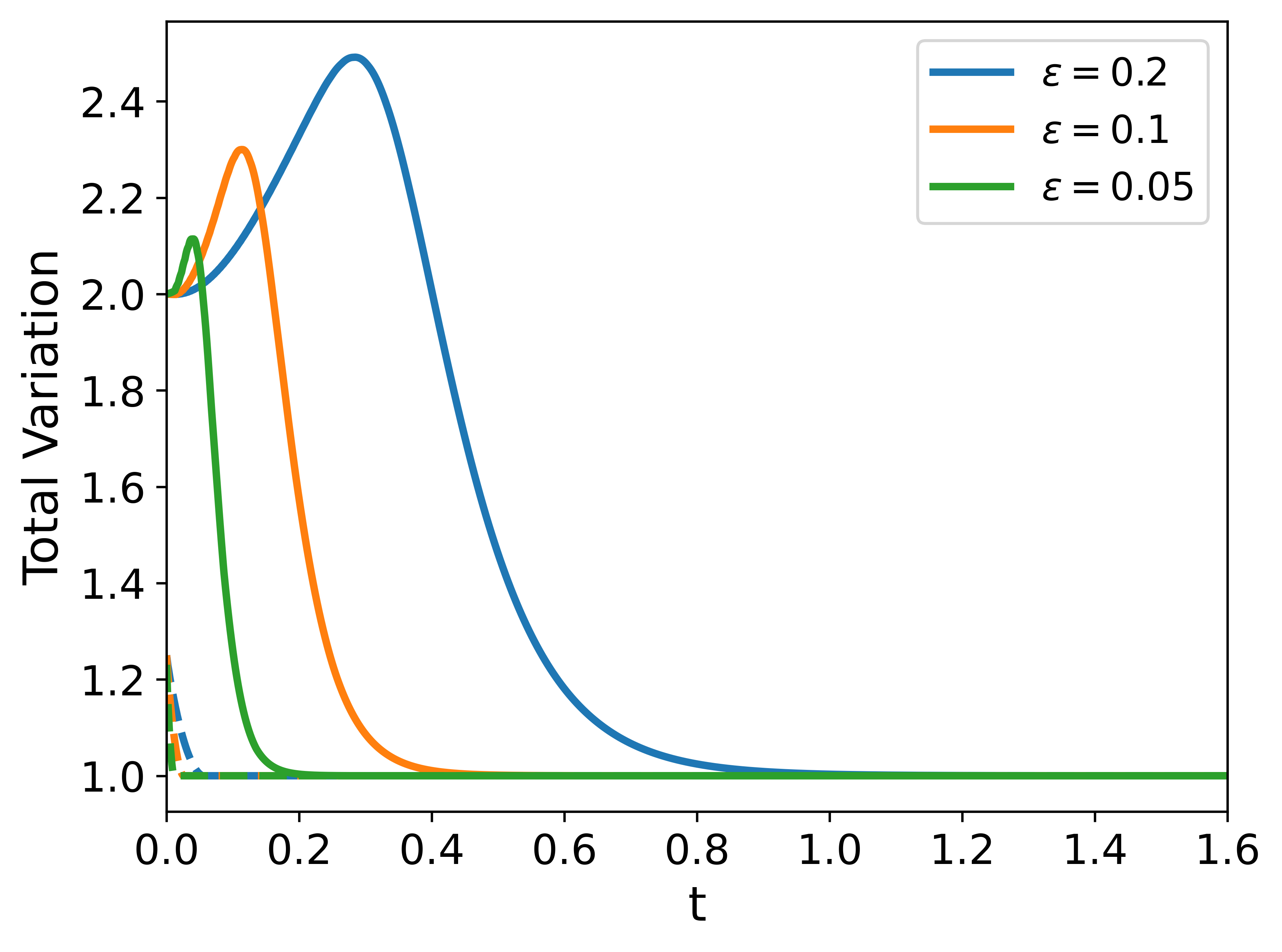}
	\end{subfigure}
	\begin{subfigure}{.32\textwidth}
	\includegraphics[width=\textwidth]{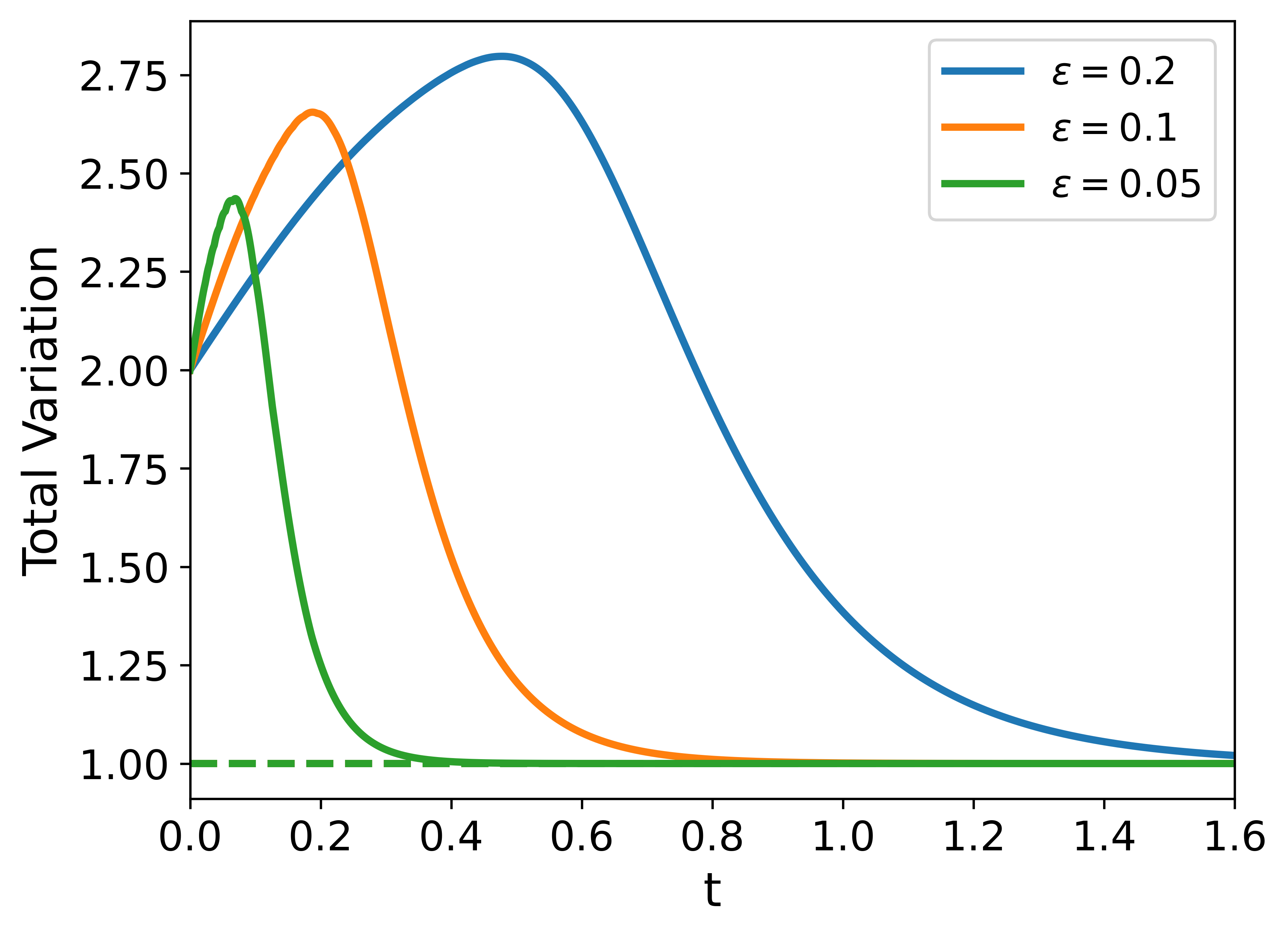}
	\end{subfigure}
	\caption{\textsc{Top:} Snapshots of $\rho_{\varepsilon,h}$ for $\varepsilon=0.2$ at selected times; \textsc{Bottom:} Total variation of $\rho_{\varepsilon,h}$ (solid lines) and $W_{\varepsilon,h}$ (dashed lines) versus time $t$ for $\varepsilon=0.2,0.1,0.05$. Initial data are \cref{eq:tvinc_ini} with $\delta=\varepsilon$, the nonlocal kernels \crefrange{eq:gamma-expo}{eq:gamma-const}, and the mesh size $h=2\times10^{-3}$.} 
	\label{fig:exp_S1}
\end{figure}

Next, we examine the entropy condition. As noted in \cref{rk:hfixed}, in the local limit where $W_j^n=\rho_j^n$, the scheme \crefrange{eq:godunov}{eq:num_nonlocal_W} reduces to the three-point monotone scheme \cref{eq:local_scheme} for solving \cref{eq:cl}. Therefore, the numerical solution satisfies a discrete entropy condition derived from \cref{eq:local_scheme}:
\begin{align*}
E_{j,n}^\rho \coloneqq \frac{|\rho_j^{n+1} - c| - |\rho_j^n - c|}{\tau} + \frac{ \Psi_c(\rho_j^n, \rho_{j+1}^n) - \Psi_c(\rho_{j-1}^n, \rho_j^n) }{h} \leq 0 \quad \text{for all } j, n,
\end{align*}
where $\Psi_c(\rho_{j-1}^n, \rho_j^n) \coloneqq (\rho_{j-1}^n \vee c) V(\rho_j^n \vee c) - (\rho_{j-1}^n \wedge c) V(\rho_j^n \wedge c)$, for any $c \in \mathbb{R}$. Equivalently,
\begin{align*}
E_{j,n}^W \coloneqq \frac{|W_j^{n+1} - c| - |W_j^n - c|}{\tau} + \frac{ \Psi_c(W_j^n, W_{j+1}^n) -\Psi_c(W_{j-1}^n, W_j^n) }{h} \leq 0 \quad \text{for all } j, n.
\end{align*}
Here, $E_{j,n}^\rho$ and $E_{j,n}^W$ measure the local entropy condition violation, staying non-positive in the local case. In the nonlocal case, \cref{lm:discrete-entropy-2} and \cref{lm:entro-disc-2} derive nonlocal entropy conditions for $W_{\eps,h}$ with general convex kernels and for $\rho_{\eps,h}$ with the exponential kernel \cref{eq:gamma-expo}, both aligning with the local case as $\eps\searrow0$, where the positive parts of $E_{j,n}^\rho$ and $E_{j,n}^W$ vanish for all $j,n$. This motivates the evaluation metrics
\begin{align}\label{eq:entropy_error}
\mathcal{E}^\rho \coloneqq \tau h \sum_{j,n} E_{j,n}^\rho, \qquad \mathcal{E}^W \coloneqq \tau h \sum_{j,n} E_{j,n}^W
\end{align}
to quantify the local entropy condition violation for $\rho_{\eps,h}$ and $W_{\eps,h}$, respectively. 

\begin{example}
We fix the mesh size $h=2\times10^{-3}$ and select $\varepsilon=2\times10^{-l}$ for $l=1,2,3$. For each $\varepsilon$, we compute numerical solutions $\rho_{\varepsilon,h}$ and $W_{\varepsilon,h}$ over the time horizon $t\in[0,1]$ using the initial data \crefrange{eq:ini_riemann_shock}{eq:ini_bellshape} and nonlocal kernels \crefrange{eq:gamma-expo}{eq:gamma-const}. We then evaluate the metrics $\mathcal{E}^\rho$ and $\mathcal{E}^W$ defined in \cref{eq:entropy_error} with $c=0.5$ in Kru\v{z}kov's entropy. The results are presented in \cref{tab:entropy}.
\end{example}

The results in \cref{tab:entropy} show that the local entropy condition violation for $\rho_{\eps,h}$ and $W_{\eps,h}$, quantified by \(\mathcal{E}^\rho\) and \(\mathcal{E}^W\) respectively, decreases as $\varepsilon$ approaches zero. These findings confirm the asymptotic compatibility of the nonlocal entropy conditions derived in \cref{lm:discrete-entropy-2} and \cref{lm:entro-disc-2}, suggesting that such conditions may extend to both $W_{\eps,h}$ and $\rho_{\eps,h}$ for a class of kernels including all convex kernels and the constant kernel \cref{eq:gamma-const}.
Furthermore, the local entropy condition is fully satisfied for both $\rho_{\eps,h}$ and $W_{\eps,h}$ (with \(\mathcal{E}^\rho=\mathcal{E}^W=0\)) when using the linear and constant kernels \crefrange{eq:gamma-lin}{eq:gamma-const} with $\eps=h=2\times10^{-3}$, confirming the local limit behavior in \cref{rk:hfixed}.

\begin{table}[htbp]
    \centering
    \setlength{\tabcolsep}{4pt}
    \renewcommand{\arraystretch}{1.3}
    \begin{tabular}{ccccccccccc}
        \toprule
        \textbf{} & \textbf{} & \multicolumn{3}{c}{\textbf{Initial data \cref{eq:ini_riemann_shock}}} & \multicolumn{3}{c}{\textbf{Initial data \cref{eq:ini_riemann_rarefaction}}} & \multicolumn{3}{c}{\textbf{Initial data \cref{eq:ini_bellshape}}} \\ \cmidrule{3-11}
        \textbf{\(\varepsilon\)} & \textbf{Metric} & \textbf{Exp.} & \textbf{Linear} & \textbf{Const.} & \textbf{Exp.} & \textbf{Linear} & \textbf{Const.} & \textbf{Exp.} & \textbf{Linear} & \textbf{Const.} \\ \midrule
        \multirow{2}{*}{2e-1} & \(\mathcal{E}^\rho\) & 8.3e-3 & 5.5e-3 & 8.2e-3 & 9.4e-3 & 5.8e-3 & 5.5e-2 & 4.6e-2 & 1.1e-2 & 2.5e-2 \\
        & \(\mathcal{E}^W\) & 2.2e-2 & 2.0e-2 & 2.1e-2 & 1.0e-3 & 8.5e-4 & 7.4e-3 & 2.3e-2 & 7.5e-3 & 1.5e-2 \\ \midrule
        \multirow{2}{*}{2e-2} & \(\mathcal{E}^\rho\) & 1.2e-4 & 0 & 0 & 6.2e-4 & 1.9e-4 & 1.1e-3 & 4.5e-3 & 2.8e-3 & 3.5e-3 \\
        & \(\mathcal{E}^W\) & 1.7e-2 & 6.5e-3 & 8.0e-3 & 1.6e-4 & 1.1e-4 & 3.0e-4 & 4.1e-3 & 2.8e-3 & 3.5e-3 \\ \midrule
        \multirow{2}{*}{2e-3} & \(\mathcal{E}^\rho\) & 0 & 0 & 0 & 3.3e-5 & 0 & 0 & 8.0e-4 & 0 & 0 \\
        & \(\mathcal{E}^W\) & 5.0e-4 & 0 & 0 & 3.9e-5 & 0 & 0 & 8.4e-4 & 0 & 0 \\ \bottomrule
    \end{tabular}
    \caption{Local entropy condition violation for initial data \crefrange{eq:ini_riemann_shock}{eq:ini_bellshape} with nonlocal kernels \crefrange{eq:gamma-expo}{eq:gamma-const}, evaluated by metrics \(\mathcal{E}^\rho\) and \(\mathcal{E}^W\) defined in \cref{eq:entropy_error} with $c=0.5$ in Kru\v{z}kov's entropy.}
    \label{tab:entropy}
\end{table}

\section{Conclusions and future directions}
\label{sec:conclusion}

In this work, we studied a Godunov-type numerical discretization for a class of nonlocal conservation laws modeling traffic flows. We proved asymptotic compatibility of the discretization, i.\,e., as the nonlocal parameter $\eps$ and mesh size $h$ vanish, the discretization converges to the entropy solution of the respective (local) scalar conservation law, with an explicit convergence rate in terms of both $\eps$ and $h$.
These results justify that the numerical discretization can provide robust numerical computation for the model under variations of the nonlocal parameter, which is of both theoretical and practical significance.

The results of this study open several avenues for future research. First, extending the results to nonlocal kernels with weaker properties, such as constant kernels, and to initial data with unbounded variation, is suggested by our numerical experiments. Second, applying the asymptotic compatibility framework to other first-order finite-volume methods or higher-order DG and WENO methods, could improve accuracy and broaden applicability. Third, investigating the $\mathrm L^1$-contraction property at both continuous and discrete levels for a broader class of kernels, like all convex kernels, presents a promising direction. Finally, we are interested in extending these results to more general nonlocal conservation laws in diverse applications, such as those with nonlocal fluxes of the form $f(\rho_\eps)V(W_\eps[\rho_\eps])$ for nonlinear $f$, or to nonlocal systems of conservation laws.

\vspace{0.5cm}

\section*{Acknowledgments}

N.~De Nitti is a member of the Gruppo Nazionale per l’Analisi Matematica, la Probabilità e le loro Applicazioni (GNAMPA) of the Istituto Nazionale di Alta Matematica (INdAM). He has been partially supported by the Swiss State Secretariat for Education, Research and Innovation (SERI) under contract number MB22.00034 through the project TENSE.

N.~De Nitti thanks The Chinese University of Hong Kong, where part of this work was carried out, for their kind hospitality. 

K.~Huang was supported by a Direct Grant of Research (2024/25) from The Chinese University of Hong Kong.

We thank A.~Coclite, G.~M.~Coclite, and M.~Colombo for their feedback on this manuscript.

\vspace{0.5cm}

\printbibliography

\vfill 

\end{document}